\documentclass{amsart}

\usepackage{slashed}
\usepackage{graphicx}
\usepackage{amsfonts}
\usepackage{amssymb}
\usepackage{amsmath}
\usepackage{amsxtra}
\usepackage{latexsym}

\newtheorem{theorem}{Theorem}[section]
\newtheorem{lemma}{Lemma}[section]
\newtheorem{Assumption}{Assumption}[section]
\newtheorem{corollary}{Corollary}
\newtheorem{proposition}{Proposition}[section]

\theoremstyle{definition}
\newtheorem{definition}[theorem]{Definition}

\theoremstyle{remark}
\newtheorem{remark}[theorem]{Remark}

\numberwithin{equation}{section}

\def\a{\alpha}

\def\Ab{\underline{A}}
\def\b{\beta}
\def\bb{{\underline{\b}}}
\def\B{\mathcal{B}}
\def\bd{{\mathbf{D}}}
\def\bg{{\mathbf{g}}}
\def\bM{{\mathbf{M}}}
\def\bR{\mathbf{R}}
\def\bT{\mathbf{T}}
\def\c{\cdot}
\def\chib{\underline{\chi}}
\def\chibh{\hat{\chib}}
\def\chih{\hat{\chi}}
\def\ckk{\check}
\def\curl{\mbox{\,curl\,}}
\def\D{\mathcal{D}}
\def\div{\mbox{div}}
\def\Dt{\mathcal{D}_t}
\def\dt{{\frac{d}{dt}}}
\def\E{{\mathcal E}}
\def\e{\eta}
\def\eh{\hat{\eta}}
\def\ep{\epsilon}
\def\er{\mbox{err}}
\def\f12{\frac{1}{2}}
\def\F{{\mathcal{F}}}
\def\ff{{\mathfrak{F}}}
\def\G{\mathcal{G}}
\def\Ga{\Gamma}
\def\ga{\gamma}
\def\gac{\stackrel{\circ}\ga}
\def\gc{{\Delta_0^2+\RR}}
\def\gag{{GaNi}}
\def\ggg{\mathbf{g}}
\def\hk{\hat{k}}
\def\hot{\widehat{\otimes}}
\def\I{{\mathcal I}}
\def\k{\mathcal{K}_0}
\def\K{\underline{K}}
\def\kp{\kappa}
\def\l{\langle}
\def\r{\rangle}
\def\La{\Lambda}
\def\L{\mathcal L}
\def\les{\lesssim}
\def\M{\mathcal{M}}
\def\N{\mathcal{N}}
\def\Lb{{\underline{L}}}
\def\H{\mathcal{H}}
\def\nab{\nabla}
\def\nn{\nonumber}
\def\os#1{{\mathcal Osc}(#1)}
\def\Ot{\otimes}
\def\ovl{\overline}
\def\p{\partial}
\def\P{\mathcal{P}}
\def\poin{{Poin}}

\def\rhob{{\bar\rho}}
\def\rhoc{\check{\rho}}
\def\rhocb{{\bar{\ckk\rho}}}
\def\R{\mathcal{R}}
\def\RR{{\mathcal R}_0}
\def\S{\mathcal{S}}
\def\s{\sigma}
\def\sig{\sigma}
\def\sigc{\check{\sigma}}
\def\sigcb{{\bar {\ckk\sigma}}}
\def\sD{\slashed{\Delta}}
\def\sl{\slashed}
\def\smi{{SobM1}}
\def\smii{{SobM2}}
\def\sn{\slashed{\nabla}}
\def\sob{{Sob}}
\def\t1a{r^{-\frac{1}{b}}}
\def\ti{\tilde}
\def\Tr{\mbox{Tr}}
\def\tr{\mbox{tr}}
\def\tR{{\tilde R}}
\def\tt{{{t'}}}
\def\U{{\mathcal U}}
\def\udb{\underline{\beta}}
\def\und{\underline}
\def\V{{\mathcal V}}
\def\W{{\mathcal W}}
\def\with{\mbox{ with }}
\def\wt{\widetilde}
\def\zb{\underline{\zeta}}

\begin{document}

\title[]
{On Ricci coefficients of null hypersurfaces with time
foliation in Einstein vacuum space-time}

\author{Qian Wang}
\address{Department of Mathematics, Stony Brook University, Stony Brook, NY 11794}
\email{qwang@math.sunysb.edu}
\curraddr{} \email{}

\date{}





\begin{abstract}
The main objective of this paper is to control the geometry of
null cones with time foliation in Einstein  vacuum spacetime under
the assumptions of small curvature flux and a weaker condition on
the deformation tensor for $\bT$.
We establish a series of estimates on Ricci coefficients, which
plays a crucial role to prove the improved breakdown criterion in \cite{Qwang1}.
\end{abstract}

\maketitle

\newtheorem{assumption}{Assumption}

\section{\bf Introduction}
\setcounter{equation}{0}

Consider a (3+1)-dimensional Einstein vacuum spacetime $(\bM, \bg)$
foliated by $\Sigma_t$ which are level hypersurfaces of a time
function $t$ monotonically increasing towards the future. Let $\bd$
and $\nab$ denote the covariant differentiations with respect to
$\bg$ and the induced metric $g$ on $\Sigma_t$ respectively.
We define on each $\Sigma_t$ the lapse function $n$ and the second
fundamental form $k$  by
$$
n:=\left(-\bg(\bd t, \bd t)\right)^{1/2} \quad \mbox{and}\quad k(X,
Y):=-\bg (\bd_X \bT, Y),
$$
where $\bT$ denotes the future directed unit normal to $\Sigma_t$
and $X, Y\in T\Sigma_t$. For any coordinate chart ${\mathcal
O}\subset \Sigma_{t_0}$ with coordinates $x=(x^1, x^2, x^3)$, let
$x^0=t, x^1, x^2, x^3$ be the transported coordinates  obtained by
following the integral curves of $\bT$. Under these coordinates the
metric $\bg$ takes the form
\begin{equation}\label{bg}
\bg=-n^2 d t^2+g_{ij}d x^i d x^j, \quad \p_t g_{ij}=-2n k_{ij}.
\end{equation}

\subsection{Main result}

Consider an outgoing null cone contained in $(\bM, \bg)$, whose
vertex is denoted by $p$ and intersections with $\Sigma_t$ are
denoted by $S_t$. The null vector $l_{\omega}$ in $T_p\bM$
parametrized  with $ \omega\in {\Bbb S}^2$, is normalized by
$\bg(l_\omega, \bT_p)=-1$. We denote by $\Gamma_\omega(s)$ the
outgoing null geodesic from $p$ with
$$
\Gamma_\omega(0)=p, \quad \frac{d}{ds}\Gamma_{\omega}(0)=l_\omega
$$
and define the null vector field $L$ by
$$
L(\Ga_\omega(s))=\frac{d}{ds}\Ga_\omega(s).
$$
Then $\bd_L L=0$. The affine parameter $s$ of null geodesic is
chosen such that $s(p)=0$ and $L(s)=1.$ Let $\H=\cup_{0<t\le 1}
S_t$, $t(p)=0$\begin{footnote}{we can always suppose   $t(p)=0$ and
$0\le t\le 1$ on a null cone by a standard rescaling of the
coordinates $(t,x)$ of $(\bM, \bg)$. The normalized time function,
for simplicity, is still denoted by $t$.}
\end{footnote}and suppose the exponential
map $\G_t: \omega\rightarrow \Ga_\omega(s(t))$ is a global
diffeomorphism  from ${\Bbb S}^2$ to $S_t$ for any $t\in(0, 1]$. We
now define a conjugate null vector $\underline{L}$ on $\H$ with
$\bg(L, \underline{L})=-2$ and such that $\underline{L}$ is
orthogonal to the leafs $S_t$. In addition we can choose an
orthonormal frame $(e_A)_{A=1,2}$ tangent to $S_t$ such that
$(e_A)_{A=1,2}$, $e_3=\Lb$, $e_4=L$ form a null frame, i.e.
$$
\bg(L,\Lb)=-2, \,\, \bg(L,L)=\bg(\Lb,\Lb)=\bg(L, e_A)=\bg(\Lb,
e_A)=0, \,\,\bg(e_A, e_B)=\delta_{AB}.
$$
Let $a^{-1}=-\l L, \bT\r$ with $a(p)=1$. It follows that along any
null geodesic $\Ga_\omega$, there holds
\begin{equation}\label{st}
\frac{dt}{ds}=n^{-1}a^{-1},\quad t(p)=0.
\end{equation}

Let  $N$ be the outward unit normal of $S_t$ on $\Sigma_t$. Then
\begin{equation}\label{TN}
\bT=\frac{1}{2}\left(a L +a^{-1}\Lb\right), \quad \quad
N=\frac{1}{2}\left(a L-a^{-1}\Lb\right).
\end{equation}
We define the Ricci coefficients $\chi, \chib, \zeta, \zb,\varpi$
via the frame equations
\begin{eqnarray*}
\bd_A
L=\chi_{AB}e_B-\zeta_A L, &&\quad \bd_A \Lb=\chib_{AB}e_B+\zeta_A \Lb,\\
\bd_L \Lb=2\zb_A e_A, &&\quad  \bd_\Lb L=2\zeta_A e_A -2 \varpi L.
\end{eqnarray*}
Thus we also have
\begin{align*}
\bd_L e_A&=\sn_L e_A +\zb_A L,\quad \bd_B e_A =\sn_B
e_A+\frac{1}{2}\chi_{AB} e_3+\frac{1}{2}\chib_{AB} e_4
\end{align*}
where $\sn$ denotes the covariant derivative restricted on $S_t$.

Let $\lambda=-\frac{1}{3}\Tr k$,  where $\Tr k=g^{ij}k_{ij}$. We
decompose $\hk_{ij}:=k_{ij}+\lambda g_{ij}$, the traceless part of
$k$, relative to the orthonormal frame $\{N, e_A, A=1,2,\}$ along
null cone $\H$ by introducing the following components
\begin{equation}\label{compo}
\e_{AB}=\hk_{AB}\quad \quad \ep_A=\hk_{AN}\quad\quad
\delta=\hk_{NN}.
\end{equation}
Denote by $\eh_{AB}$ the traceless part of $\e$. Since $\delta^{AB}
\e_{AB}=-\delta$, it is easy to see
\begin{equation*}
\eh_{AB}=\e_{AB}+\frac{1}{2}\delta_{AB}\delta.
\end{equation*}
We denote by $\sl{\pi}$  one of the following $S_t$ tangent tensors
$\{\eh, \delta, \ep, \lambda, -\sn \log n, -\nab_N \log n\}$. It is
easy to check by definition that the Ricci coefficients $\zeta,
\zb,\nu$ verify
\begin{align}
&\nu:=-L(a)=-\sn_N\log n+\delta-\lambda \label{c5},\\
&\zeta_A=\sn_A \log a+\ep_A, \quad \zb_A=\sn_A \log
n-\ep_A\label{c4}.
\end{align}
Let us define $\theta_{AB}:=\l \bd_A N, e_B\r$. By definition of
$\chi, \chib$ and (\ref{TN}), it follows that
\begin{align}
&a\chi_{AB}=\theta_{AB}-k_{AB},\quad
a^{-1}\chib_{AB}=-\theta_{AB}-k_{AB}, \label{c2}\\
&a\tr\chi=\tr\theta+\delta+2\lambda,\quad a^{-1}
\tr\chib=-\tr\theta+\delta+2\lambda. \label{m11}
\end{align}

We define the null components of Riemannian curvature tensor
relative to $t$-foliation,
\begin{eqnarray}
\alpha_{AB}=\bR(L, e_A, L, e_B),&& \quad \beta_A=\frac{1}{2}\bR(e_A,
L, \underline{L}, L),\nonumber\\
\rho=\frac{1}{4}\bR(\underline{L},L, \underline{L},L),&&\quad
\sigma=\frac{1}{4}{{}^\star \bR}(\underline{L}, L, \underline{L}, L),\label{f14}\\
\underline{\beta}_A=\frac{1}{2}\bR(e_A, \underline{L},
\underline{L},L),&&\quad  \underline{\a}_{A B}=\bR(\underline{L},
e_A, \underline{L}, e_B).\nonumber
\end{eqnarray}
We define also the mass aspect functions  $\mu$ and
$\underline{\mu}$ as follows
\begin{align}
\mu&=-\frac{1}{2} \bd_3 \tr\chi+\frac{a^2}{4}(\tr\chi)^2-\varpi
\tr\chi, \label{m1}\\
\underline{\mu}&=\bd_4 \tr\chib+\frac{1}{2} \tr\chi\c \tr\chib.
\label{m2}
\end{align}

Denote $\ga_{t}:=\ga(t,\omega)$ the induced metric of $\bg$ on
$S_{t}$, relative to normal coordinates $\omega=(\omega_1,\omega_2)$
in the tangent space at $p$. Define the radius function of $S_t$ to
be $r(t)=\sqrt{(4\pi)^{-1} |S_{t}|}$ and define the metric $\gac$ by
$\gac=r^{-2}\ga$. We denote by $\ga^{(0)}$ the canonical metric on
${\Bbb S}^2$.
 On each
$S_t$ we introduce the ratio of area elements
\begin{equation}\label{vt1}
v_t(\omega):=\frac{\sqrt{|\ga_t|}}{\sqrt{|\ga^{(0)}|}}, \qquad
\omega\in {\Bbb S}^2.
\end{equation}
For smooth scalar functions $f$, the average of $f$ on $S_t$ is
defined by $\bar f:=\frac{1}{|S_t|}\int_{S_t} f d\mu_\ga.$
 For any scalar functions $f$,
\begin{equation*}
\int_{\H} f:=\int_0^1 \int_{S_t} f na d\mu_\ga dt=\int_0^1
\int_{|\omega|=1}na f v_t d\mu_{{\Bbb S}^2} dt.
\end{equation*}
We  define $L_\omega^p$ norm for smooth functions $f$ on $S_t$ by $
\|f\|_{L_\omega^p}^p=\int_{S_t} |f|^p d\mu_{\Bbb S^2}$, and define
$L^2$ norm on null cone $\H$ for any smooth function $f$ by
\begin{equation*}
\|f\|_{L^2(\H)}^2=\int_0^1 \int_{S_t}|f|^2 na d\mu_\ga dt.
\end{equation*}
For simplicity, we will suppress the $\H$ in the definition of norms
on $\H$ whenever there occurs no confusion. Define for any $S_t$
tangent tensor $F$ the norm on $\H$
\begin{equation*}
\N_1(F)= \|\sn_L F\|_{L^2}+\|\sn F\|_{L^2}+\|r^{-1} F\|_{L^2}.
\end{equation*}
 Define $\R(\H)$,  curvature flux on $\H$ relative to
$t$-foliation, by
\begin{equation*}
\R(\H)^2=\int_0^1 \int_{S_t}
an(|\a|^2+|\b|^2+|\rho|^2+|\sigma|^2+|\udb|^2) d \mu_{\ga} dt.
\end{equation*}
For any $S_t$-tangent tensor field $F$ we define the norm
$\|F\|_{L_\omega^\infty L_t^2(\H)}$ by
\begin{equation*}
\|F\|_{L_\omega^\infty L_t^2}^2:=\sup_{\omega\in {\Bbb
S}^2} \int_0^1  |F|^2 na dt:=\sup_{\omega\in {\Bbb
S}^2} \int_{\Gamma_\omega} |F|^2 na dt.
\end{equation*}

The main result of this paper is  the following

\begin{theorem}[{\bf Main Theorem}]\label{T8.1}
Let $({\bold M}, \ggg)$ be a smooth 3+1 Einstein vacuum spacetime
foliated by $\Sigma_t$, the level hypersurfaces of a time function
$t$ with lapse function  $n$. Consider an outgoing null hypersurface
$\H=\cup_{0<t<1}S_t$ in $({\bold M}, \ggg)$ initiating from a point
p, whose leaves are $S_t=\Sigma_t\cap \H$ and $t(p)=0$. Assume $
  C^{-1}<n<
C$ on $\H$ with certain positive constant $C$  and assume that
\begin{equation}\label{cond1}
 \R(\H)+\N_1(\sl{\pi})\le \RR, \mbox{ on } \H
\end{equation}
 with $\RR$  sufficiently
small. Then the following estimates hold true
\begin{align}
\left\|\tr \chi -\frac{2}{s}\right\|_{L^\infty(\H)}&\lesssim
\RR^2,\label{mt0}\\
|a-1|&\le \frac{1}{2},\label{mt7}
\end{align}
\begin{align}
&\left\|\int_0^1 |\chih|^2 na dt \right\|_{L_\omega^\infty}+
\left\|\int_0^1 |\zeta|^2 na
dt \right\|_{L_\omega^\infty}\les \RR^2,\label{mt1}\\
&\left\|\int_0^1 |\zb|^2 na dt \right\|_{L_\omega^\infty}+
\left\|\int_0^1 |\nu|^2 na dt \right\|_{L_\omega^\infty}\les
{\mathcal
R}_0^2,\label{mt2}\\
&\|\sn\tr\chi\|_{\P^0}+\|\mu \|_{\P^0}+\|\sn
\tr\chi\|_{L^2(\H)}+\|\mu\|_{L^2(\H)} \les
{\mathcal R}_0,\label{mt3}\\
 &{\mathcal N}_1(\chih)+{\mathcal
N}_1(\zeta)+{\mathcal N}_1\left(\tr
\chi-\frac{2}{r}\right)+\N_1(\tr\chi-(an)^{-1}\ovl{an\tr\chi})\les {\mathcal R}_0,\label{mt4}\\
&\|\tr\chi-\frac{2}{r}\|_{L_t^2 L_\omega^\infty}+\|\tr\chi-(an)^{-1}\ovl{an\tr\chi}\|_{L_t^2 L_\omega^\infty}\les \RR,\label{mt5}\\
 &\left\|\sup_{t\le 1}|r^{3/2}\sn \tr \chi|\right\|_{L^2_\omega}+
\left\|\sup_{t\le 1}r^{\frac{3}{2}}|\mu|
\right\|_{L^2_\omega}+\left\|r^{1/2}\sn \tr \chi
\right\|_{\B^0}+\|r^{1/2}\mu\|_{\B^0}\les {\mathcal R}_0.\label{mt6}
\end{align}
\end{theorem}

The Besov norms $\P^0$ and $\B^0$ appearing in the above statement
will be defined by (\ref{g4.5}) and (\ref{g5}).
 Throughout this paper we
will use the notation $A\lesssim B$ to mean $A\le C\cdot B$ for some
appropriate  constant $C$.

It is worthy to remark that (\ref{mt0}) is the most important
estimate in the main theorem and none of the estimates among
 (\ref{mt0})-(\ref{mt5}) can be proved
independent of others. Therefore we have to  prove all of them
simultaneously with a delicate bootstrap argument.

\subsection{Application}

By a standard rescaling argument, see \cite[page 363]{KR2}, the
estimates (\ref{mt0})-(\ref{mt6}) in Theorem \ref{T8.1} can be
rephrased as \cite[Theorem 7, Proposition 14]{Qwang1}, which are the
crucial components in the proof of an improved breakdown criterion
for Einstein vacuum spacetime in CMC gauge, stated as follows
\begin{theorem}\cite[Theorem 1]{Qwang1}\label{thm3}
Let $(\M, \ggg)$ be a globally hyperbolic development of
$\Sigma_{t_0}$ foliated by the CMC level hypersurfaces of a time
function $t<0$. Then the space-time together with the foliation
$\Sigma_t$ can be extended beyond any value $t_*<0$ for which,
\begin{equation}\label{aben1}
\int_{t_0}^{t_*} \left(\|k\|_{L^\infty(\Sigma_t)}+\|\nab \log n\|_{
L^\infty(\Sigma_t)}\right) d t=\k<\infty.
\end{equation}
\end{theorem}
Under the assumption (\ref{aben1}) only, we have proved in
\cite[Section 3]{Qwang1} that  $C^{-1}<n<C$ with $C$ depending only
on $t^*,\,\k$ and the Bel-Robinson energy on the initial slice
$\Sigma_{t_0}$. The condition (\ref{cond1}) has also been achieved
in \cite[Theorems 5,6]{Qwang1} under the assumption (\ref{aben1})
with the help of a bootstrap argument (see \cite[BA1--BA3]{Qwang1}),
in particular involved with energy estimate for the geometric wave
equation of the second fundamental form $k$.  Therefore we can apply
Theorem \ref{T8.1} to close the proof for Theorem \ref{thm3}
(\cite[Theorem 1]{Qwang1}).

We recall that in \cite{KRradius,KR2} the  estimates  in Theorem
\ref{T8.1} were obtained under the assumption
\begin{equation}\label{aben2}
\sup_{[t_0,t_*)}\left(\|k\|_{L^\infty(\Sigma_t)}+\|\nab\log
n\|_{L^\infty( \Sigma_t)}\right)=\Lambda_0<\infty
\end{equation}
combined with the same assumptions on $\R(\H)$ and $n$, both of
which can be obtained under  (\ref{aben2}) or the weaker assumption
(\ref{aben1}). The key improvement in Theorem \ref{T8.1} lies in
that it relies on the weaker assumption
\begin{equation}\label{n1pi}
 \N_1(\sl{\pi})\le\RR,
 \end{equation}
which comes naturally as a consequence of (\ref{aben1}).

Due to the much weaker assumption (\ref{n1pi}), we no longer can
adopt the approach contained in \cite{KRradiusp,KRradius,KR2} to
control Ricci coefficients and $|a-1|$. For instance, let us
consider the estimate on $|a-1|$.  As required in \cite{KR2} and
\cite{Qwang1}, we need to show that
\begin{equation}\label{a11}
|a-1|\le \f12.
\end{equation}
For $q>1$, an assumption of
\begin{equation}\label{app}
\left(\int_{t_0}^{t_*}(\|k\|_{ L^\infty(\Sigma)}^q+\|\nab\log
n\|_{L^\infty(\Sigma)}^q)\right)^{\frac{1}{q}}<\Lambda_0<\infty
\end{equation}
by rescaling takes the following form on null cone $\H$
\begin{equation*}\|k\|_{L_t^q L^\infty_x(\H)}+\|\nab\log n\|_{L_t^q L^\infty_x(\H)}<\RR,
\end{equation*}
which by (\ref{c5}) and (\ref{c4}) immediately gives
\begin{equation}\label{esp}
\|\nu\|_{L_t^q L_x^\infty(\H)}+\|\zb\|_{L_t^q L_x^\infty(\H)}\les
\RR.
\end{equation}
In view of the definition  $\nu:=-\sn_L a$ and $a(p)=1$, we can
obtain (\ref{a11}) by integrating along any null geodesic
$\Ga_\omega$ as long as  $\RR$ is sufficiently small.

If $q=1$, the above simple argument fails due to the fact that  by
recaling, there holds
\begin{equation*}
\|k\|_{L_t^1 L_x^\infty(\H)}+\|\nab \log n\|_{L_t^1
L_x^\infty(H)}<\k
\end{equation*}
which fails to be small. Thus, it is impossible to derive
(\ref{esp}) immediately from assumption.
 Theorem \ref{T8.1} however provides the trace estimate for $\|\nu\|_{L_x^\infty L_t^2(\H)}$ in (\ref{mt2}) which is strong enough to guarantee  $|a-1|\le \f12$.

As remarked after the statement of Theorem \ref{T8.1}, estimates
such as
\begin{equation*}
\|\zb\|_{L_\omega^\infty L_t^2(\H)}\les \RR \cdots
\end{equation*}
are indispensable for establishing (\ref{mt0}), (\ref{mt1}) and
estimates on $\mu$ and $\sn\tr\chi$, all of which were employed to
prove breakdown criterion in \cite{KR2} and \cite{Qwang1}. The above
estimate for $\zb$ can not directly follow from the assumption
(\ref{app}) with $q<2$.

It is well known that the embedding $H^1(S_t)\hookrightarrow
L^\infty(S_t)$ fails. Therefore the assumption (\ref{n1pi}), which
is a consequence of (\ref{aben1}) by energy estimate, can neither
control $\|\zb,\nu\|_{L_t^2 L_x^\infty}$ immediately nor by Sobolev
embedding. This forces us to estimate the weaker norm
$\|\cdot\|_{L_x^\infty L_t^2(\H)}$, which, according to our
experience, would succeed only when special structures for  $\sn
\zb$ and $\sn \nu$ can be found.

\subsection{Comparison to geodesic foliation}
 Let us draw comparisons between geodesic foliation and time
foliation on null hypersurfaces as follows.

(1) In the case of geodesic foliation (\cite{KR1,Qwang}),  the
complete set of estimates in the main theorem can be obtained under
the small curvature flux only. However in time foliation,
(\ref{cond1}) contains one more assumption (\ref{n1pi}).
 In order to understand the reason for assuming
 (\ref{n1pi}),
let us sketch the approach to derive (\ref{mt2}).
 We will  prove and employ the sharp trace inequality (see Theorem \ref{T1.1}) on null
cones with time foliation, which lied in the heart of
\cite{KR1,KRs,KR4} for null hypersurfaces with geodesic foliation.
To implement this idea,
  it is a must to control
$\N_1(\nu, \zb)$. In view of (\ref{c5}) and (\ref{c4}), $\nu$ and
$\zb$ are combinations of elements  of  $\sl{\pi}$, therefore the
assumption (\ref{n1pi}) guarantees the necessary control on $\nu,
\zb$.


(2) The identity $\zeta+\zb=0$ only  holds on null hypersurface $\H$
with geodesic foliation. Therefore, in the case of time foliation,
the estimates for $\zb$ are no longer identical to those for
$\zeta$. Note that $\zb$ and $\zeta$ are on an equal footing in
structure equations (\ref{s4}) and (\ref{mumu}), we need
$\N_1(\cdot)$ and $\|\cdot \|_{L_x^\infty L_t^2(\H)}$ estimates for
both of them. $\zeta$ will be treated in the same fashion as in
geodesic foliation, while the estimate for $\|\zb\|_{L_x^\infty
L_t^2(\H)}$ requires further study on structure equations. Similar
to (\ref{hdg1}) and (\ref{hdg2}), the quantity $\underline{\mu}$ is
connected to $\zb$ by the Hodge system, (\ref{s5}) and (\ref{0.8}),
$$
\left\{\begin{array}{lll} \div\zeta=-\rhoc-\mu+\cdots\\
\curl\zeta=\ckk \sigma,
\end{array}\right.
\quad\,
\left\{\begin{array}{lll} \div\zb=-\rhoc+\f12\und\mu+\cdots\\
\curl\zb=-\ckk \sigma.
\end{array}\right.
$$
 However
 $\underline{\mu}$ fails to satisfy a similar transport equation as ( \ref{mumu}) for
$\mu$, $\und\mu$ consequently does not verify  $\|\cdot\|_{\P^0}$
estimate as $\mu$,  the treatment of $\zb$ is therefore different
from $\zeta$.  Note that
 if there holds the following decomposition for  $\sn\zb$,
\begin{equation}\label{snz}
\sn\zb=\sn_L P+E
\end{equation}
with $P$ and $E$ satisfying appropriate estimates, we may rely on
the sharp trace inequality to estimate $\|\zb\|_{L_\omega^\infty
L_t^2}$. To obtain the important structure (\ref{snz}), we first
derive a refined Hodge system, (\ref{ldelta}) and (\ref{0.8}),
\begin{equation}\label{hdg3}
\left\{\begin{array}{lll} \div\zb=-\ckk
\rho+L(a\delta+2a\lambda)+\cdots\\
\curl\zb=-\ckk \sigma.
\end{array}\right.
\end{equation}
  The pair of quantities $(\ckk \rho, \ckk\sigma)$ can be decomposed in the same way as contained in \cite{KR1,Qwang}.
 We set $\D_1 :
F\rightarrow (\div F, \curl F)$ for any smooth $S_t$ tangent 1-form
$F$. (\ref{snz})  then will be obtained from (\ref{hdg3})
 by commuting $\sn_L$ with $\sn\D_1^{-1}$.
The new commutator $[\sn_L, \sn \D_1^{-1}](a\delta+2a\lambda)$ will
be decomposed in Proposition \ref{de:er} together with other
commutators arising in control of $\|\chih, \zeta\|_{L_x^\infty
L_t^2(\H)}$.

Recall that in order to control $|a-1|<1/2$ in (\ref{a11}), we need
to estimate $\|\nu\|_{L_\omega^\infty L_t^2}$, which is a quantity
that does not arise in the case of geodesic foliation.  We again
rely on sharp trace inequality to derive $\|\nu\|_{L_\omega^\infty
L_t^2}$, which requires another remarkable structure of the form,
$$\sn \nu=\sn_L P+E.$$ This will be done by deriving
the transport equation for $\sn a$, i.e. (\ref{lloga}),
$$
\sn \nu=-\nab_L(\sn a)-\frac{1}{2}\tr\chi \sn a+\cdots.
$$
To prove sharp trace inequality and to control $P$ and $E$ actually
dominate the article. Further comparison on technical details  will
be made in Section \ref{presob}.

\subsection{Organization of the paper}This paper is organized as follows. In
Section \ref{presob}, we start with providing all the structure
equations and making bootstrap assumptions. Then by using structure
equations (\ref{s1})-(\ref{gauss}) and Sobolev embedding, we
establish a series of  preliminary estimates including weakly
spherical property for the metric $\gac$. In Section \ref{ellp}, we
prove $\|\La^{-\a} \K\|_{L_t^\infty L_x^2}\les \gc$ with $\a\ge 1/2$
and establish a series of elliptic estimates. In Section \ref{relp},
we briefly review the theory of geometric Littlewood Paley
decomposition (GLP) and define Besov norms. We give the equivalence
relation on Besov norms in Proposition \ref{eq1.2} and the reduction
argument in Lemma \ref{correct4} based on the weakly spherical
property for $\gac$.
 With the help of these two arguments, in Section
\ref{strace}, we prove the sharp trace theorem, i.e. Theorem
\ref{T1.1}. In order to obtain the structure required in Theorem
\ref{T1.1}, the commutators involved  are decomposed in Proposition
\ref{de:er}, which is the main purpose of Section \ref{err.e}. In
Section \ref{me}, we estimate $\|\chih, \zb, \zeta,
\nu\|_{L_\omega^\infty L_t^2}$ by using Theorem \ref{T1.1} and
Proposition \ref{de:er}. In Section \ref{ap}, we prove dyadic Sobolev
inequalities and (\ref{ini4}) in Theorem \ref{it0.1}.

\section{\bf Preliminary estimates}\label{presob}
\setcounter{equation}{0}

\subsection{Structure equations}
The proof of Main Theorem relies crucially on the following set of
structure equations. We will prove (\ref{ldelta}) and (\ref{lloga})
and refer the reader to \cite[Chapter 11]{KC} and \cite[Section 2]{KR1} for the derivation of
all other formulae.
\begin{align}
&\frac{d\tr \chi}{ds}+\frac{1}{2}(\tr\chi)^2=-|\chih|^2, \label{s1}\\
&\frac{d\chih_{AB}}{ds}+\tr\chi\chih_{AB}=-\a_{AB}, \label{s2}\\
&\frac{d}{ds} \zeta_A=-\chi_{AB}\zeta_B+\chi_{AB}\zb_B-\b_A, \label{s3}\\
&\frac{d}{ds}\sn \tr\chi+\frac{3}{2}\tr\chi \sn\tr\chi=-\chih\c \sn
\tr\chi-2\chih\c
\sn\chih-(\zeta+\zb)(|\chih|^2+\frac{1}{2}(\tr\chi)^2), \label{s4}\\
&\frac{d}{ds}\tr\chib+\frac{1}{2}\tr\chi\tr\chib=2\div\zb-\chih \c \chibh+2|\zb|^2+2\rho, \label{s5}\\
&\div \chih=\frac{1}{2}\sn \tr \chi +\frac{1}{2}\tr\chi\c\zeta-\chih\cdot \zeta-\beta, \label{0.7}\\
&\curl\zb=\frac{1}{2}\chih\wedge \chibh-\sigma, \label{0.8}\\
&\div\zeta=-\mu-\ckk\rho-|\zeta|^2+\frac{1}{2}a\delta \tr\chi+a\lambda \tr\chi\label{hdg1}\\
&\curl\zeta=\ckk\sigma\label{hdg2}
\end{align}
In what follows, we record null Bianchi equations
\begin{align}
\sn_L \b_A&=\div
\a-2\tr\chi\b_A+(2\zeta_A+\zb_A)\a_{AB}\label{Bia1}\\
L\rhoc+\frac{3}{2}\tr\chi\c \rhoc&=\div
\b+(\zeta+2\zb)\c\b-\frac{1}{2}\chih\c(\sn\hot\zb-\frac{1}{2}\tr\chib\c
\chih+\zb\hot\zb),\label{Bia2}\\
L\sigc+\frac{3}{2}\tr\chi\c\sigc&=-\curl
\b-(\zeta+2\zb)\wedge\b-\frac{1}{2}\chih\wedge(\sn\hot
\zb+\zb\hot\zb),\label{Bia3}\\
\sn_L \udb+\tr\chi\udb&=-\sn \rho+(\sn \sigma)^\star+2\chibh\c
\b-3(\zb\rho-\zb{}^\star\s)\label{Bia4}
\end{align}
where
\begin{align}
\rhoc&=\rho-\frac{1}{2}\chih\c \chibh, \quad
\sigc=\s-\frac{1}{2}\chih\wedge\chibh.
\end{align}
Moreover, there hold
\begin{align}
L(\mu)+\tr\chi\mu &=2\chih\c \sn \zeta+(\zeta-\zb) \c (\sn_A
\tr\chi+\tr\chi\zeta_A)-\frac{1}{2}\tr\chi(\chih\c
\chibh-2\rho+2\zb\c \zeta)\nn\\&\quad\quad+2\zeta\c \chih\c \zeta
+\left(\frac{1}{4}a^2 \tr\chi-a
(\delta+2\lambda)\right)|\chih|^2-\frac{1}{2}a \nu
(\tr\chi)^2\label{mumu},\\
&L(a\delta+2a\lambda)+\frac{3}{2}a \delta \tr\chi=\div \zb+\ckk
\rho+|\zb|^2+a\tr\chi\nab_N\log n\label{ldelta},\\
&\sn_L(\sn a)+\frac{1}{2}\tr\chi \sn a=-\sn \nu-\chih\c \sn
 a -(\zb+\zeta)\c \nu.\label{lloga}
\end{align}
The Gauss curvature $K$ on each $S_t$ verifies
\begin{equation}\label{gauss}
K=-\frac{1}{4}\tr\chi \tr\chib+\frac{1}{2}\chibh\c \chih-\rho.
\end{equation}

The following two commutation formulas hold:

\begin{enumerate}
\item[]
\begin{enumerate}
\item[(i)] For any scalar functions $U$,
\begin{equation}\label{u1}
\frac{d}{ds} \sn_A U+\chi_{AB}\sn_B U=\sn_A F+(\zb+\zeta)F,
\end{equation}
where $F=\frac{d}{ds} U$.

\item[(ii)] For $S_t$ tangent $1$-form $U_A$ satisfying $\frac{d
U_A}{ds}=F_A$, there holds
\begin{equation}\label{u2}
\frac{d}{ds} \div U+\chi_{AB} \sn_A U_B=\div F+(\zeta+\zb)\c F+(\f12
\tr\chi \zb_A-\chih_{AB} \zb_B +\b_A) U_A.
\end{equation}
\end{enumerate}
\end{enumerate}

\begin{proof}[Proof of (\ref{ldelta})]
In view of (\ref{s5}) and (\ref{s1}),
\begin{align*}
L(a^{-1}\tr\chib)&=L(\tr\chib) a^{-1}+\tr\chib L(a^{-1})\\
&=a^{-1}(2 \div \zb+2\rhoc+2|\zb|^2-\f12 \tr\chi\c
\tr\chib)-a^{-2}L(a) \tr\chib,\\
L(a\tr\chi)&=L(\tr\chi) a+\tr\chi
L(a)=-a(|\chih|^2+\f12(\tr\chi)^2)+L(a)\tr\chi
\end{align*}

By using (\ref{m11}),
\begin{align*}
L(2\delta+4\lambda)&=a^{-1}(2\div \zb+2\rhoc+2|\zb|^2)-\f12
\tr\chi(a^{-1}\tr\chib+a\tr\chi)+L\log a(a\tr\chi-a^{-1}\tr\chib)\\
&=a^{-1}2(\div \zb+\rhoc+|\zb|^2)- \tr\chi(\delta+2\lambda)+2 L(\log
a) \tr\theta,
\end{align*}
we can obtain with the help of (\ref{m11}) and (\ref{c5})
\begin{align*}
L(a\delta)&=a L(\delta)+\delta L(a)\\
&= (\div
\zb+\rhoc+|\zb|^2)-\frac{1}{2}a\tr\chi(\delta+2\lambda)+L(a)
(\tr\theta+\delta)-2a L(\lambda)\\
&= \div \zb+\rhoc+|\zb|^2+a\tr\chi(-\frac{3}{2}\delta+\nab_N \log
n)-L(2a\lambda)
\end{align*}
which gives the desired formula.
\end{proof}

\begin{proof}[Proof of (\ref{lloga})]
Recall that by $L(a)=-\nu$, combined with (\ref{u1})
\begin{equation*}
[\sn_L, \sn]a=-\chi\c \sn a-(\zeta+\zb)\nu
\end{equation*}
we may obtain relative to orthornomal frame on $S_{t}$,
\begin{align}
\sn_L \sn_B a+\frac{1}{2}\tr\chi\sn_B a&=-\sn_{B}\nu-\chih_{BC}\c
\sn_C a-(\zeta_B+\zb_B)\nu\nn\end{align}
This completes the proof.
\end{proof}

Let  $\Dt$ denote $\frac{d}{dt}$ along a null geodesics initiating
from $p$. In view of (\ref{st}), $\Dt=an \frac{d}{ds}$.  In
comparison with geodesic foliation, the term $(\zeta+\zb)\sn_L U$ in
(\ref{u1}) is no longer trivial due to $\zeta+\zb\neq 0$. This term
however can be avoided if we consider the commutator $[\Dt, \sn]U$
instead.  Similarly, in view of \cite[Lemma 13.1.2]{KC}, it is
simpler to consider $[\Dt,\sn]F$ than  $[\sn_L, \sn]F$  for
$S$-tangent tensor fields $F$.
\begin{proposition}\label{commu1}
For any smooth scalar function $f$,
\begin{equation}\label{commff}
[\Dt, \sn]f=-an \chi\c \sn f
\end{equation}
In view of (\ref{commff}), (\ref{u2})  and  (\ref{0.7}), there holds
\begin{align*}
[\Dt, \sD]f&=-an tr\chi \sD f-2 an \chih\c \sn^2 f +2an\b\c \sn
f-2an \zb\chih\c \sn f\\
&\quad\, -an \zeta\tr\chi\sn f -an \sn tr\chi \sn f.
\end{align*}
\end{proposition}

 Combining \cite[P.288]{Qwang} with the comparison formulas in \cite[Section 2]{KR1},  we have
\begin{lemma}\label{inii}
\begin{align*}
&V,\sn a,  r\sn \tr\chi, r^2\mu\rightarrow 0 \mbox{ as }
t\rightarrow 0,\quad \lim_{t\rightarrow 0}\|\chih,
\zeta,\zb,\nu\|_{L^\infty(S_t)}<\infty
\end{align*}
\end{lemma}

For $S$ tangent tensor fields $F$ on $\H$, we introduce the
following norms. For $1\le p, q\le\infty$ we define the $L_t^q
L_x^p$ norm on $\H$
\begin{equation*}
\|F\|_{L_t^q L_x^p}:=\left(\int_0^1 \left(\int_{|\omega|=1}|F(t,
\omega)|^p na v_td\mu_{{\Bbb S}^2}\right)^{\frac{q}{p}}
dt\right)^{\frac{1}{q}}
\end{equation*}
and  the $L_x^p L_t^\infty$ norm
\begin{equation*}
\|F\|_{L_x^p L_t^\infty}^p :=\int_{{\Bbb S}^2} \sup_{t\in\Ga_\omega}
\left( v_t |F|^p \right) d\mu_{{\Bbb S}^2}.
\end{equation*}

\subsection{Notations and Bootstrap assumptions}\label{boots}

We fix the following conventions
\begin{enumerate}
\item[$\bullet$]$\sl{\pi}$ denotes the collection of $\eh,\,\ep,\,\delta,\, \nab_N\log n, \, \sn\log
n, \,\lambda,$
\item[$\bullet$]  $\iota:=\tr\chi-\frac{2}{r}$,\,$V:=\tr\chi-\frac{2}{s}$, \, $\kappa:=\tr\chi-(an)^{-1}\ovl{an\tr\chi},$
\item[$\bullet$] $A$ denotes the collection of $ \chih,\,\zeta,\, \zb,\, \nu,$
\item[$\bullet$]$\Ab$ denotes the collection of $A$ and  $\chibh,\, \sn\log a, \,\sl{\pi},$
\item[$\bullet$]The pair of quantities $(M, \D_0 M)$ denotes either $ (\sn \tr\chi,\sn \chih),$ or $(\mu,\sn \zeta),$
\item[$\bullet$] $R_0$ denotes the collection of $\a,\,\b,\rho, \sigma, \udb$,
\item[$\bullet$]$\bar R$ denotes the collection of $ R_0, \tr\chi \Ab,
A\c\Ab,$
\item[$\bullet$] $\tR$ denotes the collection of $\bar R$ and $\sn A$,
\item[$\bullet$] $\H_t:=\cup_{t'\in [0,t]} S_{t'}, \mbox{with}\quad 0<t\le
1,$
\item[$\bullet$] $S:=S_t$,\, $\gac:=r^{-2}\ga$,\, $\ga^{(0)}:=\ga_{{\Bbb
S}^2}$, \,$\K:=K-\frac{1}{r^2}$.
\end{enumerate}

\begin{Assumption}
We make the following bootstrap assumption:
\begin{equation*}
 \|V\|_{L^\infty(\H)}\le \Delta_0, \quad \|\chih, \nu,\zeta, \zb \|_{L_\omega^\infty
L_t^2(\H)}\le \Delta_0,\quad |a-1|\le \f12\tag{BA1}
\end{equation*}
where we can assume that $0<\RR<\Delta_0<1/2.$
\end{Assumption}

 The goal is to improve the inequalities in BA1  with
the $\Delta_0$ replaced by $\gc$, and $|a-1|\le \frac{1}{4}$. When
$\RR$ is sufficiently small, $\gc<\frac{1}{2}\Delta_0$ can be
achieved. We will start with deriving estimates for $\N_1(\Ab)$ by
establishing related estimates for  $M=\sn \tr\chi, \mu$, which will
be contained in Propositions \ref{MN1} and \ref{smr}.  Then we prove
that $\kp$ and $\iota$ verify stronger estimates than $A$, which can be
seen in (\ref{kap}) and Proposition \ref{n1tr}. At last we prove
$(S_t, \gac)$ is weakly spherical.

\subsection{Estimates for $\N_1(\Ab)$, $\|r^{1/2} M\|_{L_x^2
L_t^\infty}$ and $\|M\|_{L^2}$}

Recall a few results that have been proved in \cite{Qwang1} and
\cite{KR1,KR4,WangQ}.

\begin{proposition}\cite{Qwang1}\label{cmps}
Under the assumption BA1, there hold
\begin{equation}\label{cmps1}
C^{-1}\le v_t/s^2\le C,\quad C^{-1}<\frac{r}{s}<C,
\end{equation}
where $C$ is a positive  constant.
\end{proposition}

It is easy to derive from $|V|\le\Delta_0$ in BA1 and Proposition
\ref{cmps} that
\begin{equation}\label{comp1}
|s\tr\chi |+ |r\tr\chi|\le C
\end{equation}
and from $|a-1|\le \f12$ in BA1 and $C^{-1}<n<C$ that
\begin{equation}\label{cp}
C^{-1}< an< C,
\end{equation}
with $C$ positive constants.

With the help of Proposition \ref{cmps},  there hold the following
simple inequalities by Sobolev embedding in 2-D slices $S=S_t$.

\begin{enumerate}

\item[$\bullet$]
Let $\os{f}:=f-\bar f$ for any smooth function $f$ on $S$ where
$\bar f=\frac{1}{|S|}\int_{S} f d\mu_\ga,$ there holds the Poincare
inequality
\begin{equation}\label{poin}
\|r^{-1}\os{f}\|_{L^2(S)}\les \|\sn f\|_{L^2(S)}\tag{\bf\poin}.
\end{equation}

\item[$\bullet$]
For a smooth function $\Omega$ on $S$ with vanishing mean, there
holds the following Sobolev inequality (see \cite{KR4})
\begin{equation}\label{gag}
\|\Omega\|_{L^\infty(S)}\les \|\sn^2 \Omega\|_{L^1(S)}+\|\sn
\Omega\|_{L^2(S)}\tag{\bf\gag},
\end{equation}
which  implies
\begin{align}
&r^{-1}\|\Omega\|_{L^\infty(S)}\les \|\sn^2
\Omega\|_{L^2(S)}+r^{-1}\|\sn
\Omega\|_{L^2(S)},\label{ome1}\\
&\|r^{-1} \Omega\|_{L_t^2 L_x^\infty}\les \|\sn^2
\Omega\|_{L^2}+\|r^{-1}\sn \Omega\|_{L^2}.\label{omeg1}
\end{align}

\item[$\bullet$] Let $F$ be a $S$ tangent tensor field, (see
\cite{KR4})
\begin{equation}\label{sob.01}
\|F\|_{L_x^p(S)}\les \|\sn
F\|_{L_x^2(S)}^{1-\frac{2}{p}}\|F\|_{L_x^2(S)}^{\frac{2}{p}}+\|r^{-1+\frac{2}{p}}F\|_{L_x^2(S)},
\with 2<p<\infty\tag{\bf\sob}.
\end{equation}

\item[$\bullet$] Let $F$ be a $S$ tangent tensor field, there hold (see
\cite{KR1,Qwang})
\begin{align}
&\|r^{-1/2} F\|_{L_x^2 L_t^\infty}+\|F\|_{L_x^4
L_t^\infty}+\|F\|_{L^6}\les\N_1(F)\label{sob.m}\tag{\bf\smi},
\end{align}
\begin{equation}\label{sob.m2}
\|r^{-\f12}F\|_{L^\infty}+\|r^{-1}F\|_{L_t^2 L_x^\infty}\les
\N_2(F)\tag{\bf\smii},
\end{equation}
where
\begin{align}\label{5.01.7}
{\mathcal N}_2(F)&:=\|r^{-2}F\|_{L^2}+\|r^{-1}\sn_L
F\|_{L^2}+\|r^{-1}\sn F\|_{L^2} +\|\sn\Dt
F\|_{L^2}+\|\sn^2F\|_{L^2}.
\end{align}
By interpolation,
\begin{equation}\label{sob.in}
\|r^{-\frac{1}{b}} F\|_{L_t^b
L_x^4}+\|r^{-\frac{1}{q}-\frac{1}{2}}F\|_{L_t^q L_x^2}\les \N_1(F),
\with b\ge4,\, q\ge2.
\end{equation}
\end{enumerate}

\begin{lemma}
\begin{equation}\label{vi}
\|V\|_{L^\infty}\les \Delta_0^2
\end{equation}
\end{lemma}

\begin{proof}
This can be obtained by integrating along $\Ga_\omega$  the equation
(\ref{s1}), i.e.
$$
\frac{d}{ds} V+\frac{2}{s} V=-\frac{1}{2} V^2-|\chih|^2
$$
with the help of Proposition \ref{cmps}, Lemma \ref{inii} and
$\|V\|_{L^\infty}+\|\chih\|_{L_\omega^\infty L_t^2}\le \Delta_0$ in
BA1.
\end{proof}

\begin{lemma}\label{tsp2}
For a $S$ tangent tensor field $F$ verifying
\begin{equation}\label{tsp1}
\sn_L F+\frac{p}{2} \tr\chi F= G\c F+ H
\end{equation}
with $p\ge 1$ certain integer, if $\lim_{t\rightarrow 0} r(t)^p F=0$
and $\|G\|_{L_\omega^\infty L_t^2}\les \Delta_0$, then the following
estimate holds
\begin{equation}\label{tran.2}
|F|\les v_t^{-\frac{p}{2}}\int_0^t { v_\tt}^{\frac{p}{2}} |H|na
d\tt.
\end{equation}
\end{lemma}

We will constantly use the  Hardy-Littlewood inequality for scalar
$f$ on $\H$,
\begin{equation}\label{hlm}
\left\|\frac{1}{s}\int_0^s |f|\right\|_{L_s^2}\les \|f\|_{L_s^2}
\end{equation}
With the help of Lemmas \ref{inii} and \ref{tsp2}, (\ref{s2}),
(\ref{s3}), BA1, (\ref{hlm}) and (\ref{cond1}) we obtain

\begin{lemma}\label{f.01}
\begin{equation}\label{pr.1}
\|r^{-1}\chih, r^{-1}\zeta\|_{L^2}+\|r^{-1/2}\chih,
r^{-1/2}\zeta\|_{L_x^2 L_t^\infty }\les \RR,
\end{equation}
\begin{equation}\label{pr.2}
\|\sn_L \chih,\sn_L \zeta\|_{L^2}\les \RR+\Delta_0^2.
\end{equation}
\end{lemma}

In view of Lemma \ref{f.01}, (\ref{cond1}) and (\ref{sob.m}), by
definition of elements of $\Ab$, we can summarize the estimates for
$\Ab$,

\begin{proposition}\label{propab}
\begin{equation*}
\|r^{-1} \Ab\|_{L^2}+\|r^{-1/2} \Ab\|_{L_x^2 L_t^\infty }+\|\sn_L
\Ab\|_{L^2}\les \Delta_0^2+\RR.
\end{equation*}
\end{proposition}

For the proofs of Lemmas \ref{tsp2} and \ref{f.01} and Proposition
\ref{propab}, see \cite{Qwang1}.

With the help of  Proposition \ref{propab} and Lemma \ref{tsp2}, we
prove the following result under the assumption of BA1.

\begin{lemma}
\begin{equation}\label{eq4}
\|\sn\log s\|_{L_\omega^\infty L_t^2} \les\Delta_0,
\end{equation}
\begin{equation}\label{eq5}
\|\sn\log s\|_{L_t^2 L_\omega^2}+\|s^{1/2} \sn\log s\|_{L_\omega^2
L_t^\infty}\les \gc,
\end{equation}
\begin{equation}\label{eq6}
\|\os{\frac{1}{s}}\|_{L_t^2 L_\omega^2}+\|s^\f12
\os{\frac{1}{s}}\|_{L_t^\infty L_\omega^2}\les \gc,
\end{equation}
\begin{equation}\label{eq.7}
\|\kp\|_{L_t^2 L_\omega^2}+\|r^{1/2}\kp\|_{L_t^\infty
L_\omega^2}\les\gc.
\end{equation}
\end{lemma}

\begin{proof}
Apply (\ref{u1}) to $U=s$, we can derive the transport equation
\begin{equation}\label{trans.s1}
\frac{d}{ds} \sn_A(s)+\frac{1}{2}\tr\chi \sn_A(s)=-\chih_{AB}\sn_B
(s)+\zeta_A+\zb_A.
\end{equation}
Note that $e_A(s)\rightarrow 0, \mbox{ as } t\rightarrow
0$,\begin{footnote}{This initial condition can be easily checked by
using the comparison formulas in \cite[Section 2]{KR1}.
}\end{footnote} in view of Lemma \ref{tsp2} with $G=\chih$ and BA1,
we can derive by integrating along a null geodesic $\Ga_\omega$
initiating from vertex,
\begin{equation}\label{eq.6}
|s^{-1}\sn(s)(t)|\les\frac{1}{s} v_t^{-\f12}\int_0^{t} v_{t'}^\f12
|\zeta+\zb| na dt'.
\end{equation}
Taking $L_t^2$ norm first then $L_\omega^\infty({\Bbb S}^2)$, with
the help of BA1,
\begin{equation*}
\|s^{-1}\sn(s)\|_{L_\omega^\infty L_t^2}\les
\|\zeta+\zb\|_{L_\omega^\infty L_t^2}\les \Delta_0
\end{equation*}
which gives (\ref{eq4}).

By taking $L_t^2$ norm first then $L_\omega^2({\Bbb S}^2)$, we can
obtain from (\ref{eq.6}) by using (\ref{hlm}) that
\begin{equation*}
\|s^{-1}\sn(s)\|_{L_\omega^2 L_t^2}\les
\|r^{-1}(\zeta+\zb)\|_{L^2}\les \gc,
\end{equation*}
where for the last inequality we employed
$\|r^{-1}\Ab\|_{L^2}\les\gc$ in  Proposition \ref{propab}. Similarly
\begin{equation*}
\|s^{-1/2}\sn  s\|_{L_\omega^2 L_t^\infty}\les\gc.
\end{equation*}
Hence (\ref{eq5}) is proved.

Applying (\ref{poin}) to $f=\frac{1}{s}$, (\ref{eq6}) follows as a
consequence of (\ref{eq5}).

According to definition, we can derive
\begin{align}
\kp&=\tr\chi-\frac{2}{s}-(an)^{-1}\ovl{an(\tr\chi-\frac{2}{s})}+\frac{2}{s}(1-(an)^{-1}\ovl{an})\nn\\
&+2\os{\frac{1}{s}}(an)^{-1}\ovl{an}-2(an)^{-1}\ovl{s^{-1}\os{an}}\label{kp}.
\end{align}
By (\ref{poin}), (\ref{cp}) and also in view of  Propositions
\ref{cmps} and \ref{propab}, we obtain
\begin{equation}
\|s^{-1}\os{an}\|_{L_t^2 L_\omega^2}\les \|\sn\log (an)\|_{L_t^2
L_\omega^2}\les\|r^{-1}(\zeta+\zb)\|_{L^2(\H)}\les\gc\label{osc.3},
\end{equation}
and similarly
\begin{equation}
\|s^{-\f12}\os{an} \|_{L_t^\infty L_\omega^2}\les
\|r^{\f12}(\zeta+\zb)\|_{L_t^\infty L_\omega^2}\les\gc.\label{osc.4}
\end{equation}
 Using (\ref{cp}) and (\ref{osc.3}),  the last term in (\ref{kp}) can be
estimated as follows
\begin{align}
\|(an)^{-1}\ovl{s^{-1}\os{an}}\|_{L_t^2}&\les
\|s^{-1}\os{an}\|_{L_t^2 L_\omega^1}\les\gc\label{lkp}.
\end{align}
Combining (\ref{osc.3}), (\ref{lkp}) and (\ref{vi}), we obtain
\begin{equation*}
\|r^{-1}\kp\|_{L^2}\les \|r^{-1}\os{\frac{1}{s}}\|_{L^2}+\gc\les
\gc.
\end{equation*}
where, for the last inequality, we employed (\ref{eq6}).

By  (\ref{eq6}), (\ref{osc.4}) and (\ref{vi}), we can get
\begin{align*}
\|r^{\f12}\kp\|_{L_t^\infty L_\omega^2}&\le\|r^{\f12}
\os{\frac{1}{s}}\|_{L_t^\infty L_\omega^2}+\|(an)^{-1}
r^{-\f12}\os{an}\|_{L_t^\infty
L_\omega^1}+|V|\\&+\|s^{-\f12}((\ovl{an})^{-1}an-1)\|_{L_t^\infty
L_\omega^2}\les\gc.
\end{align*}
The proof is complete.
\end{proof}

\begin{lemma}
\begin{align}
\left\|\ovl{\tr\chi}-\frac{2}{r}\right\|_{L_t^2} &\les\gc, \label{eq8}\\
\left\|\tr\chi-\frac{2}{r}\right\|_{L_t^2 L_\omega^2} &\les
\gc.\label{eq.9}
\end{align}
\end{lemma}

\begin{proof}
We can derive  the transport equation
\begin{equation}\label{trans.as1}
\frac{d}{ds}(r(\ovl{tr\chi}-\frac{2}{r}))=(an)^{-1}\frac{r}{2}\ovl{an
\tr\chi\kp}-r(an)^{-1}\ovl{an|\chih|^2}.
\end{equation}
by combining
\begin{equation}\label{dr}
\frac{d}{ds}r=(an)^{-1}\frac{r}{2}\ovl{an\tr\chi}.
\end{equation}
 with
\begin{equation*}
\frac{d}{ds}\ovl{\tr\chi}=-(an)^{-1}\ovl{an
\tr\chi}\c\ovl{\tr\chi}+(an)^{-1}\ovl{an(\frac{1}{2}(\tr\chi)^2-|\chih|^2)}
\end{equation*}
which can be checked in view of the definition of $\ovl{tr\chi}$ and
(\ref{s1}).

Integrate (\ref{trans.as1}) in $t$ in view of
$r\ovl{\tr\chi}-2\rightarrow 0$ as $t\rightarrow 0$,
\begin{align*}
|\ovl{\tr\chi}-\frac{2}{r}|&\les \frac{1}{r}\int_0^{s(t)}
\{\left|(an)^{-1}\frac{r}{2} \ovl{an\tr\chi\kp}\right|+ r(an)^{-1}
\ovl{an|\chih|^2}\}d s(t').
\end{align*}
In view of (\ref{cp}),  taking $L_t^2$ with the help of (\ref{hlm})
yields
\begin{equation}
\|\ovl{\tr\chi}-\frac{2}{r}\|_{L_t^2}\les\|r\kp\tr\chi\|_{L_t^2
L_\omega^1}+\|r|\chih|^2\|_{L_t^2 L_\omega^1}\label{eq.10}.
\end{equation}
By (\ref{vi}) and (\ref{eq.7}), we can obtain
\begin{align}
\|r\tr\chi\kp\|_{L_t^2 L_\omega^1} &\les \|r V\kp\|_{L_t^2
L_\omega^1}+\|2\frac{r}{s}\kp\|_{L_t^2 L_\omega^1}\les
(\Delta_0^2+1)\|\kp\|_{L_t^2 L_\omega^1}\les\gc.\label{eq.11}
\end{align}
By  Proposition \ref{propab}, we have
\begin{equation}\label{chav}
\|r|\chih|^2\|_{L_t^2 L_\omega^1}\les \|r^{1/2} \chih\|_{L_t^\infty
L_\omega^2}\|r^{1/2}\chih\|_{L_t^2 L_\omega^2}\les\gc.
\end{equation}
(\ref{eq8}) then follows by connecting (\ref{eq.10}), (\ref{eq.11})
and (\ref{chav}).

Note that it is straightforward to have
\begin{equation}\label{eq9.1}
\tr\chi-\frac{2}{r}=V-\ovl{V}+2\os{\frac{1}{s}}+\ovl{\tr\chi}-\frac{2}{r},
\end{equation}
hence
\begin{equation*}
\|\tr\chi-\frac{2}{r}\|_{L_t^2 L_\omega^2}\les
\|V\|_{L^\infty}+\|\os{\frac{1}{s}}\|_{L_t^2
L_\omega^2}+\|\ovl{tr\chi}-\frac{2}{r}\|_{L_t^2}
\end{equation*}
which implies (\ref{eq.9}) with the help of (\ref{vi}), (\ref{eq6})
and (\ref{eq8}).
\end{proof}

\begin{lemma}
Denote by  $\bar R$ one of the quantities, $R_0, \,\tr\chi \Ab,
  \, A\c\Ab$, there holds
\begin{equation}\label{simp1}
\|\bar R\|_{L^2(\H_t)}\les \Delta_0^2+\RR.
\end{equation}
\end{lemma}

\begin{proof}
The estimate about $R_0$ can be obtained directly from
(\ref{cond1}). With the help of Proposition \ref{propab} and BA1
\begin{equation}\label{aab}
\|A\c\Ab\|_{L^2(\H_t)}\les \|\Ab\|_{L_x^2
L_t^\infty}\|A\|_{L_\omega^\infty L_t^2}\les\Delta_0^2+\RR.
\end{equation}
By BA1 and  Proposition \ref{propab}, we have
$$
\|\tr\chi \Ab\|_{L^2(\H_t)}\les \|V\c \Ab\|_{L^2(\H_t)}+\|s^{-1}
\Ab\|_{L^2(\H_t)}\les\|r^{-1}\Ab\|_{L^2}\les\gc.
$$
The estimate thus follows.
\end{proof}

\begin{lemma}\label{gauss3}
Let $\K=K-\frac{1}{r^2}$, then
\begin{equation*}
\|\K\|_{L^2(\H_t)}\les \Delta_0^2+\RR.
\end{equation*}
\end{lemma}

\begin{proof}
In view of (\ref{gauss}) and (\ref{m11}),
\begin{align}
K-\frac{1}{r^{2}}&=\frac{a^2-1}{r^2}+\frac{a^2\iota}{2
r}+\frac{\iota (V+\frac{2}{s}) a^2}{4}-a\tr\chi(\lambda+ \f12\delta)
 -\ckk\rho\label{gauss2}
\end{align}
By $\sn_L a=-\nu$, (\ref{hlm}) and  (\ref{cond1})
\begin{equation}\label{a1}
\left\|\frac{a^2-1}{r^2}\right\|_{L^2(\H_t)}=\left\|r^{-1} \int_0^s
2 a\nu \right\|_{L_\omega^2 L_t^2}\les
\|r^{-1}\sl{\pi}\|_{L^2}\les\gc.
\end{equation}
In view of (\ref{gauss2}), by (\ref{vi}) and $r\approx s$ in
(\ref{cmps1}), also using (\ref{a1}), (\ref{eq.9}), (\ref{simp1}),
\begin{align*}
\left\|K-\frac{1}{r^2}\right\|_{L^2(\H_t)}&\les
\left\|\frac{a^2-1}{r^2}\right\|_{L^2}+\|r^{-1}\iota\|_{L^2}+\|\bar
R\|_{L^2}\les \gc
\end{align*}
which is the desired estimate.
\end{proof}

With Lemma \ref{gauss3}, we can prove the following estimates.

\begin{lemma}\label{hdgf}
Let $\D_1$ be the operator that takes any $S$ tangent 1-form $F$ to
$(\div F, \curl F)$. Let $\D_2$ be the operator that takes any
$S$-tangent symmetric, traceless, 2-tenorfields $F$ to $\div F$.
Denote by $\D$ one of the operators $\D_1, \D_2$. For any
appropriate $S$ tangent  tensor fields $F$ in the domain of $\D$, if
$\|r^{-\f12}F\|_{L_t^\infty L_x^2}\les \Delta_0$, there holds
\begin{equation}\label{gauss4}
\|\sn F\|_{L^2(\H_t)}+\|r^{-1} F\|_{L^2(\H_t)}\les \|\D
F\|_{L^2(\H_t)}+\gc.
\end{equation}
For smooth scalar functions $\Omega$, if  $\|r^{-\f12}\sn
\Omega\|_{L_t^\infty L_x^2}\les\Delta_0$, then
\begin{equation}\label{gauss5}
\|\sn^2 \Omega\|_{L^2(\H_t)} +\|r^{-1}\sn
\Omega\|_{L^2(\H_t)}\lesssim \|\sD\Omega\|_{L^2(\H_t)}+\gc.
\end{equation}
\end{lemma}

\begin{proof}
Let us prove (\ref{gauss5}) first. In view of B\"ochner identity on
$S$,
\begin{equation}\label{Bochner}
\int_{S}|\sn^2 \Omega|^2+K|\sn \Omega|^2=\int_{S} |\sD \Omega|^2,
\end{equation}
we obtain
\begin{equation}\label{eq2}
\int_{S_\tt} |\sn^2 \Omega|^2+r^{-2} |\sn \Omega|^2=\int_{S_\tt}
|\sD\Omega|^2-\int_{S_\tt}\K|\sn\Omega|^2.
\end{equation}
Noticing that on $S_\tt$, there holds by (\ref{sob.01}) and
$\|r^{-\f12}\sn\Omega\|_{L_t^\infty L_x^2}\les \Delta_0$,
\begin{align}
\|\sn\Omega\|_{L_x^4(S_\tt)}^2&\les \|\sn^2
\Omega\|_{L_x^2(S_\tt)}\|\sn \Omega\|_{L_x^2(S_\tt)}+r^{-1}\|\sn
\Omega\|_{L_x^2(S_\tt)}^2\nn\\&\les \Delta_0 \|\sn^2
\Omega\|_{L_x^2(S_\tt)}+\Delta_0^2.\label{gas1}
\end{align}
Integrating (\ref{eq2}) on $0< t'\le t$, in view of   (\ref{gas1})
and Lemma \ref{gauss3}, (\ref{gauss5}) follows by using Young's
inequality.

For $\D_1: F\rightarrow (\div F, \curl F)$ and $\D_2: F\rightarrow
\div F$ we recall the identities (see \cite[Proposition 2.2.1]{KC})
\begin{align*}
&\int_{S} |\sn F|^2+K|F|^2=\int_S |\D_1 F|^2,\quad \int_{S} |\sn
F|^2+2K|F|^2=2\int_{S}|\D_2 F|^2.
\end{align*}
Then (\ref{gauss4}) follows in the same way as (\ref{gauss5}).
\end{proof}

By Lemma \ref{hdgf} and (\ref{simp1}), we can derive the following

\begin{lemma}\label{hdgm1}
For $M=\sn \tr\chi, \mu$, there holds
\begin{equation}\label{hdgm}
\|\D_0 M\|_{L^2(\H_t)}\les \Delta_0^2+\RR+\|M\|_{L^2(\H_t)}.
\end{equation}
More precisely,
\begin{align*}
&\|\sn \chih\|_{L^2(\H_t)}\les \gc+\|\sn \tr\chi\|_{L^2(\H_t)},\\
&\|\sn \zeta\|_{L^2(\H_t)}\les \gc+\|\mu\|_{L^2(\H_t)}.
\end{align*}
\end{lemma}

\begin{proof}
By using $\|r^{-\f12} \Ab\|_{L_x^2 L_t^\infty}\les \gc$ in
Proposition \ref{propab} and applying  Lemma \ref{hdgf} to $F=\chih,
\zeta$, we can derive that
\begin{align*}
\|\sn \chih\|_{L^2(\H_t)} &\les
\|\D_2\chih\|_{L^2(\H_t)}+\Delta_0^2+\RR,\\
\|\sn \zeta\|_{L^2(\H_t)} &\les
\|\D_1\zeta\|_{L^2(\H_t)}+\Delta_0^2+\RR.
\end{align*}
In view of (\ref{0.7}), (\ref{hdg1}) and (\ref{hdg2}),
$$
\D_2\chih=\sn\tr\chi+\bar R,\,\,\quad  \D_1 \zeta=(\mu,0)+\bar R.
$$
By using (\ref{simp1}), (\ref{hdgm}) can be proved.
\end{proof}

Recall that $M=\sn \tr\chi \mbox{ or } \mu$.  (\ref{s4}) and
(\ref{mumu}) can be symbolically
\begin{footnote}
{This means signs and coefficients on the right side of the
expression can be ignored.}
\end{footnote}
recast as
\begin{equation}\label{LM}
\sn_L M+\frac{p}{2}\tr\chi M=\chih\c M+H_1+H_2+H_3
\end{equation}
where \begin{footnote} {In view of $|a-1|\le 1/2$ in BA1, the
factors $a^m, \,m\in {\Bbb N}$  can be ignored in $H_i$ when we
employ (\ref{LM}) to prove Proposition \ref{MN1}.}
\end{footnote}
 \begin{equation*}
H_1=A\c (\D_0 M+F), \quad H_2=\Ab\c A\c A,\quad \mbox{and}\quad H_3=
r^{-1}\bar R, \iota \c \bar R,
\end{equation*}
with
\begin{equation}\label{fm}
(p,F)=\left \{\begin{array}{lll}(2,  \sn \tr\chi) \quad \mbox{if } M=\mu\\
\\(3, 0) \quad \mbox{if } M=\sn\tr\chi.
\end{array}\right.
\end{equation}

We will establish the following estimates

\begin{proposition}\label{MN1}
Let $M$ denote either $\sn \tr\chi$ or $\mu$, there hold
\begin{equation}\label{M1}
\|r^{1/2} M\|_{L_x^2 L_t^\infty}+ \|M\|_{L^2}\les \RR+\Delta_0^2,
\end{equation}
\begin{equation}\label{n1cz}
\|\sn \zeta,\sn \chih\|_{L^2}\les \RR+\Delta_0^2.
\end{equation}
\end{proposition}

\begin{proof}
Noticing that (\ref{n1cz}) can be obtained immediately by combining
the second estimate in (\ref{M1}) with Lemma \ref{hdgm1}, we first
consider the second norm in (\ref{M1}). By (\ref{tran.2}),
(\ref{LM}) and Lemma \ref{inii}, integrating along the null geodesic
$\Ga_\omega$ initiating from vertex, we obtain
\begin{equation}\label{diftrh}
|M|\le \sum_{i=1}^3 \left|v_t^{-\frac{p}{2}}\int_0^t
v_{t'}^{\frac{p}{2}}|H_i|na dt'\right|=I_1(t)+I_2(t)+I_3(t).
\end{equation}
Consider $H_3$ with the help of (\ref{hlm}),
\begin{align}
\int_0^1 \left|v_t^{1/2} I_3(t)\right|^2 na dt&=\int_0^1
\left(v_t^{-\frac{p}{2}+\frac{1}{2}}\int_0^t v_{t'}^\frac{p}{2}
|H_3| na dt'\right)^2 na dt\label{i3}\\
& \les\int_0^1\left|r^{-p+1}\int_0^t {r'}^{p-2} |r'\bar R| na
dt'\right|^2 na dt\nn\\
&\les \left|\int_0^1 |r' \bar R|^2 na d t'\right|\nn
\end{align}
where we employed $v_{t'}\approx (r')^2\approx (t')^2.$

Then by taking $L_\omega^1({\Bbb S}^2)$, with the help of
(\ref{simp1}), we obtain
\begin{equation}\label{i33}
\|I_3\|_{L^2}\les \|\bar R\|_{L^2}\les \RR+\Delta_0^2.
\end{equation}

Similarly, we have
\begin{align*}
\int_0^1 \left|v_t^{1/2}I_2(t)\right|^2 na dt&=\int_0^1
\left(v_t^{-\frac{p}{2}+\frac{1}{2}}\int_0^t v_{t'}^{\frac{p}{2}}
|H_2| na dt'\right)^2 na dt\\&\le\int_0^1
\left(v_t^{-\frac{p}{2}+\frac{1}{2}}\int_0^t v_{t'}^{\frac{p}{2}}
|A|^2|\Ab| na dt'\right)^2 na dt\\
&\les \|A\|_{L_t^2}^4\|r^{1/2} \Ab\|_{L_t^\infty}^2.
\end{align*}
By taking $L_\omega^1$, we obtain in view of BA1 and Proposition
\ref{propab} that
\begin{equation}\label{i2}
\|I_2\|_{L^2}\les \|A\|_{L_\omega^\infty
L_t^2}^2\|r^{1/2}\Ab\|_{L_\omega^2 L_t^\infty} \les \Delta_0^2(\gc).
\end{equation}

Now consider $H_1:=A\c (\D_0 M+F)$. We have
\begin{align}
\left|v_t^{1/2} I_1(t)\right|&\les\|A\|_{L_t^2} \left( v_t^{-p+1}
\int_0^{s(t)} v_{t'}^{p}\left(|\D_0 M|^2+ |F|^2\right)
ds(t')\right)^{1/2}\label{i1t}.
 \end{align}
Take $L_\omega^2$,
\begin{equation*}
\|v_t^{1/2} I_1\|_{L_\omega^2}\les \|A\|_{L_\omega^\infty L_t^2}
(\|\D_0 M\|_{L^2}+\|F\|_{L^2}).
\end{equation*}
Hence in view of (\ref{hdgm}) and BA1
\begin{equation}\label{i1}
\|I_1\|_{L^2}\les \Delta_0(\|M\|_{L^2}+\|F\|_{L^2}+\gc).
\end{equation}
Consequently, in view of (\ref{i33}), (\ref{i2}) and (\ref{i1}), if
$M=\sn \tr\chi$,  we have
\begin{align}
&\|\sn\tr\chi\|_{L^2}\les \Delta_0\|\sn
\tr\chi\|_{L^2}+\Delta_0^2+\RR,\label{M.1}
\end{align}
and  if $M=\mu$
\begin{align}
&\|\mu\|_{L^2}\les \Delta_0(\|\mu\|_{L^2}+\|\sn
\tr\chi\|_{L^2})+\Delta_0^2+\RR. \label{M.2}
\end{align}

From (\ref{M.1}), noticing that $0<\Delta_0<1/2$, we conclude
\begin{equation}\label{M.3}
\|\sn \tr\chi\|_{L^2}\les \Delta_0^2+\RR.
\end{equation}

$ \|\mu\|_{L^2}\les \Delta_0^2+\RR $ then follows from (\ref{M.2})
by using (\ref{M.3}).

Now we consider the first estimate in (\ref{M1}).
\begin{align*}
v_t^{\frac{3}{4}} I_3(t) &\le v_t^{\frac{3}{4}-\frac{p}{2}}
\int_0^{s(t)} v_\tt^{\frac{p}{2}-\f12} |\bar R| ds(\tt) \les
s^{\frac{3}{2}-p} \int_0^{s(t)} s^{p-2} |\bar R r| d s(\tt)\\
&\les\left(\int_0^{s(t)} |\bar R r|^2 d s(t') \right)^{\f12}
\end{align*}
where we used  $s^2\approx v_t$, $r\approx s$ in Proposition
\ref{cmps} to derive the second inequality. Thus
\begin{equation}\label{i32}
\left\|\sup_{t\in(0,1]}|v_t^{\frac{3}{4}}
I_3(t)|\right\|_{L_\omega^2} \les \|\bar R\|_{L^2(\H)}\les \gc.
\end{equation}
We can proceed in a similar fashion for the other two terms,
\begin{align*}
v_t^{\frac{3}{4}} I_2(t) &\le
v_t^{\frac{3}{4}-\frac{p}{2}}\int_0^{s(t)} v_\tt^{\frac{p}{2}} |A|^2
|\Ab| d s(\tt) \les s^{\frac{3}{2}-p} \int_0^{s(t)} s^{p}|A|^2 |\Ab|
d s(\tt)\\
&\les r\|A\|_{L_t^2}^2\|r^{1/2}\Ab\|_{L_t^\infty},
\end{align*}
then
\begin{equation}\label{i2t}
\left\|\sup_{0<t\le 1}|v_t^{\frac{3}{4}}
I_2(t)|\right\|_{L_\omega^2} \les \|A\|_{L_\omega^\infty
L_t^2}^2\|r^{1/2} \Ab\|_{L_\omega^2 L_t^\infty}\les \Delta_0^2(\gc).
\end{equation}
Similarly,
\begin{align*}
&v_t^{\frac{3}{4}} I_1(t) \les \|A\|_{L_t^2}\|r (|\D_0
M|+|F|)\|_{L_t^2}r^{1/2}.
\end{align*}
Taking $L_\omega^2$, using (\ref{n1cz}) for $\D_0 M$ and (\ref{M.3})
for $F$, also in view of BA1
\begin{align}
\left\|\sup_{0<t\le 1}|v_t^{\frac{3}{4}} I_1(t)|
\right\|_{L_\omega^2}&\les (\|\D_0
M\|_{L^2}+\|F\|_{L^2})\|A\|_{L_\omega^\infty
L_t^2}\les(\gc)\Delta_0\label{i1t1}.
\end{align}
Combine (\ref{i32}), (\ref{i2t}), (\ref{i1t1}) we conclude that
$\|r^{1/2} M\|_{L_x^2 L_t^\infty(\H)}\les \gc$.
\end{proof}

Let us summarize the major estimates that have been obtained so far.

\begin{proposition}\label{smr}
Let $M$ denote either $\mu$ or $\sn \tr\chi$.  There holds
\begin{equation}\label{smry1}
\N_1(\Ab)+\|r^{1/2}M\|_{L_x^2 L_t^\infty}+\|M\|_{L^2}\les
\Delta_0^2+\RR
\end{equation}
\end{proposition}

\begin{remark}\label{abb}
By (\ref{sob.m}),(\ref{sob.01}) and (\ref{cond1}), it is easy to
check $(a^m A)$ and $(a^m\Ab)$ with $m\in {\Bbb N}$ verify the same
estimates as $A$ and $\Ab$ respectively. Consequently they can also
be regarded  as elements of $A$ and $\Ab$ respectively.
\end{remark}

\subsection{More Estimates for $\kp$ and $\iota$}\label{kpl}

The main purpose of this subsection is to provide estimates for
$\N_1(\kp, \iota)$ and $\|\kp,\iota\|_{L_t^2 L_\omega^\infty}$. We
first derive a simple consequence from (\ref{smry1}) with the help
of (\ref{omeg1}).

\begin{lemma}\label{osc2}
\begin{equation}\label{osc}
\|r^{-1}\os{an}\|_{L_t^2 L_\omega^\infty} \les\gc
\end{equation}
\end{lemma}

\begin{proof}
Apply  (\ref{omeg1}) to $\Omega=an-\ovl{an}$,
\begin{equation}\label{ome2}
\|r^{-1}\Omega\|_{L_t^2 L_\omega^\infty}\les
\|\sn^2\Omega\|_{L^2}+\|r^{-1} \sn \Omega\|_{L^2}.
\end{equation}
Note that  with the help of $\zeta+\zb=\sn \log(an)$ and (\ref{cp}),
there holds
\begin{equation*}
\|r^{-\f12}\sn \Omega\|_{L_t^\infty L_x^2}\les \|r^{-\f12}\sn\log
(an)\|_{L_t^\infty L_x^2}=\|r^{-\f12} (\zeta+\zb)\|_{L_t^\infty
L_x^2}\les \gc.
\end{equation*}
in view of (\ref{gauss5}),  we deduce
\begin{equation*}
\|r^{-1}\Omega\|_{L_t^2 L_\omega^\infty}\les
\|\sD\Omega\|_{L^2}+\gc.
\end{equation*}
We obtain in view of  $
 \sD(an)= an(|\zeta+\zb|^2+\sD\log (an))$, (\ref{cp})
\begin{align}
\|r^{-1}\Omega\|_{L_t^2 L_\omega^\infty}&\les \|\sD\log
(an)\|_{L^2}+\|\zb+\zeta\|_{L^4}^2+\gc\nn\\&\les
\N_1(\zeta+\zb)(1+\N_1(\zeta+\zb))+\gc \\
&\les \gc\nn
\end{align}
where we employed (\ref{sob.m}) and (\ref{smry1})  for the last two
inequalities.
\end{proof}

\begin{proposition}
\begin{align}
\|r\sn^2 (\frac{1}{s})\|_{L^2} &\les\gc,\label{eq3}\\
\|\os{\frac{1}{s}}, \os{\tr\chi},\kp, \iota \|_{L_t^2
L_\omega^\infty} &\les \gc.\label{kap}
\end{align}
\end{proposition}

\begin{proof}
Using (\ref{trans.s1}), in view of the commutation formula in
\cite[Lemma 13.1.2]{KC},  symbolically, we obtain
\begin{align}
\frac{d}{ds}\sn^2 s+\tr\chi \sn^2 s&=-\frac{1}{2}\sn \tr\chi\sn
s+\chih\c\sn^2 s-\frac{1}{2}\tr\chi (\zeta+\zb) \sn s-\sn \chih\c\sn
s\nn\\&-(\zeta+\zb)\chih\c \sn s+(\zeta+\zb)\c (\zeta+\zb)+\sn
(\zeta+\zb)+(\chi\c \zb+\b)\sn s\nn.
\end{align}
We then rewrite it as
\begin{equation*}
\frac{d}{ds} \sn^2 s+\tr\chi \sn^2 s=\chih\c \sn^2 s+(\bar R+M+\sn
\chih)\c \sn s+A\c A+\sn(\zeta+\zb).
\end{equation*}
Apply Lemma \ref{tsp2} to the above equation with the help of
$\|\chih\|_{L_\omega^\infty L_t^2(\H)}\le \Delta_0$ in BA1 and
$\lim_{t\rightarrow 0}r^2 \sn^2s=0$,
\begin{equation}\label{trans.2s}
|\sn^2 s|\les v_t^{-1}\int_0^{s(t)} v_{t'}| (\bar R+M+\sn \chih)\c
\sn s +A\c A+\sn (\zeta+\zb)| ds(t').
\end{equation}
Hence, by H\"older inequality and (\ref{hlm}), we have
\begin{align}
\|\sn^2 s\|_{L_\omega^2 L_t^2}\les \|s^{-1}\sn s\|_{L_\omega^\infty
L_t^2}&(\|\bar R\|_{L^2}+\|M\|_{L^2}+\|\sn A\|_{L^2})\nn\\&+\|r(A\c
A+\sn A)\|_{L_\omega^2 L_t^2}\label{2s1}.
\end{align}
By (\ref{simp1}), (\ref{M1}), (\ref{eq4}),
 (\ref{aab}) and (\ref{smry1}), we obtain
\begin{equation}\label{2s2}
\|\sn^2 s\|_{L_\omega^2 L_t^2}\les \gc.
\end{equation}
By a straightforward calculation,
\begin{equation*}
-s^2 \sn^2 (\frac{1}{s})=s^2\sn(s^{-2}\sn s)=\sn^2 s-2 s^{-1}\sn s\c
\sn s
\end{equation*}
we deduce
\begin{equation*}
\|r\sn^2(\os{1/s})\|_{L^2}\les \|\sn^2 s\|_{L_t^2
L_\omega^2}+\|\sn\log s\c \sn s\|_{L_t^2 L_\omega^2}.
\end{equation*}
By (\ref{eq5}) and (\ref{eq4})
\begin{equation}\label{la1}
\|\sn\log s\c \sn s\|_{L_t^2 L_\omega^2}\les\|\sn\log
s\|_{L_\omega^\infty L_t^2}\|s\sn \log s\|_{L_\omega^2
L_t^\infty}\les\Delta_0(\gc).
\end{equation}
Combined with (\ref{2s2}),
\begin{equation}\label{oss1}
\|r\sn^2(\os{1/s})\|_{L^2}\les \gc.
\end{equation}

Now we apply (\ref{omeg1}) to $\Omega=r\os{\frac{1}{s}}$. By
(\ref{eq5}), $\|\sn(\os{\frac{1}{s}})\|_{L^2}\les \gc$. Also in view
of (\ref{oss1}), we conclude that
\begin{equation}\label{eq.12}
\|\os{\frac{1}{s}}\|_{L_t^2 L_\omega^\infty}\les \gc.
\end{equation}
By $\os{\tr\chi}=\os{V}+2\os{\frac{1}{s}}$ and (\ref{vi}),
 it follows from (\ref{eq.12}) that
\begin{equation}\label{eq.13}
\|\os{\tr\chi}\|_{L_t^2 L_\omega^\infty}\les \gc.
\end{equation}
In view of  $\iota=\os{\tr\chi}+\ovl{\tr\chi}-\frac{2}{r}$,
(\ref{eq.13}) together with  (\ref{eq8}) implies
$$
\|\iota\|_{L_t^2 L_\omega^\infty}\les\gc.
$$
In view of (\ref{kp}), symbolically,
\begin{equation}\label{kpp}
\kp=V-(an)^{-1} \ovl{an
V}+\os{\frac{1}{s}}\frac{\ovl{an}}{an}+s^{-1}(an)^{-1}\os{an}+(an)^{-1}\ovl{s^{-1}\os{an}}.
\end{equation}
Taking $L_t^2 L_\omega^\infty$ of $\kp$ in view of (\ref{kpp}),   by
(\ref{osc}) and (\ref{cp}), the last two terms are bounded by $\gc$.
Due to (\ref{eq.12}) and (\ref{vi}), $L_t^2 L_\omega^\infty$ of the
remaining three terms are bounded by $\gc$. Thus we conclude
$\|\kp\|_{L_t^2 L_\omega^\infty}\les \gc$.
\end{proof}

\begin{proposition}\label{n1tr}
\begin{equation*}
\N_1(\iota)+\N_1(\kp)\les \gc.
\end{equation*}
\end{proposition}
\begin{proof}
In view of (\ref{s1}) and (\ref{dr}), we have
\begin{equation}
\sn_L(\tr\chi-\frac{2}{r})=-\frac{1}{2}(\tr\chi-\frac{2}{r})
\tr\chi-|\chih|^2-\frac{1}{r}\kp.\nn
\end{equation}
By BA1 and (\ref{eq.9}), (\ref{eq.7}) and (\ref{aab}), we have
$\|\sn_L(\iota)\|_{L^2}\les\gc$. Together with (\ref{eq.9}) and
(\ref{smry1}), we conclude $\N_1(\iota)\les\gc$.

 Now we take $L^2(\H)$ norm of $\frac{d}{ds}\kp$ with the help of
\begin{align}
\frac{d}{ds}\kp+\tr\chi\c\kp
&=-|\chih|^2+(an)^{-2}\ovl{|an\chih|^2}+\frac{1}{2}\kp^2
-\frac{1}{2}(an)^{-2}\ovl{(an)^2\kp^2}\label{lkp1}\\
&+(an)^{-2}\left(\ovl{an\tr\chi}\os{\sn_L(an)}
-\ovl{\os{\sn_L(an)}an\kp}\right)\nn.
\end{align}
By (\ref{eq.7}) and (\ref{comp1}), $\|\tr\chi\kp\|_{L^2}\les \gc$.
For the terms  on the right of (\ref{lkp1}), we first claim
\begin{equation}\label{osc11}
\|r^{-1}\os{\sn_L (an)}\|_{L^2}\les\gc
\end{equation}
Indeed,
\begin{equation*}
\sn\sn_L(an)=\sn(an\sn_L \log (an))=\sn\Dt\log(an)
\end{equation*}
by (\ref{commff}),
\begin{equation}\label{lan1}
\sn\sn_L(an)=\Dt\sn\log (an)+an\chi\c \sn\log (an)=an\{\sn_L
(\zeta+\zb)+\chi\c (\zeta+\zb)\}.
\end{equation}
By (\ref{poin}), (\ref{lan1}) and (\ref{cp})
\begin{align}
\|r^{-1}\os{\sn_L (an)}\|_{L^2}&\les\|\sn\sn_L(an)\|_{L^2}\nn\\
&\les\|\tr\chi A\|_{L^2}+\|\sn_L A\|_{L^2}+\|\chih\c A\|_{L^2}
\end{align}
(\ref{osc11}) follows by using (\ref{simp1}) and Proposition
\ref{propab}.

Using (\ref{osc11}), (\ref{cp})  and (\ref{comp1})
\begin{equation*}
\|(an)^{-2}
\ovl{an\tr\chi}\os{\sn_L(an)}\|_{L^2}\les\|r^{-1}\os{\sn_L
(an)}\|_{L^2}\les \gc.
\end{equation*}
Similarly
\begin{align*}
\|(an)^{-2}\ovl{\os{\sn_L(an)}an\kp}\|_{L^2}&\les\|\os{\sn_L(an)}\|_{L_t^2
L_\omega^2}\les\gc.
\end{align*}

Now consider other terms on the right of (\ref{lkp1}), by
(\ref{eq.7}) and (\ref{kap})
\begin{equation*}
\|\kp^2\|_{L^2}\les \|r\kp\|_{L_t^\infty L_\omega^2}\|\kp\|_{L_t^2
L_\omega^\infty}\les \gc,
\end{equation*}
by (\ref{sob.m}) and (\ref{smry1})
\begin{equation*}
\|\chih\c \chih\|_{L^2}\les \|\chih\|_{L^4}^2\les \gc
\end{equation*}
and  the other two terms can be estimated similarly in view of
(\ref{cp}). Hence
\begin{equation}\label{lkp2}
 \|\sn_L
\kp\|_{L^2}\les \gc.
\end{equation}
At last,  we  obtain by definition
\begin{align*}
\sn \kp&=\sn\tr\chi+\sn\log(an)\c (an)^{-1}\ovl{an\tr\chi}.
\end{align*}
 By using(\ref{comp1}), (\ref{cp}) and (\ref{smry1}),
\begin{equation*}
\|\sn\kp\|_{L^2(\H)}\les\|\sn\tr\chi\|_{L^2(\H)}+\|r^{-1}(\zeta+\zb)\|_{L^2(\H)}\les\gc.
\end{equation*}
Combined with (\ref{eq.7}) and (\ref{lkp2}), we can conclude
$\N_1(\kp)\les \gc$.
\end{proof}

\begin{remark}\label{n1kp}
By Proposition \ref{n1tr} and (\ref{kap}), we can regard $\kp$ and
$\iota$ as elements of $A$.

 Without making the strong  assumption  (\ref{aben2}) (see
\cite{KR2}), under the  weaker condition (\ref{cond1}) only,
$\tr\chi-\frac{2}{r}$ no longer satisfies the $L^\infty(\H)$
estimate as (\ref{mt0}) for $\tr\chi-\frac{2}{s}$. Since $S_t$ is
not a level set of affine parameter $s$,  obviously, $\sn r=0$ while
$\sn s\neq 0$.  We will check the weakly spherical property  for
$\gac=r^{-2}\ga$. Integral operators $\La^{-\a}$,
 geometric Littlewood Paley decompositions $P_k$ and Besov norms
will then be defined by heat flow $U(\tau)$ with respect to $\gac$
instead of $s^{-2}\ga$.
Due to the factor ``$an$" in (\ref{dr}), the nontrivial evolution of
$r$ adds technical complexity. This issue can be settled by proving
 $\tr\chi-(an)^{-1}\ovl{an\tr\chi}$ and $\tr\chi-\frac{2}{r}$ verify
stronger estimates than $\chih$. Examples of the application of
(\ref{kap}) and
 Proposition \ref{n1tr} can be seen in the proof of (\ref{8.0.3}), Proposition \ref{commu0}, Theorem \ref{T1.1}, etc.
\end{remark}

\subsection{Weakly spherical surfaces}\label{wss}

Let $\ga$ be the restriction metric on $S_t$, and define the
rescaled metric $\gac$ on $S_t$ by $\gac=r^{-2}\ga$. Let
${\gamma}_{ij}^{(0)}$ denote the canonical metric on ${\Bbb S}^2$.
Note that
\begin{equation}
\lim_{t\rightarrow 0}\stackrel{\circ}
\gamma_{ij}={\gamma}_{ij}^{(0)},\qquad  \lim_{t\rightarrow
0}\p_k\stackrel{\circ}\ga_{ij}=\p_k{\ga}_{ij}^{(0)}\label{8.1.1}
\end{equation}
where $i, j, k=1,2$. With its aid, using the bootstrap assumption
BA1, (\ref{smry1}), we will prove that for each $0 < t\le 1$ the
leave $(S_t, \gac)$ is a weakly spherical surface.

\begin{proposition}\label{l8.1}
For the transport local coordinates $(t,\omega)$, the following
properties hold true for all surfaces ${S}_t$ of the time foliation
on the null cone $\H$: the metric $\stackrel{\circ}\gamma_{ij}(t)$
on each ${S}_t$ verifies weakly spherical conditions i.e.
\begin{eqnarray}
\|\stackrel{\circ}\gamma_{ij}(t)-\gamma_{ij}^{(0)}
\|_{L^\infty_\omega}\lesssim \Delta_0\label{8.4}\\
\|\p_k \stackrel{\circ}\gamma_{ij}(t)-
\p_k\gamma^{(0)}_{ij}\|_{L^2_\omega L_t^\infty}\lesssim \Delta_0
\label{8.0.3}
\end{eqnarray}
\end{proposition}

\begin{proof}
Since relative to the transport coordinate on $\H$,
$\frac{d}{ds}\ga_{ij}=2\chi_{ij}$,
\begin{equation}\label{dtga}
\frac{d}{dt}(\gac_{ij})=an(\kp \c \ga_{ij}+2\chih_{ij})r^{-2}.
\end{equation}
Integrating (\ref{dtga}) along null geodesic initiating from vertex,
with the help of (\ref{8.1.1}), (\ref{8.4}) follows in view of
(\ref{kap}) and BA1.

Integrating the following transport equation along a null geodesic
initiating from vertex,
\begin{align*}
\frac{d}{dt}\p_k \gac_{ij}&=an\left(\p_k \log(an)\kp \gac_{ij}+\p_k
\tr\chi \gac_{ij}\right.\\
&\quad \, \left. +\kp \p_k \gac_{ij}+2 \p_k \chih_{ij} r^{-2}+2\p_k
\log(an) \chih_{ij}r^{-2}\right)
\end{align*}
where $i, j, k=1,2$,  with the initial condition given by
(\ref{8.1.1}), by $\|\kp\|_{L_t^2 L_\omega^\infty}\les \gc$ in
(\ref{kap}) and a similar argument to Lemma \ref{tsp2}, we can
obtain
\begin{align*}
\left|\p_k\stackrel{\circ}\gamma_{ij}(t)-\p_k{\gamma}^{(0)}_{ij}\right|
&\les\Big|\int_0^{s(t)}\left(\p_k \tr\chi\cdot
{\stackrel{\circ}\ga}_{ij}+r^{-2}(\sn_k\chih_{ij}-\Gamma\cdot
\chih)\right.\\
&\quad \, \left. +\p_k \log(an)\gac_{ij}\kp+2\p_k \log(an)
\chih_{ij}r^{-2}+\kp \p_k \ga^{(0)}_{ij}\right)  d s(\tt)\Big|,
\end{align*}
where $\Gamma$ represents Christoffel symbols, and $\Gamma \cdot
\chih$ stands for the terms $\sum_{l=1}^2\Ga_{ki}^l\chih_{lj}$ with
$l=1,2$. Then with the help of (\ref{8.4}),
\begin{align}
& \left\|\sup_{0<t\le 1}\left|\p_k
\stackrel{\circ}\gamma_{ij}(t)-\p_k{\gamma}^{(0)}_{ij}\right|\right\|_{L_\omega^2} \nn\\
&\qquad \qquad \lesssim\|\Gamma\|_{L_\omega^2
L_t^2}\|\chih\|_{L_\omega^\infty L_t^2}+\|\sn \tr \chi \|_{L_x^2
L_t^1}+\|\sn \chih\|_{L_x^2
L_t^1}\nn\\
&\qquad \qquad +\|r(\zeta+\zb)\c\kp\|_{L_\omega^2 L_t^1}+
\left\|r|\zeta+\zb| |\chih| \right\|_{L_\omega^2 L_t^1}
+\|\kp\|_{L_t^2 L_\omega^2}\label{li2}.
\end{align}
By BA1, (\ref{eq.7}), (\ref{aab}), we obtain the terms in the line
of
\begin{equation*}
(\ref{li2})\les\|\kp\|_{L_t^2 L_\omega^2}(\| A\|_{L_\omega^\infty
L_t^2}+1)+\|A\c A\|_{L^2(\H)}\les\gc.
\end{equation*}
Using Proposition \ref{MN1},
\begin{equation*}
\left\|\sup_{0<t\le 1}\left|\p_k
\stackrel{\circ}\gamma_{ij}(t)-\p_k{\gamma}^{(0)}_{ij}\right|\right\|_{L_\omega^2}\les\gc+\|\chih\|_{L_\omega^\infty
L_t^2} \|\Gamma\|_{L_\omega^2 L_t^2}.
\end{equation*}

Sum over all $i,j,k=1,2$, also using (\ref{8.4})
$$
\|\Ga\|_{L_\omega^2}\les\sum_{i,j,k=1,2} \|\p_k
(\gac_{ij}-{\gamma}^{(0)}_{ij})\|_{L_\omega^2}+C,
$$
where $C$ is the constant such that  the Christoffel symbol of
$\ga^{(0)}$ satisfies $|\p \ga^{(0)}|\le C$, (\ref{8.0.3}) then
follows by  using $\|\chih\|_{L_\omega^\infty L_t^2}\le
\Delta_0<1/2$ in BA1.
\end{proof}

\section{\bf $\|\Lambda^{-\a}\K\|_{L_t^\infty L_x^2}$ and elliptic estimates}\label{ellp} \setcounter{equation}{0}

Define the operator $\La^{a}$ with $a\le 0$ such that  for any
$S$-tangent tensor fields $F$
\begin{equation}\label{laalf}
\La^{a}F:=\frac{r^{-a}}{\Gamma(-a/2)}\int_0^\infty
\tau^{-\frac{a}{2}-1}e^{-\tau}U(\tau)F d\tau,
\end{equation}
where $\Ga$ denotes Gamma function and $U(\tau) F$ is defined on
$(S_t, \stackrel{\circ}\ga)$ by
\begin{equation}\label{ut1}
\frac{\partial}{\partial \tau}U(\tau) F-\Delta_{\stackrel{\circ}\ga}
U(\tau)F=0,\quad U(0)F=F.
\end{equation}
The definition of $\La^a$ extends to the range $a>0$ by defining for
$0<a\le 2m$ that
$$
\La^{a}F=\La^{a-2m}\c (r^{-2}Id-\Delta_\ga)^m F.
$$

We record the  basic properties of ${\La^{a}}$ in the following
result (see \cite{KR4}).

\begin{proposition}\label{lambda:1}
\begin{enumerate}
\item[(i)] $\La^0=Id$ and $\La^{a}\c \La^{b}=\La^{a+b}$ for any $a, b\in {\mathbb R}$.

\item[(ii)] For any $S$-tangent tensor field $F$ and any $a\le 0$
$$
r^{a}\|\La^{a}F\|_{L^2(S)}\les \|F\|_{L^2(S)}.
$$

\item[(iii)] For any $S$-tangent tensor field $F$ and any $b\ge a\ge0$
$$
r^a \|\La^{a}F\|_{L^2(S)}\les r^{b} \|\La^{b}F\|_{L^2(S)}\quad
\mbox{and} \quad  \|\La^{a}F\|_{L^2(S)}\les
\|\La^{b}F\|_{L^2(S)}^{\frac{a}{b}}\|F\|_{L^2(S)}^{1-\frac{a}{b}}.
$$

\item[(iv)] For any $S$-tangent tensor fields $F$ and $G$ and any $0\le a< 1$
$$
\|\Lambda^a (F\cdot G)\|_{L^2(S)} \les \|\Lambda F\|_{L^2(S)}
\|\Lambda^a G\|_{L^2(S)} +\|\Lambda^a F\|_{L^2(S)} \|\Lambda
G\|_{L^2(S)}.
$$

\item[(v)] For any $S$-tangent tensor field $F$ there holds with $2<p<\infty$ and $a>1-\frac{2}{p}$
\begin{equation*}
\|F\|_{L^p(S)}\les \|\La^{a}F\|_{L^2(S)}.
\end{equation*}

\item[(vi)] For any $a\in {\mathbb R}$ and any $S$-tangent tensor field $F$
\begin{equation*}
\|F\|_{H^a(S)}:=\|\La^{a}F\|_{L^2(S)}
\end{equation*}
\end{enumerate}
\end{proposition}

\begin{proposition}\label{ep1.3}
Under the assumption of BA1 , if $\RR>0$ is sufficiently small, then
for all $\frac{1}{2}\le\a<1$ there hold
\begin{align}
&{\underline K}_\a:=\|\La^{-\a}(K-r^{-2})\|_{L_t^\infty L_x^2}\les
\gc,\label{lk}\\ &\|\La^{-\a}\rhoc\|_{L_t^\infty L_x^2}\les
\gc.\label{lrhoc}
\end{align}
\end{proposition}

For smooth scalar functions $f$ on $\H$, with $p>2$, $\a\ge 1/2$,
define the good part of the commutator $[\Lambda^{-\a}, \Dt]f$  by
\begin{align*}
[\Lambda^{-\a}, \Dt]_g f:=r^\a C_\a \int_0^\infty
\tau^{\frac{\a}{2}-1} e^{-\tau}[U(\tau), \Dt] f d\tau, \with
C_\a=\frac{1}{\Ga(\frac{\a}{2})}.
\end{align*}

We first prove Proposition \ref{ep1.3} by assuming the following
result.
\begin{proposition}\label{commu0}
Let $f$ be a smooth scalar function $\H$, there holds
\begin{equation*}
\|r^{-\a}[\Lambda^{-\a}, \Dt]_g f\|_{L_t^1
L_x^2}\les(1+I_\a^{1-\frac{2}{p}})(\gc)\|f\|_{L^2},
\end{equation*}
where $p>2$  and $ 1/2\le \a<1$, and $
I_\a:=1+\K_\a^{\frac{1}{1-\a}}+\K_\a^{\f12}.$
\end{proposition}

\begin{proof}[Proof of Proposition \ref{ep1.3}]
It suffices to prove the following estimates
\begin{align}
&\|\La^{-\a}(\K+\rhoc)\|_{L_t^\infty L_x^2}\les \gc,\label{pr.3}\\
&\|\La^{-\a}\rhoc\|_{L_t^\infty L_x^2}\les
(\gc)(1+I_\a^{1-\frac{2}{p}}), \quad p>2 .\label{pr.4}
\end{align}
By (\ref{pr.3}) and (\ref{pr.4}), according to the definition of
$I_\a$, with $\frac{1}{1-\a}\c (1-\frac{2}{p})< 1$, we can obtain
(\ref{lk}).

Let us first prove (\ref{pr.3}). By (\ref{gauss2}),
\begin{equation}\label{gauss6}
\K+\rhoc=\frac{a^2-1}{r^2}+\frac{a^2 \iota }{2r}+\frac{\iota \tr\chi
a^2}{4}+\tr\chi\Ab,
\end{equation}
By Proposition \ref{lambda:1} (ii), BA1, (\ref{cmps1}) Proposition
\ref{propab} and (\ref{comp1}), for $\a\ge 1/2$,
\begin{equation*}
\|\La^{-\a} (\tr\chi\Ab)\|_{L_t^\infty L_x^2}\les \|r^{\a+1}
\tr\chi\Ab\|_{L_t^\infty L_\omega^2}\les\gc.
\end{equation*}
By Proposition \ref{lambda:1} (ii) and (\ref{cond1}), we have for
$\a\ge\f12$
\begin{align*}
\left\|\La^{-\a} (\frac{a^2-1}{r^2})\right\|_{L_t^\infty L_x^2}
&\les \|r^{-1}\La^{-\a}(a^2-1)\|_{L_t^\infty L_\omega^2}\les
\|r^{{\a}-1} (a^2-1)\|_{L_t^\infty L_\omega^2}\\
&\les \left\|r^{\a-1}\int_0^t a^2 n \sn_L a d\tt
\right\|_{L_t^\infty L_\omega^2}\les \|\sn_L
a\|_{L_\omega^2L_t^2}\\
& \les \|r^{-1}\nu\|_{L^2}\les \gc.
\end{align*}
By Proposition \ref{lambda:1} (ii),  (\ref{sob.m}) and Proposition
\ref{n1tr}, we have for $\a\ge \f12$,
\begin{align*}
&\left\|\La^{-\a}(\frac{a^2\iota}{r}) \right\|_{L_t^\infty L_x^2}
\les \|a^2 r^\a \iota\|_{L_t^\infty L_\omega^2}
\les\N_1(\iota)\les\gc.
\end{align*}
Similarly, by (\ref{comp1}),
\begin{align*}
\|\La^{-\a}(a^2\tr\chi \iota)\|_{L_t^\infty L_x^2}&\les \|r^{1+\a}
(a^2\tr\chi \iota)\|_{L_t^\infty L_\omega^2}\les \|r^\a
\iota\|_{L_t^\infty L_\omega^2}\les\gc.
\end{align*}
This finishes the proof of (\ref{pr.3}).

Next we prove (\ref{pr.4}). Let
$W(t)=(\La^{-\a}\rhoc)^2(t)-(\La^{-\a} \rhoc)^2(0)$. Then
\begin{align}
W(t)&=\int_0^t \int_{S_{t'}} \left(\dt(\La^{-\a}\rhoc)^2+an
\tr\chi(\La^{-\a}\rhoc)^2\right) d\mu_{\ga} dt'\nn\\
&=\int_0^t \int_{S_\tt}\Big(2[\Dt, \La^{-\a}]_g\rhoc \c
\La^{-\a}\rhoc+(an
\tr\chi+\a\ovl{an\tr\chi})(\La^{-\a}\rhoc)^2 \nn\\
&\qquad \qquad\quad +2\La^{-\a} \D_t \rhoc\c \La^{-\a}\rhoc\Big)
d\mu_\ga dt'\label{wt1}.
\end{align}
Let $\vartheta(t)$ be a  smooth cut-off function with
$\vartheta(0)=1$ and supported in $[0,\f12]$, $|\vartheta|\le 1$,
\begin{equation}\label{ini1}
(\La^{-\a}\rhoc(0))^2=
(\vartheta\La^{-\a}\rhoc)(0)^2-(\vartheta\La^{-\a} \rhoc)(1)^2
\end{equation}
can be treated similar to $W(t)$, then
\begin{align}
&\|\La^{-\a}\rhoc\|_{L_t^\infty L_x^2}^2\nn\\
&\les \|[\Dt, \La^{-\a}]_g\rhoc\|_{L_t^1 L_x^2} \c
\|\La^{-\a}\rhoc\|_{L_t^\infty L_x^2}+\|(an
\tr\chi+\a\ovl{an\tr\chi})(\La^{-\a}\rhoc)^2\|_{L_t^1
L_x^1}\label{wt9}\\
&+\int_0^1 \left|\int_{S_\tt} \La^{-\a} \D_t \rhoc\c \La^{-\a}\rhoc
d\mu_\ga \right| d\tt +\int_0^1 \left|
\int_{S_\tt}\vartheta^2\La^{-\a} \D_t \rhoc\c \La^{-\a}\rhoc
d\mu_\ga \right| d\tt\label{wt8}\\
& +\int_0^1 \int_{S_\tt} \left|\vartheta \dt \vartheta
(\La^{-\a}\rhoc)^2\right| d\mu_\ga d\tt\label{wt7}.
\end{align}
In view of (\ref{comp1}) and  Proposition \ref{lambda:1} (ii),
\begin{align*}
\int_0^1\int_{S_\tt}(|an \tr\chi
|+|\ovl{an\tr\chi}|)(\La^{-\a}\rhoc)^2 d\mu_{\ga'}
dt'&\les\|r^{-\f12} \La^{-\a}\rhoc\|_{L^2}^2\les \|\rhoc\|_{L^2}^2.
\end{align*}
By $|\dt \vartheta|\les 1$ and Proposition \ref{lambda:1} (ii), for
$\a\ge \f12$,
\begin{equation*}
 (\ref{wt7})\les \|\La^{-\a}
 \rhoc\|_{L^2}^2\les\|\rhoc\|_{L^2}^2.
\end{equation*}
We only need to  estimate the first term in (\ref{wt8}), and the
second one will follow similarly.
\begin{align}
\int_0^1 \left| \int_{S_{t'}} \La^{-\a} \Dt \rhoc\c \La^{-\a}\rhoc
d\mu_\ga \right|dt'\les\|\La^{-2\a}
\Dt\rhoc\|_{L^2}\|\rhoc\|_{L^2}\label{wt2}.
\end{align}
Assuming
 \begin{equation}\label{wt3}
\|\La^{-2\a} \Dt \rhoc\|_{L^2}\les \gc,
\end{equation}
then $(\ref{wt2})\les (\gc)^2,$ and it follows that
\begin{equation*}
\|\La^{-\a}\rhoc\|_{L_t^\infty L_x^2}^2\les
(\gc)^2(1+I_\a^{1-\frac{2}{p}})\|\La^{-\a}\rhoc\|_{L_t^\infty
L_x^2}+(\gc)^2.
\end{equation*}
Thus (\ref{pr.4}) is proved. (\ref{lrhoc}) follows as an immediate
consequence.

To prove (\ref{wt3}), we rely on the transport equation  derived by
(\ref{Bia2}),
\begin{align*}
\Dt\rhoc+\frac{3}{2}an \tr\chi\rhoc&=\div(an \b)-\sn(an)\b+an A\c
\tR=\div (an\b)+an A\c \tR.
\end{align*}
By Proposition \ref{lambda:1}, (\ref{comp1})  and (\ref{cond1}),
with $\a\ge\f12$,
\begin{align*}
\|\La^{-2\a} (an\tr\chi\rhoc)\|_{L^2}&\les
\|r^{2\a}\tr\chi\rhoc\|_{L^2}\les \|\rhoc\|_{L^2}\les \gc.\nn
\end{align*}
By Proposition \ref{lambda:1} and (\ref{cond1}), for $\a\ge \f12$,
\begin{equation}\label{wt5}
\|\La^{-2\a} \div (an\b)\|_{L^2}\les\|r^{2\a-1}an\b\|_{L^2}\les \gc.
\end{equation}
By  Proposition \ref{lambda:1} (v) and H\"older inequality, we
obtain
\begin{align*}
\int_0^1\|\La^{-2\a}(an A\c \tR)\|_{L^2(S_\tt)}^2
d\tt&=\int_0^1\int_{S_\tt}\La^{-4\a}(an A\c \tR)\c (anA\c \tR)
d\mu_\ga d\tt\\&\le\|\La^{-4\a}(an A\c \tR)\|_{L_t^2
L_x^{4}}\|an A\c \tR\|_{L_t^2 L_x^{4/3}}\\
&\les \|\La^{-4\a+\f12+}(an A\c \tR)\|_{L_t^2
L_x^2}\|A\|_{L_t^\infty L_x^4} \|\tR\|_{L^2}.
\end{align*}
Consequently, by (\ref{sob.m}), (\ref{smry1}), (\ref{simp1}) and
Proposition \ref{lambda:1} (ii),
\begin{align*}
\int_0^t\|\La^{-2\a}(an A\c \tR)\|_{L^2(S_\tt)}^2 d\tt&\les
\|r^{2\a-\f12-}\La^{-2\a}(an A\c \tR)\|_{L^2(\H)}(\gc),
\end{align*}
which implies
\begin{equation}\label{wt6}
\|\La^{-2\a}(an A\c \ti R)\|_{L^2(\H)}\les \gc.
\end{equation}
Thus we complete the proof of (\ref{wt3}).
\end{proof}

\begin{proof}[Proof of Proposition \ref{commu0}]
By definition (\ref{laalf}) and  Proposition \ref{commu1},
\begin{align}
&r^{-\a} [\La^{-\a}, \Dt]_gf\nn \\
&=C_\a \int_0^\infty d\tau \tau^{\frac{\a}{2}-1} e^{-\tau}
\int_0^\tau r^2 U(\tau-\tau')([\sD,\Dt] U(\tau')f-\ovl{an\tr\chi}\sD
U(\tau') f) d\tau'\nn\\
&=C_\a \int_0^\infty  d\tau \tau^{\frac{\a}{2}-1} e^{-\tau}
\int_0^\tau r^2 U(\tau-\tau')(\sn \phi_1(\tau')+\phi_2(\tau')),
\label{av2}
\end{align}
where
\begin{align*}
\phi_1(\tau')&=an (\chih+\kp)\c \sn U(\tau')f,\\
\phi_2(\tau')&= \{an(\b+\sn A+A\c A+r^{-1}A)+\sn(an\kp)\} \sn
U(\tau') f.
\end{align*}
Noticing that $\kp$ can be regarded as an element of $A$, thus we
have $\sn(an\kp)=an A \c A+an \sn A$, and
\begin{align}
\phi_1(\tau')&=an A\c \sn U(\tau')f,\quad \phi_2(\tau')= an(\b+\sn
A+A\c A+r^{-1}A)\sn U(\tau') f.\label{f64}
\end{align}
Let us set
\begin{align*}
\Phi_1&=C_\a\int_0^\infty d\tau \tau^{\frac{\a}{2}-1}
e^{-\tau}\int_0^\tau r^2 U(\tau-\tau') \sn \phi_1(\tau') d\tau'\\
\Phi_2&=C_\a \int_0^\infty d\tau \tau^{\frac{\a}{2}-1} e^{-\tau}
\int_0^\tau r^2 U(\tau-\tau')\phi_2(\tau') d\tau'.
\end{align*}
The difference between
$\phi_1,\, \phi_2$ in (\ref{f64}) and those in \cite[Page 41]{WangQ} is
the extra  factor ``$an$" in (\ref{f64}), which can be easily treated by using the estimates
$\|\sn (an)\|_{L_t^\infty L_x^4}\les \gc$ and (\ref{cp}). Based on
the estimate  $\N_1(A)\les \gc$ which has been proved in (\ref{smry1}) and Proposition \ref{n1tr},  also
using (\ref{cp}),
 we can
proceed exactly as  \cite[Appendix]{KR1} or
 \cite[pages 41--43]{WangQ} to obtain for $\f12\le \a<1$ and $p>2$,
\begin{equation*}
\|\Phi_1\|_{L_t^1 L_x^2}+\|\Phi_2\|_{L_t^1 L_x^2}\les
\|f\|_{L^2}(1+I_\a^{1-\frac{2}{p}})(\gc).
\end{equation*}
The proof is therefore complete.
\end{proof}

\subsection{$L^2$ estimates for Hodge operators}

Consider the following Hodge operators
\begin{footnote}{For various properties of these operators 
please refer to \cite[Page 38]{KC} and \cite[Section 4]{KR1}.}\end{footnote}on 2-surfaces $S:=S_t$

\begin{enumerate}

\item [$\bullet$]The operator ${\D}_1$ takes any 1-forms $F$ into the
pairs of functions $(\div F, \curl F)$.

\item[$\bullet$] The operator ${\D}_2$ takes any symmetric traceless
2-tensors $F$ on $S$ into the 1-forms $\div F$.

\item[$\bullet$] The operator ${}^\star {\D}_1$ takes the pairs of scalar
functions $(\rho, \sig)$ into the 1-forms $-\sn\rho+(\sn\s)^\star$
on $S$.

\item[$\bullet$] The operator ${}^\star {\D}_2$ takes 1-forms $F$ on
$S$ into the 2-covariant, symmetric, traceless tensors
$-\frac{1}{2}\widehat{\L_F\ga}$, where
$$
\label{d2ker}(\widehat{\L_F\ga})_{ab}=\sn_b F_a+\sn_a F_b-(\div
F)\ga_{ab}.
$$
\end{enumerate}

Using Proposition \ref{ep1.3} and B\"ochner identity, we can follow
the the same way as \cite{KR1,WangQ} to obtain elliptic estimates
for Hodge operators.

\begin{proposition}\label{P4.1}
The following estimates hold on $\H$,
\begin{enumerate}
\item[(i)] Let $\D$ denote either $\D_1$ or $\D_2$. The operator $\D$ is invertible on its
range, for  $S$ tangent tensor $F$ in the range of $\D$,
$$
\|\sn {\mathcal D}^{-1} F\|_{L^2(S)}+\|r^{-1}{\mathcal D}^{-1}
F\|_{L^2(S)}\lesssim \|F\|_{L^2(S)}.
$$

\item[(ii)] The operator $(-\sD)$ is invertible on its range
and its inverse $(-\sD)^{-1}$ verifies the estimate
$$
\|\sn^2 (-\sD)^{-1}f\|_{L^2(S)} +\|r^{-1}\sn (-\sD)^{-1}
f\|_{L^2(S)}\lesssim \|f\|_{L^2(S)}.
$$

\item[(iii)] The operator ${}^\star {\D}_1$ is invertible as an
operator defined for pairs of $H^1$ functions with mean zero (i.e.
the quotient of $H^1$ by the kernel of ${}^\star {\D}_1$) and its
inverse ${}^\star {\D}_1^{-1}$ takes $S$-tangent $L^2$ 1-forms
$F$(i.e. the full range of ${}^\star {\D}_1$) into pair of functions
$(\rho, \sig)$ with mean zero, such that $-\sn \rho+(\sn
\s)^\star=F$, verifies the estimate
$$
\|\sn {}^\star{\mathcal D}_1^{-1} F\|_{L^2(S)}\lesssim
\|F\|_{L^2(S)},
$$
and
by (i) and duality argument $$
\|r^{-1}{}^\star\D_1^{-1} F\|_{L^2(S)}\lesssim
\|F\|_{L^2(S)}.
$$
\item[(iv)] The operator ${}^\star {\D}_2$ is invertible as an
operator defined on the quotient of $H^1$-vector fields by the
kernel of ${}^\star {\D}_2$. Its inverse ${}^\star{\D}_2^{-1}$ takes
$S$-tangent 2-forms $Z$ which is in $L^2$ space into $S$ tangent
1-forms $F$ (orthogonal to the kernel of ${\D}_2$), such that
${}^\star {\D}_2 F=Z$, verifies the estimate
$$\|\sn {}^\star \D_2^{-1}Z\|_{L^2(S)}\lesssim
\|Z\|_{L^2(S)}.$$
\end{enumerate}

As a consequence of (i)-(iv), let ${\D}^{-1}$ be one of the
operators ${\D}_1^{-1}$, ${\D}_2^{-1}$, ${}^\star{\D}_1^{-1}$ or
${}^\star{\D}_2^{-1}$. By duality argument, we have the following
estimate for appropriate\begin{footnote}{ By `` appropriate" , we
mean the tensor $F$ such that $\div F$ is in the space where
${\D}^{-1}$ is well-defined.}\end{footnote}tensor fields $F$,
\begin{equation*}
\|{\D}^{-1}\div F\|_{L^2(S)}\les \|F\|_{L^2(S)}.
\end{equation*}
\end{proposition}

Using Proposition \ref{ep1.3} and Lemma \ref{gauss3} and following
the similar argument in \cite[Proposition 4.24 and Lemma 6.14]{KR1}
we can obtain

\begin{lemma}\label{A.8}
Let $\D$ one of the operators $\D_1$, $\D_2$ and ${}^\star \D_1$.
For  $F$ pairs of scalar functions in the first case, $S$-tangent
one form for the second and third case, there hold
$$
{\mathcal N}_2({\mathcal D}^{-1}F)\lesssim {\mathcal N}_1(F)\quad
\mbox{and}\quad {\mathcal N}_1(\sn {\mathcal D}^{-1}F)\lesssim
{\mathcal N}_1(F).
$$
\end{lemma}

\section{\bf A brief review of theory of  geometric Littlewood
Paley}\label{relp} \setcounter{equation}{0}

Consider $\S$  the collection of smooth functions on $[0, \infty)$
vanishing sufficiently fast at $\infty$ and verifying the vanishing
moment property
$$
\int_0^\infty \tau^{k_1} \partial^{k_2} m(\tau) d\tau=0, \quad
k_1+k_2\le N.
$$
We set $m_k(\tau):=2^{2k}m(2^{2k}\tau)$ for some smooth function
$m\in \S$. Recall from  \cite{KR4} the geometric Littlewood-Paley
(GLP) projections $P_k$ associated to $m$ which take the form
$$
P_k F:=\int_0^\infty m_k(\tau) U(\tau) F d\tau
$$
for any $S_t$ tangent tensor field $F$, where $U(\tau) F$ is defined
by the heat flow (\ref{ut1}) on $(S_t, \stackrel{\circ}\ga).$

\begin{proposition}\label{P2}\begin{footnote}{
For more properties, one can refer to \cite{KR4,KRs} and \cite{WangQ}.}
\end{footnote}There exists $m\in \S$ such that the GLP projections $P_k$
associated to $m$ verify $U(\infty)+\sum_{k\in \Bbb Z} P_k^2=Id$. By
$f$ we denote a scalar function and $F$ a $S$-tangent tensor field
on $\H$, the GLP projections $P_k$ associated to arbitrary induced
function $m$ verify the following properties:
\begin{enumerate}
\item[(i)] ($L^p$-boundedness)  For any $1\le p\le \infty$, and
any interval $I\subset {\Bbb Z}$,
\begin{equation*}
\|P_{I} F\|_{L^p( S)}\lesssim \|F\|_{L^p(S)}
\end{equation*}

\item[(ii)](Bessel inequality) For any tensorfield $F$ on $S$,
\begin{align*}
\sum_k\|P_k F\|^2_{L^2(S)}\lesssim \|F\|^2_{L^2(S)},\quad
\sum_{k}2^{2k}r^{-2}\|P_k F\|_{L^2(S)}^2\les \|\sn F\|_{L^2(S)}^2
\end{align*}

\item[(iii)](Finite band property) For any $1\le p\le \infty$,
$k\ge 0$,
\begin{equation}\label{FBD}
\|\sD P_k F\|_{L^p(S)}\lesssim2^{2k}r^{-2}\|F\|_{L^p( S)}.
\end{equation}
Moreover given $m \in {\S}$ we can find $\ti m\in \S$ such that
$2^{2k} P_k=\sD \ti P_k,$  with $\ti P_k$ the geometric Littlewood
Paley projections associated to ${\ti m}$, then
\begin{equation}\label{FD}
P_k F=2^{-2k}\ti P_k \sD F, \quad \|P_k F\|_{L^p(S)}\lesssim 2^{-2k}r^2\|\sD F\|_{L^p(S)}.
\end{equation}
In addition, there hold $L^2$ estimates
\begin{equation}\label{FBB}\left\{
\begin{array}{lll}
\|\sn P_k F\|_{L^2(S)} \lesssim 2^k r^{-1}\|F\|_{L^2(S)}
\\ \|P_k \sn F\|_{L^2(S)}\lesssim
2^k r^{-1}\|F\|_{L^2(S)}\\
\|\sn P_{\le 0}F\|_{L^2(S)}\lesssim r^{-1}\|F\|_{L^2(S)},
\end{array}\right.
\end{equation}
and
\begin{equation}\label{FB}
 \|P_k F\|_{L^2(S)} \lesssim 2^{-k}r\|\sn F\|_{L^2(S)}
\end{equation}
 \item[(iv)] (Bernstein Inequality) For appropriate tensor fields
F on $S$, we have weak Bernstein inequalities
\begin{equation}\label{WB}\left\{
\begin{array}{lll}
\|P_k F\|_{L^p(S)}\les
r^{\frac{2}{p}-1}\left(2^{(1-\frac{2}{p})k}+1\right)\|F\|_{L^2(
S)},\\
 \|P_{\le 0} F\|_{L^p(S)}\les r^{\frac{2}{p}-1}\|F\|_{L^2(S)}
\\\|P_k F\|_{L^2(S)}\les r^{\frac{2}{p}-1}
\left(2^{(1-\frac{2}{p})k} +1\right)\|F\|_{L^{p'}(S)}\\
\|P_{\le 0}F\|_{L^2{(S)}}\les r^{\frac{2}{p}-1}\|F\|_{L^{p'}(S)}.
\end{array}\right.
\end{equation}
where $2\le p<\infty$ and $\frac{1}{p}+\frac{1}{p'}=1$, $k\ge 0$.
\item[(v)] With the help of Proposition \ref{ep1.3},
we have the sharp Bernstein inequalities
\begin{equation}\label{SB}
\|P_k f\|_{L^\infty(S)}\les 2^k r^{-1}\|f\|_{L^2(S)},\quad \|P_k
f\|_{L^2(S)}\les 2^k r^{-1}\|f\|_{L^1(S)}.
\end{equation}
\end{enumerate}
\end{proposition}

 We will also
use the notations for any $S$-tangent tensor field $F$,
\begin{equation}\label{comv}
F_n:=P_n^2 F, \quad F_{\le 0}:=\sum_{k\le 0} P_k^2 F.
\end{equation}
 Now we define for $0\le
\theta\le 1$ the Besov ${\B}^\theta$, ${\P}^\theta$ norms for
$S$-tangent tensor fields  $F$ on $\H$ and $B_{2,1}^\theta$ norm on
$S$ as follows:
\begin{align}
\|F\|_{{\mathcal B}^\theta}&=\sum_{k>0}\|(2^k r^{-1})^{\theta}P_k
F\|_{L_t^\infty
L_x^2}+\|r^{-\theta}F\|_{L_t^\infty L_x^2},\label{g4.5}\\
\|F\|_{{\mathcal P}^\theta}&=\sum_{k>0}\|(2^{k}r^{-1})^{\theta}P_k
F\|_{L_t^2 L_x^2}+\|r^{-\theta}F\|_{L_t^2 L_x^2},\label{g5}\\
\|F\|_{B_{2,1}^\theta(S)}&=\sum_{k>0}\|(2^{k}r^{-1})^{\theta}P_k
F\|_{L_x^2}+\|r^{-\theta}F\|_{L_x^2}.\label{g6}
\end{align}
\begin{remark}
For any $a\in {\mathbb R}$ and any $S$-tangent tensor field $F$
\begin{equation*}
\|F\|_{H^a(S)}^2:=\|\La^{a}F\|_{L^2(S)}^2\approx\sum_{k>0}2^{2k
a}r^{-2a}\|P_k F\|_{L^2(S)}^2+r^{-2a}\|P_{\le 0}F\|_{L_x^2}^2.
\end{equation*}
\end{remark}

In certain situation, it is more convenient to work with the Besov
norms defined by the classical Littlewood-Paley (LP) projections
$E_k$. Recall that (see \cite{Stein1,Stein2}) for any scalar
function $f$ on ${\mathbb R}^2$ we can define
$$
E_k f=\frac{1}{(2\pi)^2} \int_{{\mathbb R}^2} \varphi(\xi/2^k)
\hat{f}(\xi) e^{ix\xi} d\xi, \begin{footnote}{See \cite[Page 2]{KR4}
for the finite band  and sharp Berstein inequalities of
$E_k$.}\end{footnote}
$$
where $\varphi$ is a smooth function support in the dyadic shell
$\{\frac{1}{2}\le |\xi|\le 2\}$ and satisfying $\sum_{k\in {\mathbb
Z}} \varphi(2^{-k}\xi)=1$ when $\xi\ne 0$.

 Define for any $0\le
\theta<1$ the ${\tilde \B}^\theta$ and ${\tilde \P}^\theta$ norms of
any scalar function $f$ on $\H$ by
\begin{align}
\|f\|_{\tilde\B^\theta}&:=\sum_{k>0}\|(2^k r^{-1})^{\theta}E_k
f\|_{L_t^\infty
L_x^2}+\|r^{-\theta}f\|_{L_t^\infty L_x^2},\label{btilde0}\\
\|f\|_{\tilde\P^\theta}&:=\sum_{k>0}\|(2^{k}r^{-1})^{\theta}E_k
f\|_{L_t^2 L_x^2}+\|r^{-\theta}f\|_{L_t^2 L_x^2}\label{btilde}.
\end{align}


Using Proposition \ref{l8.1} and  BA1  we can adapt
\cite[Proposition 3.28]{KR2} to obtain the following lemma.

\begin{lemma}\label{correct4} Under the bootstrap assumptions (BA1) ,
there exists a finite number of vector fields $\{X_i\}_{i=1}^l$
verifying the conditions
$$
\left\{\begin{array}{lll} \|X, r\nab_{0}X\|_{L_t^\infty
L_\omega^\infty}\les 1,
\quad \|r\sn(\nab_{0}X)\|_{L_x^2 L_t^\infty}\les 1,\\
\|(\sn-\nab_0)X\|_{L_x^2 L_t^\infty}\les \Delta_0,\quad \sn_L X=0,
\end{array}\right.
$$
where $\nab_0$ represents the covariant derivative induced by the
metric $r^2 \ga^{(0)}$.  For appropriate $S$-tangent tensor $F\in
L_t^\infty L_x^2$, $F\in \B^\theta$ if and only if $F\c X_i\in
\B^\theta$, and
$$
C^{-1}\sum_{i}\|F\c X_i\|_{\B^\theta}\le  \|F\|_{\B^\theta}\le
C\sum_{i}\|F\c X_i\|_{\B^\theta}, \mbox{ with }\, 0\le \theta<1,
$$
where $C$ is a positive constant. The same results hold for the
spaces $\P^\theta$. Moreover
$$
{\mathcal N}_1(F\Ot X)+\|F\Ot X\|_{L_\omega^\infty L_t^2}\les
{\N}_1(F)+\|F\|_{L_\omega^\infty L_t^2},
$$
where $\Ot$ stands for either a tensor product or a contraction.
\end{lemma}

Lemma \ref{correct4} allows us to define Besov norms for arbitrary
$S$-tangent tensor fields $F$ on $\mathcal H$ by the classical LP
projections.

\begin{definition}\label{DD1}
Let $F$ be an $(m,n)$ $S$-tangent tensor field on $\H$ and let
$F_{i_1 i_2\cdots i_n}^{j_1 j_2 \cdots j_m}$ be the local components
of $F$ relative to $\{X_i\}_{i=1}^l$. We define the
${\tilde\B^\theta}$ and ${\tilde\P^\theta}$ norms of $F$ by
\begin{align*}
\|F\|_{\tilde\B^\theta}=\sum\|F_{i_1 i_2\cdots i_n}^{j_1 j_2 \cdots
j_m}\|_{\tilde\B^\theta} \quad \mbox{and} \quad
\|F\|_{\tilde\P^\theta}=\sum \|F_{i_1 i_2\cdots i_n}^{j_1 j_2 \cdots
j_m}\|_{\tilde\P^\theta},
\end{align*}
where the summation is taken over all possible $(i_1\cdots
i_n;j_1\cdots j_m)$.
\end{definition}

The equivalence between $\B^\theta$, $\P^\theta$ norms and
$\tilde\B^\theta$, $\tilde\P^\theta$ norms is given in the following
result whose proof can be found in \cite{WangQ}.

\begin{proposition}\label{eq1.2}
Under the bootstrap assumptions BA1 for arbitrary $S$-tangent tensor
fields $F$ on $\H$ there hold for $0\le \theta<1$,
\begin{align*}
\|F\|_{ \tilde\B^\theta}\approx\|F\|_{{\mathcal B}^\theta}\quad
\mbox{and}\quad \|F\|_{\tilde\P^\theta}\approx\|F\|_{{\mathcal
P}^\theta}.
\end{align*}
\end{proposition}

\begin{lemma}
Let  $H$ be any $S_t$ tangent tensor  and let $f$ be a smooth
function,
\begin{align}
\|f H\|_{\P^0}&\les \| H\|_{\P^0}\|f\|_{L^\infty}+\|H\|_{L^2(\H)}
\|r^{1/2}\sn f\|_{L_t^\infty L_x^4}.\label{A93}
\end{align}
and the similar estimates hold for $\B^0$ and $B_{2,1}^0(S)$.
\end{lemma}

\begin{proof}
By GLP decomposition, $H=\bar H+\sum_{n\in \Bbb N} P_n^2 H+P_{\le 0}
H$, where $\bar H=U(\infty) H$.
\begin{align}\label{fhp}
\|f H\|_{\P^0}&\le\sum_{k>0}\|P_k(f H_{< k})\|_{L^2}+\sum_{k>0}
\|P_k(f H_{\ge k})\|_{L^2} +\sum_{k>0}\|P_k(f \bar H)\|_{L^2}+\|f
H\|_{L^2}.
\end{align}
Let $H_n:=P_n^2 H$. Consider the first term in (\ref{fhp}). By
(\ref{FB}), (\ref{WB}) and (\ref{FBB}),
\begin{align*}
\sum_{k>0,k>n>0}\|P_k(f H_n)\|_{L^2}&\les \sum_{k>0,k>n>0}2^{-k}\|r
P_k\sn(f H_n)\|_{L^2}\\&\les \sum_{k>0,k>n>0}2^{-k}\left(\|rP_k(\sn
f \c H_n)\|_{L^2}+\|r
P_k(f\c \sn H_n)\|_{L^2}\right)\\
&\les \sum_{k>0,k>n>0}\left(2^{-k+\frac{k}{2}} \|r^{\frac{1}{2}} \sn
f \c
H_n\|_{L_t^2 L_x^{4/3}}+2^{-k+n} \|f\|_{L^\infty} \|H_n\|_{L^2}\right)\\
&\les \sum_{k>0,k>n>0}\left( 2^{-\frac{k}{2}} \|r^{\frac{1}{2}}\sn
f\|_{L_t^\infty L_x^4} \|H_n\|_{L^2}+2^{-k+n}\|f\|_{L^\infty}\|
P_n H\|_{L^2}\right)\\
&\les \|H\|_{L^2}\|r^{1/2}\sn f\|_{L_t^\infty
L_x^4}+\|f\|_{L^\infty} \|H\|_{\P^0},
\end{align*}
and similarly,
 $$\sum_{k>0}\|P_k( f H_{\le 0})\|_{L^2}\les
\|H\|_{L^2}\|r^{1/2}\sn f\|_{L_t^\infty L_x^4}+\|f\|_{L^\infty}
\|H\|_{\P^0}.
$$

Now consider the second term.  By (\ref{FD}), (\ref{FBB}) and
(\ref{WB}), we obtain
\begin{align*}
\sum_{k>0,n\ge k}\|P_k&(f H_n)\|_{L^2}\\
& \les \sum_{k>0,n\ge k}
2^{-2n}\|r^2 P_k (f \sD H_n)\|_{L^2}\\
&\les \sum_{k>0,n\ge k}2^{-2n} \left(\|r^2 P_k\sn(f \sn H_n)\|_{L^2}+\|r^2 P_k(\sn f\c \sn H_n)\|_{L^2}\right)\\
&\les \sum_{k>0,n\ge k}\left(
2^{-2n+k+n}\|f\|_{L^\infty}\|P_n H\|_{L^2}+2^{-2n+\frac{k}{2}}\|r^{\frac{3}{2}}\sn
f\c \sn H_n\|_{L_t^2 L_x^{4/3}}\right).
\end{align*}
 By (\ref{FBB}), we have
\begin{align*}
2^{-2n+k/2} \|r^{\frac{3}{2}}\sn f\c \sn H_n\|_{L_t^2
L_x^{4/3}}&\les 2^{-2n+\frac{k}{2}}\|r^{1/2}\sn f\|_{L_t^\infty
L_x^4}
\|r\sn H_n\|_{L^2}\\
&\les 2^{\frac{k}{2}-n}\|r^{1/2}\sn f\|_{L_t^\infty L_x^4}\|
H\|_{L^2},
\end{align*}
thus
\begin{equation*}
\sum_{k>0,n\ge k}\|P_k(f H_n)\|_{L^2}\les \| H\|_{L^2}\|r^\f12 \sn
f\|_{L_t^\infty L_x^4}+\|f\|_{L^\infty}\|H\|_{\P^0}.
\end{equation*}
At last, due to $\sn \bar H=0$ and (\ref{sob.m}),  we can deduce
\begin{equation*}
\sum_{k> 0}\|P_k(f \bar H)\|_{L^2}\les \sum_{k>0} 2^{-k} \|r^{\f12}
\sn f\|_{L_t^\infty L_x^4} \|\bar H\|_{L^2}\les \|r^\f12 \sn
f\|_{L_t^\infty L_x^4} \|H\|_{L^2}.
\end{equation*}
The proof is complete.
\end{proof}

\subsection{Product estimates and Intertwine estimates}

\begin{proposition}\label{P7.3}
Let ${\mathcal D}$ be one of the operators ${\D}_1$, ${\D}_2$ and
${}^\star{\D}_1$. Then for $1<p\le 2$ and any S-tangent tensor $F$
on ${\mathcal H}$ there holds
$$
\|{\mathcal D}^{-1} F\|_{L^2(S)}\lesssim
\|r^{2-\frac{2}{p}}F\|_{L^p(S)}.
$$
\end{proposition}

\begin{proof}
For $p=2$, the inequality follows immediately from Proposition \ref{P4.1}.
For $p>2$, from (\ref{sob.01}) and Proposition \ref{P4.1} we infer for $p'>
2$ satisfying $\frac{1}{p}+\frac{1}{p'}=1$ that
\begin{align*}
\|r^{\frac{2}{p}-2}{}^\star{\D}^{-1} F\|_{L^{p'}(S)} &\les \|\sn
{}^\star{\D}^{-1}F\|_{L^2(S)}^{1-\frac{2}{p'}}
\|r^{-1}{}^\star{\D}^{-1} F\|_{L^2(S)}^{\frac{2}{p'}}
+\|r^{-1}{}^\star{\D}^{-1}F\|_{L^2(S)}\\
&\les\|F\|_{L^2(S)}
\end{align*}
Thus, by duality,  we complete the proof.
\end{proof}

\begin{lemma}\label{dual}
Let $\D$ denote one of the Hodge operators ${\D}_1$, ${\D}_2$,
${}^\star{\D}_1$ and ${}^\star{\D}_2$, let ${\D}^{-1}$ denote the
inverse of $\D$. For $P_k F$ with $P_k$ the GLP projections
associated to the heat equation (\ref{ut1}), there hold for $k>0$,
$1<p\le 2$,
\begin{align*}
\|{\D}^{-1}P_k F\|_{L_x^2}\les 2^{-k}r\| F\|_{L_x^2} \quad
\mbox{and} \quad \|P_k {\D}^{-1}F\|_{L_x^2}\les
2^{-(2-\frac{2}{p})k}r^{2-\frac{2}{p}}\|F\|_{L_x^p}.
\end{align*}
\end{lemma}

\begin{proof}
The first inequality can be proved by using (\ref{FB}) in
Proposition \ref{P2} and Proposition \ref{P4.1}. The second can be
proved by duality with the help of the first inequality and
(\ref{sob.m}).
\end{proof}

The following result follows from the second estimate in Lemma
\ref{dual} immediately.

\begin{proposition}\label{P7.2}
Let ${\mathcal D}^{-1}$ denote either ${\mathcal D}_1^{-1}$,
${}^\star {\mathcal D}_1^{-1}$, ${\mathcal D}_2^{-1}$, then for
appropriate $S$-tangent tensor fields $F$ on ${\H}$ and any $1<p\le
2$,
\begin{equation}\label{7.14}
\|{\mathcal D}^{-1} F\|_{{\mathcal
P}^{\theta}}\lesssim\|r^{2-\frac{2}{p}-\theta}F\|_{L_t^2L _x^p}.
\end{equation}
\end{proposition}

Since the proof of the Hodge-elliptic estimate for geodesic
foliation contained in \cite[pages 295--301]{Qwang} only relied on
$$
\|\K\|_{L^2}+\|\Lambda^{-\a_0} \K\|_{L_t^\infty L_x^2}\les \gc,
\mbox{ with }\a_0\ge 1/2.
$$
Hence, based on  Lemma \ref{gauss3} and Proposition \ref{ep1.3}, the
same proof also applies to the case of time foliation. We can obtain
the result on the Hodge-elliptic ${\P}^\sigma$ estimates.

\begin{theorem}[Hodge-elliptic $\P^\s$-estimate]\label{in}
Let ${\D}$ denote either ${\D}_1$, ${\D}_2$ or their adjoint
operators ${}^\star{\D}_1$ and ${}^\star{\D}_2$. Then for any
$S$-tangent tensor fields $\xi$ and $F$ satisfying $ \D \xi = F$ and
any $\frac{1}{2}>\sig \ge 0$,
\begin{equation}\label{7.15} \|\sn\xi\|_{{\mathcal P}^\sig} \les \|F\|_{{\mathcal P}^\sig}
+ \Delta_0\|\D^{-1} F\|_{L_t^b L_x^2}^q \|F\|_{L_t^2L_x^2}^{1-q},
\end{equation}
where $1/2\le\a_0<q<1-\s$ and $b>4$.
\end{theorem}

We now give a series of product estimates for Besov norms. Proofs
can be seen in \cite[pages 302--304]{Qwang}.

\begin{lemma}\label{l7.1}
For any $S$-tangent tensor fields $F$ and $G$,
\begin{align}
\| F\cdot G\|_{{\mathcal P}^0}&\les {\mathcal N}_1(
F)(\|r^{-\frac{1}{b}}G\|_{L_t^b L_x^2}+\|r^{\frac{1}{2}}\sn
G\|_{L_t^2L_x^2}) \mbox { with } b>4, \label{7.1} \\
\|F\cdot G\|_{{\mathcal P}^0}&\les {\mathcal N}_2(r^{1/2}F)\|G\|_{{\mathcal P}^0}, \label{7.2} \\
\|F\cdot G\|_{{\P}^0}&\les {\N}_1(r^{\frac{1}{2}}F)\big(\|\sn
G\|_{L_t^2 L_x^2}+\|G\|_{L_\omega^\infty L_t^2}\big).\label{7.3}
\end{align}
\end{lemma}

\begin{corollary}
Regard $\kp, \iota$  also as  elements of $A$, there hold
\begin{align}
\|A\c F\|_{\P^0}&\les (\gc) \N_1(r^\f12 F),\quad \quad\|(\tr\chi,
r^{-1}) F\|_{\P^0}\les \N_1(F),\label{cor2}\\
\|an A\c F\|_{\P^0} &\les (\gc) \N_1(r^\f12 F), \quad \|an (\tr\chi,
r^{-1}) F\|_{\P^0}\les \N_1(F).\label{cor3}
\end{align}
\end{corollary}

\begin{proof}
Let us prove (\ref{cor2}) first. Using (\ref{7.1}) for $b>4$, we
have
\begin{align}
\|A\c F\|_{\P^0}&\les \N_1(A)(\|r^{\f12}\sn
F\|_{L^2(H)}+\|r^{-\frac{1}{b}} F\|_{L_t^b L_x^2} ) \les \N_1(A)\c
\N_1(r^{\f12}F)\label{aa}
\end{align}
where the last inequality follows by using (\ref{sob.in}).

By finite band property of GLP in Proposition \ref{P2}, it is
straightforward to obtain
\begin{equation}\label{l7.2}
\|r^{-1}F\|_{\P^0}\les \|\sn F\|_{L^2}+\|r^{-1}F\|_{L^2}.
\end{equation}
 Noticing that $\tr\chi\c F=\iota\c F+2r^{-1} \c F$ and $\N_1(\iota)\les \gc$ in Proposition \ref{n1tr},
the other inequality in (\ref{cor2}) follows by (\ref{aa}) and
(\ref{l7.2}) with $A$ replaced by $\iota$.

 Applying (\ref{A93}) to $f=an$ and
$H=A\c F, \,\tr\chi F,\, r^{-1}F$, (\ref{cor3}) can be derived by
using (\ref{cor2}) for $H$ and the following inequalities for $f$
\begin{equation}\label{an}
\|\sn (an)\|_{L_t^\infty L_x^4}\approx \|\zeta+\zb\|_{L_t^\infty
L_x^4}\les \gc, \quad\|an\|_{L^\infty}\les 1.
\end{equation}

\end{proof}

\section{\bf Sharp trace theorem}\label{strace}
\setcounter{equation}{0}

The purpose of the section is to prove

\begin{theorem}[Sharp trace theorem]\label{T1.1} Let $F$ be
an $S$-tangent tensor which admits a decomposition of the form $\sn
(an F)=\Dt P+ E$ with tensors $P$ and $E$ of the same type as $F$,
suppose $\lim_{t\rightarrow 0} \|F\|_{L_x^\infty}<\infty$ and
$\lim_{t\rightarrow 0} r|\sn F|<\infty$, there holds the following
sharp trace inequality,
\begin{equation}\label{lin2}
\|F\|_{L_\omega^\infty L_t^2} \les {\mathcal N}_1(F)+{\mathcal
N}_1(P) +\|E\|_{{\mathcal P}^0}.
\end{equation}
\end{theorem}
\begin{remark}
We will employ this theorem to estimate $\|F\|_{L_\omega^\infty
L_t^2}$ for  $F= \zb, \nu,\chih, \zeta$. By local analysis, the two
initial assumptions can be checked for the four quantities.
\end{remark}
 The following result gives the important inequalities to prove Theorem
\ref{T1.1}.

\begin{proposition}\label{main:lem}
Let $p\ge 1$ be any integer, for any $S$-tangent tensor fields $F$,
$H$ and $G$ of the same type. There holds
\begin{align}
&\left \|r^{-p}\int_0^t {r'}^p H\cdot G d
t'\right\|_{{\B}^0}\lesssim \|H\|_{\P^0}({\mathcal
N}_1(G)+\|G\|_{L_\omega^\infty L_t^2})\label{5.33}.
\end{align}
Let us further assume
\begin{equation}\label{ini}
\lim_{t\rightarrow 0} r^p\|F\|_{L_x^\infty}=0,\quad
\lim_{t\rightarrow 0} \|G\|_{L_x^\infty}<\infty, \mbox{ when } p\ge
2,
\end{equation}
 then
there holds for $p\ge 1$,
\begin{align}
&\left\|r^{-p}\int_0^t {r'}^p\Dt F\cdot G d
t'\right\|_{{\B}^0}\lesssim {\mathcal N}_1(F){\mathcal
N}_1(G).\label{5.32}
\end{align}
\end{proposition}
 Using a modified version of \cite[Lemma 5.3]{KRs}, (see in Lemma \ref{sl1} in Appendix), also using  Proposition \ref{eq1.2},
 Proposition \ref{main:lem}  follows by repeating the procedure
in  \cite{KRs} and
\cite[Appendix]{WangQ}. We omit the detail of the proof of
Proposition \ref{main:lem}.

\begin{proof}[Proof of Theorem \ref{T1.1}]
We set $ \varphi(t)=\int_0^t  an |F|^2 dt',$ then $\Dt \varphi= an
|F|^2$. Due to \cite[Proposition 5.1]{KR1} we have
\begin{equation}\label{ebdg1}
\|\varphi\|_{L^\infty(\H)}\les \|\sn
\varphi\|_{\B^0}+\|r^{-1}\varphi\|_{L_t^\infty L_x^2}.
\end{equation}
It is easy to see
\begin{equation}\label{lebdg1}
\|r^{-1}\varphi\|_{L_t^\infty L_x^2}\les \|F\|_{L_\omega^\infty
L_t^2}\c \|r^{-1}F\|_{L^2}.
\end{equation}

We now estimate $\|\sn \varphi\|_{\B^0}$. In view of (\ref{u1}), we
obtain
\begin{equation}\label{extra}
\sn_L\sn\varphi+\frac{1}{2}\tr\chi\sn \varphi=2 (an)^{-1}\sn(an\c
F)\c F-\chih\c \sn \varphi-(\zeta+\zb)|F|^2.
\end{equation}

Using the decomposition $\sn(an\c F)=\Dt P+E$ and Lemma
\ref{correct4}, we pair $\sn \varphi$ with vector field $X=X_i$ to
obtain
\begin{align}
\sn_L (r\sn\varphi\c X)&=-\frac{1}{2}r \kp\sn \varphi\c X-r\chih\sn \varphi\otimes X\nn\\
&+r\left(2\sn_L P\c F\otimes X+2 (an)^{-1}E\c F\otimes
X-|F|^2(\zeta+\zb)\c X\right).\label{eq.14}
\end{align}
We will not distinguish ``$\otimes$" with
``$\cdot$", and also suppress $X$ whenever there occurs no confusion.
Note that under the transport coordinate $(s, \omega_1, \omega_2)$,
we have
$$
\frac{\p \varphi}{\p \omega_i}=\int_0^t\{ 2\l \p_{ \omega_i} F, F\r
+ |F|^2 \p_{\omega_i}\log(an)\} na dt' \quad i=1,2,
$$
by  assumptions on initial condition, $\lim_{t\rightarrow 0}\frac{\p
\varphi}{\p \omega_i}=0$. It then follows  by integrating
(\ref{eq.14})  along a null geodesic $\Ga_\omega$ from $0$ to $s(t)$
that, symbolically,
\begin{align}
(\sn\varphi)(t)&= r^{-1} \int_0^t \{r' an (\kp+\chih)\sn \varphi+ r' E\c F\} dt'\nn\\
&+r^{-1} \int_0^t r' \D_t P\c F dt'+r^{-1}\int_0^t an r'
|F|^2(\zeta+\zb) dt'\label{eq.15}.
\end{align}
Since by Proposition \ref{eq1.2},
\begin{equation}\label{besov}
\|\sn \varphi\|_{\B^0}\approx
\sum_{k>0}\|E_k(\sn\varphi)\|_{L_t^\infty L_x^2}+\|\sn
\varphi\|_{L_t^\infty L_x^2},
\end{equation}
and the estimate of $\|\sn \varphi\|_{L_x^2 L_t^\infty}$ can be
obtained in view of (\ref{eq.15}) and (\ref{sob.m}).
\begin{align}
\|\sn\varphi\|_{L_x^2 L_t^\infty}&\les(\|\kp\|_{L_\omega^\infty
L_t^2}+\|\chih\|_{L_\omega^\infty
L_t^2})\|\sn\varphi\|_{L^2}\label{lin1}\\&+(\|E\|_{L^2}+\|\sn_L
P\|_{L^2}+\N_1(A)\N_1(F))\|r^{-1} F\|_{L^2}\nn
\end{align}
where the term on the right of (\ref{lin1}) can be absorbed in view
of $\|\kp,\chih\|_{L_\omega^\infty L_t^2}\les \Delta_0$,  obtained
from (\ref{kap}) and BA1. It remains to estimate the first term in
(\ref{besov}).

Applying Littlewood Paley decomposition $E_k$ to $\sn \varphi\otimes
X$ via (\ref{eq.15}), we only need to estimate the following terms
\begin{align*}
I_1&=\sum_{k>0} \left\|E_k\int_0^t r' an (\kp+\chih)\sn
\varphi\right\|_{L_t^\infty L_\omega^2},\quad
I_2=\sum_{k>0}\left\|E_k \int_0^t  r'E\c F d t'\right\|_{L_t^\infty L_\omega^2},\\
I_3&=\sum_{k>0}\left\|E_k\int_0^t r' \D_t P\c F
dt'\right\|_{L_t^\infty L_\omega^2},\quad \quad
I_4=\sum_{k>0}\left\|E_k \int_0^t r'|F|^2(\zeta+\zb)
dt'\right\|_{L_t^\infty L_\omega^2}.
\end{align*}
Use  (\ref{5.33}) and Proposition \ref{eq1.2},
\begin{align}
I_1&\les (\N_1(\chih)+\|\chih\|_{L_x^\infty
L_t^2}+\N_1(\kp)+\|\kp\|_{L_\omega^\infty L_t^2})\|\sn
\varphi\|_{\P^0}.\nn
\end{align}
Consequently, by (\ref{kap}), Proposition \ref{n1tr}, (\ref{smry1})
and  $\|\chih\|_{L_\omega^\infty L_t^2}\le \Delta_0$ in BA1,
\begin{equation*}
I_1\les \Delta_0 \|\sn\varphi\|_{\P^0}.
\end{equation*}
Apply (\ref{5.33}) to $I_2$,
\begin{equation*}
I_2\les \|E\|_{\P^0}(\N_1(F)+\|F\|_{L_\omega^\infty L_t^2}).
\end{equation*}
By  (\ref{5.32}) and Lemma \ref{correct4}
\begin{equation*}
I_3\les \N_1(P)(\N_1(F)+\|F\|_{L_\omega^\infty L_t^2}).
\end{equation*}
By (\ref{5.33}), (\ref{7.3}), BA1 and (\ref{smry1}), we have
\begin{align*}
I_4 &\les (\N_1(\zeta+\zb)+\|\zeta+\zb\|_{L_\omega^\infty
L_t^2})\|F\c F\|_{\P^0}\\
&\les \Delta_0\N_1(F)(\|\sn F\|_{L^2}+\|F\|_{L_\omega^\infty
L_t^2}).
\end{align*}
Thus we conclude
\begin{align*}
\|\sn \varphi\|_{\B^0}&\les \left({\mathcal
N}_1(P)+\|E\|_{\P^0}\right) \left({\mathcal
N}_1(F)+\|F\|_{L_\omega^\infty L_t^2}\right) +\Delta_0
 \|\sn \varphi\|_{{\mathcal P}^0}\\&+\Delta_0\N_1(F)(\|\sn F\|_{L^2}+\|F\|_{L_\omega^\infty L_t^2}).
\end{align*}
This inequality, together with the fact that $\|\sn
\varphi\|_{{\mathcal P}^0}\lesssim\|\sn \varphi\|_{{\mathcal B}^0}$,
yields
\begin{align}
 \|\sn \varphi\|_{{\B}^0}&\les
({\N}_1(F)+\|F\|_{L_\omega^\infty
L_t^2})\left({\N}_1(P)+\|E\|_{{\P}^0}+\Delta_0\N_1(F)\right)\nn.
\end{align}
Combine the above inequality with (\ref{ebdg1}) and (\ref{lebdg1}),
we get
\begin{align*}
\|F\|_{L_\omega^\infty L_t^2}^2&\les ({\N}_1(F)
+\|F\|_{L_\omega^\infty L_t^2})\left({\N}_1(P)
+\|E\|_{{\P}^0}+\N_1(F)\Delta_0\right)\\&+\|F\|_{L_\omega^\infty
L_t^2}\c \|r^{-1}F\|_{L^2}
\end{align*}
which implies (\ref{lin2}) by Young's inequality.
\end{proof}

\section{\bf Error estimates}\label{err.e}
\setcounter{equation}{0} We will employ the following conventions:
\begin{enumerate}
\item[$\bullet$] $\check R $ denotes either the pair $(\check
\rho, -\check\s)$ or $\udb$

 \item [$\bullet$]${\D}^{-1}\check R$ denotes either ${\D}_1^{-1}(\check \rho, -\check
\s)$ or ${}^\star{\D}_1^{-1}\udb$

\item [$\bullet$]${\D}^{-2}\check R$ denotes either
${\D}_2^{-1}{\D}_1^{-1}(\check \rho, -\check \s)$ or
${\D}_1^{-1}{}^\star {\D}_1^{-1}\udb$

\item [$\bullet$]${\D}^{-1}\Dt \check R$ denotes either
${}^\star{\D}_1^{-1}\Dt\udb$ or ${\D}_1^{-1}\Dt(\check \rho, -\check
\s)$

\item [$\bullet$]$C_0(\check R)$ denotes $[\Dt,
{\D}_1^{-1}](\check \rho, -\check \s)$ or $[\Dt,
{}^\star{\D}_1^{-1}]\udb$

 \item [$\bullet$]${\D}^{-2}\Dt \check R$ denotes
${\D}_2^{-1}{\D}_1^{-1}\Dt(\check\rho, -\check \s)$ or
${\D}_1^{-1}{}^\star {\D}_1^{-1} \Dt \udb$

\item [$\bullet$]${\D}^{-1}C_0(\check R)$ denotes  $
{\D}_2^{-1}[\Dt, {\D}_1^{-1}](\check \rho, -\check \s)$ or
${\D}_1^{-1}[\Dt, {}^\star{\D}_1^{-1}]\udb$

\item[$\bullet$]  $\F$ denotes $\D^{-1}\ckk R$ or  $(a\delta+2a\lambda)$.
\begin{footnote}{
For simplicity, we use $(a\delta+2a\lambda)$ to denote the pair of
quantities $(a\delta+2a\lambda, 0)$.}
\end{footnote}
\item[$\bullet$] $\D^{-1}\F$ denotes either $\D^{-2}\ckk R$ or $
\D_1^{-1}(a\delta+2a\lambda)$.
\end{enumerate}

\subsection{Commutation formula}\label{com1}
We will study  error terms which arise from commuting $\Dt$ with
Hodge operators. Regard $\iota$ also an element of $A$,
symbolically, the commutation formula and its good part can be
written as follows
\begin{align}
 [\Dt, \sn] F&= an\left((A+\frac{1}{r})\sn F+(A+\frac{1}{r})\c A\c F+\b\c
 F\right)\label{symbol:comm1},\\
[\Dt, \sn]_g F&:=an\left((A+\frac{1}{r})\sn F+(A+\frac{1}{r})\c A\c
F\right).\label{com2}
\end{align}

Due to the nontrivial factor ``$an$" in (\ref{symbol:comm1}), the
treatment in \cite[Section 6]{Qwang} has to be modified.  We rewrite
equations (\ref{Bia2}), (\ref{Bia3}) and (\ref{Bia4}) as
\begin{align}
&L(\check\rho, -\check\sigma) ={\mathcal D}_1\beta+r^{-1}\check R
+A\cdot\ti R,\label{7.8}\\
&\sn_L\underline{\beta}={}^\star{\mathcal D}_1(\rho,
\sigma)+r^{-1}\check R+A\cdot \ti{R},\label{7.7}
\end{align}
where
$$
\ti R:=R_0+\sn A+A\cdot \underline A+r^{-1}\underline{A}.
$$
We will consider the commutators
\begin{equation} \label{7.40}
C(\check R)=(C_1(\check R), C_2(\check R), C_3(\check R))
\end{equation}
given in \cite[Definition 6.3]{KR1} which, by using the above
conventions, can be written symbolically as
\begin{align*}
C_1(\check R)&=\sn{\mathcal D}^{-1}[\Dt,{\mathcal
D}^{-1}]\check R,\\
C_2(\check R)&=\sn[\Dt, \D^{-1}]{\mathcal
D}^{-1}\check R,\\
C_3(\check R)&= [\Dt, \sn] {\mathcal D}^{-2}\check R.
\end{align*}

Corresponding to (\ref{7.8}) and (\ref{7.7}),  we introduce the
error terms
\begin{align}\label{errp1}
Err:={\D}_1^{-1}\Dt (\check \rho, -\check \sig)-an\b \quad
\mbox{and}\quad \widetilde{Err}&:={}^\star{\D}_1^{-1}\Dt
\udb-an(\rho,\s).
\end{align}
Denote by $\ff$ either $Err$ or $\widetilde{Err}$. Symbolically,
$\ff$ has the form
$$
\ff ={\mathcal D}^{-1}\{an(r^{-1}\check R+A\cdot\ti R)\}.
$$
 We then infer from (\ref{7.8}) and (\ref{7.7}) the symbolic
expression
\begin{equation}\label{se1}
{\mathcal D}^{-1}\Dt\check R=anR_0+\ff.
\end{equation}
By using (\ref{7.14}) with $\theta=0$ and $p=\frac{4}{3}$,
Proposition \ref{P4.1} and the H\"older inequality we infer that
\begin{equation}\label{errp}
\|\ff\|_{{\P}^0}\les\|r^{\f12}A\c\ti R\|_{L_t^2
L_x^\frac{4}{3}}+\|\ckk R\|_{L^2}\les \Delta_0^2+{\R}_0.
\end{equation}

The purpose of this section is to prove

\begin{proposition}\label{de:er}
There hold the following decomposition for commutators,
\begin{align*}
C(\check R) &=\Dt P+E, \\
[\Dt, \sn\D_1^{-1}](a\delta+2a\lambda)&=\Dt P'+ E',
\end{align*}
 where $P, P'$ and $E, E'$ are $S$ tangent tensors verifying
\begin{align*}
&{\mathcal N}_1(P)+\N_1( P')+ \|E\|_{\P^0}+\|E'\|_{{\mathcal
P}^0}\les \Delta_0^2+\R_0.\\
&\lim_{t\rightarrow 0}\left( r\|P\|_{L^\infty(S)}+
r\|P'\|_{L^\infty(S)}\right)=0.
\end{align*}
\end{proposition}

\subsection{Proof of Proposition \ref{P7.1}: Part I}\label{part1}
In order to prove Proposition \ref{de:er}, let us consider the
structure of commutators.
 We first use
(\ref{symbol:comm1}) to write
\begin{align}
\left(C_2(\check R), \sn[\Dt,
\D_1^{-1}](a\delta+2a\lambda)\right)&=\sn[\Dt , {\D}^{-1}]_g \F+\sn
{\D}^{-1}(an\b\c\D^{-1}\F), \label{C2}\\
\left(C_3(\check R), [\Dt,
\sn]\D_1^{-1}(a\delta+2a\lambda)\right)&=[\Dt, \sn]_g
{\D}^{-1}\F+an\b\c {\D}^{-1}\F, \label{C3}
\end{align}
where
$$
[\Dt, {\D}^{-1}]_g \F:={\D}^{-1}\big(an(A+r^{-1})\c\sn
{\D}^{-1}\F+an(A+r^{-1})\c A\c {\D}^{-1}\F\big).
$$
The terms $\sn[\Dt , {\D}^{-1}]_g \F$ and $ [\Dt, \sn]_g
{\D}^{-1}\F$ are the ``good'' parts  in the corresponding
commutators and will be proved to be $\P^0$ bounded (see
(\ref{7.10})-(\ref{7.1.9})). The terms
$\sn{\D}^{-1}(an\b\c{\D}^{-1}\F)$ and $an\b\c {\D}^{-1}\F$ in
(\ref{C2}) and (\ref{C3}) can not be bounded in ${\P}^0$ norm and
 will be further
 decomposed in Section \ref{C23estimates}.

Let us first establish (\ref{7.21}), (\ref{7.10})-(\ref{7.1.9}) in
two steps in the following result

\begin{proposition}\label{P7.1}
For the error terms $C_0(\check R)$, $C_1(\check R)$, $C_2(\check
R)$, $C_3(\check R)$, etc, there hold
\begin{align}
&\|C_0(\check R)\|_{{\mathcal P}^0}\lesssim
\Delta_0^2+{\R}_0,\label{7.21}\\
&\|C_1(\check R)\|_{{\mathcal P}^0}\lesssim \Delta_0^2+{\mathcal
R}_0,\label{str:C1}\\
&C_2(\check R)=\sn{\D}^{-1}(an\beta\cdot {\mathcal
D}^{-2} \check R)+err,\label{7.10}\\
& C_3(\check R)=an\beta\cdot{\mathcal D}^{-2}\check
R+err\label{7.11} \\
&\sn[\Dt, \D_1^{-1}](a\delta+2a\lambda)=\sn
\D^{-1}(an\b\c\D^{-1}(a\delta+2a\lambda))+err\label{7.1.8} \\
&[\Dt, \sn]\D_1^{-1}(a\delta+2a\lambda)=an\b\c
\D^{-1}(a\delta+2a\lambda) +err\label{7.1.9}
\end{align}
with $$ \|err\|_{{\mathcal P}^0}\lesssim \Delta_0^2+{\mathcal
R}_0.$$
\end{proposition}

We will rely on (\ref{smry1}), (\ref{sob.m}), (\ref{cond1}),
Proposition \ref{n1tr} and Remark \ref{n1kp}, i.e.
\begin{equation}\label{As}
\|\Ab\|_{L_x^4 L_t^\infty}+\|\Ab\|_{L^6}+
\N_1(\Ab)+\|r^{\frac{1}{2}}\sn \tr\chi\|_{L_x^2 L_t^\infty}+\|\sn
\tr\chi\|_{L^2}+\|R_0\|_{L^2}\les \gc,
\end{equation}
where $\iota, \kp$  are regarded as elements of $A$.

{\bf Step 1.} We first prove (\ref{7.21}). In view of
(\ref{symbol:comm1}),
\begin{equation}\label{c0r}
C_0({\check R})={\mathcal D}^{-1}(an\{(A+r^{-1})(\sn{\mathcal
D}^{-1} {\check R})+ (A+r^{-1})\cdot A\cdot {\mathcal D}^{-1}
{\check R}+\beta\cdot {\mathcal D}^{-1}{\check R}\}).
\end{equation}
Using (\ref{7.14}) with $\theta=0$ and $p=\frac{4}{3}$, Proposition
\ref{P4.1}, also  with the help of (\ref{As}) and  H\"older
inequality, we can estimate the various terms in (\ref{c0r}) to get
\begin{equation}\label{7.201}
\|C_0(\check R)\|_{{\P}^0}\les \Delta_0^2+\R_0+
\Delta_0\c{\N}_1({\D}^{-1}\check R).
\end{equation}
By the definition of ${\N}_1({\D}^{-1}\check R)$ and Proposition
\ref{P4.1} it follows that
\begin{eqnarray*}
{\mathcal N}_1({\mathcal D}^{-1} \check R) \lesssim {\mathcal
R}_0+\Delta_0^2+\|{\mathcal D}^{-1} \Dt {\check R}\|_{L_t^2 L_x^2}
+\|C_0(\check R)\|_{L_t^2 L_x^2}.
\end{eqnarray*}
While it follows from (\ref{se1}), (\ref{errp}) and (\ref{cond1})
that
\begin{align*}
\|{\mathcal D}^{-1} \Dt {\check R}\|_{L^2} \lesssim
\|\ff\|_{\P^0}+\|an R_0\|_{L^2}\les \Delta_0^2+{\R}_0.
\end{align*}
Combining the above three inequalities and using the smallness of
$\Delta_0$ we obtain (\ref{7.21}).

In the above proof, together with Lemma \ref{A.8}, (\ref{smry1}) and
Remark \ref{abb} we have also verified the following

\begin{proposition}\label{add1}
\begin{align}
\|{\mathcal D}^{-1}\Dt\check
R\|_{L^2}&\lesssim{\R}_0+\Delta_0^2 ,\label{7.22}\\
\|[\Dt, {\mathcal D}^{-1}]\check
R\|_{L^2}&\lesssim{\R}_0+\Delta_0^2,\label{7.23}\\
{\mathcal N}_1(\F) &\lesssim{\mathcal
R}_0+\Delta_0^2,\label{7.24}\\
{\N}_1(\sn{\mathcal D}^{-1}\F) \lesssim {\mathcal R}_0+\Delta_0^2,
\,\,\,\,\, & \,\,\,\,\, {\mathcal N}_2({\mathcal
D}^{-1}\F)\lesssim{\mathcal R}_0+\Delta_0^2,\label{Coro1}
\end{align}
where  $\F$ denotes either $\D^{-1}\ckk R$ or $(a\delta+2a\lambda).$
\end{proposition}

{\bf Step 2.} We will prove (\ref{7.10})-(\ref{7.1.9}). Let us first
establish the following
\begin{lemma}
Denote by $\D^{-1}$  one of the operators among $\D_1^{-1}$,
$\D_2^{-1}$ and ${}^\star \D_1^{-1}$. For any $S_t$ tangent tensor
$H$, $b>4$
\begin{equation}\label{A92}
\|r^{-1-\frac{1}{b}}\D^{-1}(an\sn\D^{-1} H)\|_{L_t^b L_x^2}\les
\N_1(\D^{-1} H)\|\zeta+\zb\|_{L_t^\infty
L_x^4}+\N_1(r^{-\frac{1}{2}} \D^{-1}H).
\end{equation}
\end{lemma}

\begin{proof}
Using Propositions \ref{P4.1} and \ref{P7.3},
\begin{align*}
\|r^{-1-\frac{1}{b}}& \D^{-1}(an\sn\D^{-1} H)\|_{L_t^b L_x^2}\\
&\les \|r^{-1-\frac{1}{b}} \left(|\D^{-1}(\sn(an){\D}^{-1}
H)|+|{\D}^{-1} \sn (an \D^{-1} H)|\right)\|_{L_t^b
L_x^2}\\
&\les\|r^{-\frac{1}{2}-\frac{1}{b}}\sn(an)\D^{-1}H\|_{L_t^b
L_x^{4/3}}+\|r^{-1-\frac{1}{b}}{\D}^{-1} H\|_{L_t^b L_x^2}.
\end{align*}
Since by using (\ref{sob.in}), we obtain
\begin{align*}
\|r^{-\frac{1}{2}-\frac{1}{b}}\sn (an)\D^{-1} H\|_{L_t^b
L_x^{4/3}}&\les \|r^{-\frac{1}{2}-\frac{1}{b}} \D^{-1} H\|_{L_t^b
L_x^2}\|\sn(an)\|_{L_t^\infty L_x^4}\\&\les
\N_1(\D^{-1}H)\|\zeta+\zb\|_{L_t^\infty L_x^4},
\end{align*}
and $ \|r^{-1-\frac{1}{b}}{\D}^{-1} H\|_{L_t^b L_x^2}\les
\N_1(r^{-1/2} \D^{-1}H)$. Then (\ref{A92}) follows.
\end{proof}

The proof of (\ref{7.10})-(\ref{7.1.9}) can be completed by using
(\ref{7.24}) combined with the following result.

\begin{lemma}\label{A.91}
Denote by $\D^{-1}$ either  $\D_1^{-1}$ or $\D_2^{-1}$. For
appropriate $S$-tangent tensor field $F$, there hold
\begin{align}
&\|[\Dt, \sn]_g{\D}^{-1}F\|_{{\mathcal P}^0}+\|\sn [{\D}^{-1},
\Dt]_g F\|_{{\mathcal P}^0}\les{\N}_1(F)\label{it0.211},\\
&\|[\Dt, \sn{\D}^{-1}]_g F\|_{{\mathcal
P}^0}\les{\N}_1(F).\label{A.9}
\end{align}
\end{lemma}

\begin{proof}
Observe that
$$
\|[\Dt, \sn{\D}^{-1}]_g F\|_{{\mathcal P}^0}\les \|[\Dt, \sn]_g
{\D}^{-1}F\|_{{\mathcal P}^0}+\|\sn[{\D}^{-1}, \Dt]_g F\|_{{\mathcal
P}^0},
$$
it suffices to prove (\ref{it0.211}) only.

We first derive in view of (\ref{symbol:comm1}) by using Lemma
\ref{l7.1} with $4<b<\infty$,
\begin{align}
\|[\Dt, \sn]_g{\D}^{-1}F\|_{{\P}^0} &\les {\N}_1(\sn{
\D}^{-1}F)\left(\|r^{\frac{1}{2}}\sn(an A)\|_{L_t^2
L_x^2}+\|r^{-\frac{1}{b}}anA\|_{L_t^b L_x^2}\right)\nn\\
&\quad\, +{\N}_2({\D}^{-1}F)\big(\|anA\c
A\|_{{\P}^0}+\|r^{-\frac{1}{2}}anA\|_{{\P}^0}\big) \label{it0.213.5}\\
&\quad\, +\|r^{-1}an\sn{\D}^{-1} F\|_{{\P}^0}\label{it0.213}.
\end{align}
By (\ref{cor3}) and (\ref{As}), we obtain
\begin{equation*}
\|r^{-1} anA\|_{\P^0}\les \gc, \quad\|an A\c A\|_{\P^0}\les \Delta_0
\N_1(A)\les \gc.
\end{equation*}
Then  the term in (\ref{it0.213.5}) can be bounded by $(\gc)
\N_2(\D^{-1}F)$.

We then consider the term in (\ref{it0.213}). In view of (\ref{A93})
and Theorem \ref{in},
\begin{align*}
\|r^{-1} an\sn \D^{-1} F\|_{\P^0} &\les\|r^{-1}
F\|_{\P^0}+\Delta_0\|r^{-1}{\D}^{-1} F\|_{L_t^b
L_x^2}^q\|r^{-1}F\|_{L^2(\H)}^{1-q}.
\end{align*}
Since by Proposition \ref{P7.3}  and (\ref{sob.m}), we obtain,
\begin{equation*}
\|r^{-1} \D^{-1} F\|_{L_t^b L_x^2}\les \|F\|_{L_t^b L_x^2}\les
\N_1(F),
\end{equation*}
also by using (\ref{l7.2}), we deduce
\begin{align*}
\|r^{-1} an\sn \D^{-1} F\|_{\P^0}&\les \N_1(F)+\Delta_0 \N_1(F).
\end{align*}
By (\ref{As}), it is easy to check
\begin{equation*}
\|r^{\frac{1}{2}}\sn(an A)\|_{L_t^2
L_x^2}+\|r^{-\frac{1}{b}}anA\|_{L_t^b L_x^2}\les \gc.
\end{equation*}
Consequently, we conclude that
\begin{equation}\label{it0.210}
\|[\Dt,
\sn]_g{\D}^{-1}F\|_{{\P}^0}\les(\gc)({\N}_1(\sn{\D}^{-1}F)+{\N}_2({\D}^{-1}
F))+\N_1(F).
\end{equation}
We then infer from (\ref{it0.210}) and Lemma \ref{A.8} that
\begin{equation}\label{it0.210+}
\|[\Dt, \sn]_g{\D}^{-1}F\|_{{\mathcal P}^0}\les {\N}_1(F).
\end{equation}

Next, we prove for $S$-tangent tensor fields $F$ on $\H$ the
following inequality holds\begin{footnote} {We will improve the
right hand side of (\ref{it0.24}) to be ${\N}_1({\D}^{-1}F)$ in the
next section.}
\end{footnote}
\begin{equation}\label{it0.24}
\|[\Dt, {\D}^{-1}]_g F\|_{L_t^b L_x^2}\les {\mathcal N}_1( F) \quad
\mbox{with }\, 4<b<\infty.
\end{equation}
Indeed, by using Proposition \ref{P7.3} with $p=4/3$, Proposition
\ref{P4.1}, (\ref{sob.m}), (\ref{As}), (\ref{A92}) and Lemma
\ref{A.8} we then derive that
\begin{align*}
\|[\Dt,{\D}^{-1}]_g F \|_{L_t^b L_x^2}&\les
\|r^{1/2} A\c \sn{\D}^{-1}F\|_{L_t^b L_x^{4/3}}+\|r^{1/2}A\c A\c {\D}^{-1}F\|_{L_t^b L_x^{4/3}}\\
&\quad +\|r^{-1}{\D}^{-1}(an\sn{\D}^{-1}F)\|_{L_t^b
L_x^2}+\|r^{-1/2}A\c
{\D}^{-1}F\|_{L_t^b L_x^{4/3}}\\
&\les \|A\|_{L_t^b L_x^2}\|\sn{\D}^{-1}F\|_{L_t^\infty L_x^4}\\
&\quad \, +\|{\D}^{-1}F\|_{L_t^\infty L_x^4} \times(\|A\c A\|_{L_t^b
L_x^2}+\|\|r^{-\f12}A\|_{L_t^b L_x^2}) \\
&\quad\, +\N_1(r^{-1/2}\D^{-1}F)+\Delta_0 \N_1(\D^{-1}F)\\
&\les{\N}_1(\sn{\D}^{-1} F)+{\N}_2({\D}^{-1}F)\les {\N}_1(F).
\end{align*}
The combination of (\ref{it0.24}), (\ref{it0.210+}) and (\ref{7.15})
gives
\begin{align*}
\|\sn[\D^{-1},\Dt]_g F\|_{\P^0}&\les\|[\Dt, \sn]_g\D^{-1}
F\|_{\P^0}+\Delta_0\|[\Dt,\D^{-1}]_g F\|_{L_t^b L_x^2}^q\c
\|[\Dt,\sn]_g\D^{-1}F\|_{L^2}^{1-q}\\&\les  \N_1(F).
\end{align*}
where in the above inequalities, $b>4$ and $\a_0<q<1$.
\end{proof}

\subsection{Proof of Proposition \ref{P7.1}: Part II}
\label{C1estimate}
We record the following result, which, for completeness, will be proved in Appendix by using Lemma \ref{sl1}.
\begin{lemma}\label{sl2}
For any smooth $S$-tangent tensor field $F$ and for exponent $2\le
q\le \infty$, we have the following inequality for $k>0$
\begin{align}
\|r^{-\frac{1}{2}-\frac{1}{q}}P_k F\|_{L_t^q L_x^2}&\les
2^{-\frac{1}{2}k-\frac{1}{q}k}\N_1(F),\label{TE}\\
\|r^{-\frac{1}{q}}  F_k\|_{L_t^q L_x^4}&\les 2^{-\frac{k}{q}}
\N_1(F)\label{TEE}
\end{align}
\end{lemma}

 We will complete the proof of  Proposition \ref{P7.1}
by studying error type terms in Proposition
\ref{errortype}

Let us first establish the following result with the help of Lemma \ref{sl2}.

\begin{proposition}\label{e2two}
Let ${\D}^{-1}$ denote either ${\D}_1^{-1}$ or
${}^\star{\D}_1^{-1}$. For any $S$-tangent tensor fields $F$ and $G$
on $\H$ there holds
\begin{equation*}
\|\t1a {\D}^{-1}(an F\c \sn G)\|_{L_t^b L_x^2}\les
{\N}_1(F){\N}_1(G),\,\quad  \mbox{with }\, 4<b<\infty.
\end{equation*}
\end{proposition}

\begin{proof}
Note that based on Lemma \ref{sl1} in Appendix, the inequalities \cite[Lemma 5.1,(6.25),(6.26)]{Qwang} still hold true.
We only need to modify the case $l<n<m$ of the proof in
\cite[P.310-311]{Qwang}.

When $l<n<m$,
\begin{equation*}
an F_n\c \sn G_m=\sn( anF_n\c G_m)-\sn(an) F_n\c G_m-an \sn F_n\c
G_m,
\end{equation*}
thus we need to consider the three terms
\begin{align*}
\I_{lnm}^1:&=\|\t1a P_l \D^{-1}(an \sn  F_n\c G_m)\|_{L_t^b
L_x^2}\\
\I_{lnm}^2:&=\|\t1a P_l \D^{-1}\sn(an F_n\c G_m)\|_{L_t^b
L_x^2}\\
\I_{lnm}^3:&=\|\t1a P_l \D^{-1}(\sn(an) F_n\c G_m)\|_{L_t^b L_x^2}.
\end{align*}
Using $C^{-1}<an<C$, following the same procedure in \cite{Qwang},
we can get
\begin{equation}\label{i12}
\sum_{0<l<n<m}(\I_{lnm}^1+\I_{lnm}^2)\les \N_1(F)\N_1(G).
\end{equation}
Now consider $\I_{lnm}^3$.

By Lemma \ref{dual} with $p=\frac{4}{3}$, Proposition \ref{P2} (i)
followed with (\ref{sob.m}), and  (\ref{TEE}),
\begin{align*}
\I_{lnm}^3&=\|\t1a P_l \D^{-1}(an (\zeta+\zb)F_n\c G_m)\|_{L_t^b
L_x^2(\H)}\\
&\les 2^{-\frac{l}{2}}\|r^{\f12-\frac{1}{b}}(\zeta+\zb)an  F_n\c
G_m\|_{L_t^b L_x^{4/3}}\\
&\les2^{-\frac{l}{2}} \|r^{\f12-\frac{1}{b}}F_n\c G_m\|_{L_t^b
L_x^2}\|A\|_{L_t^\infty L_x^4}\\
&\les 2^{-\frac{l}{2}}\|r^{-\frac{1}{b}+\f12}G_m\|_{L_t^b
L_x^4}\|F_n\|_{L_t^\infty L_x^4}\\
&\les 2^{-\frac{l}{2}-\frac{m}{b}}\N_1(G)\c\N_1(F),
\end{align*}
we obtain
\begin{equation*}
\sum_{0<l<n<m}\I_{lnm}^3\les \N_1(G)\N_1(F).
\end{equation*}
The proof is therefore complete.
\end{proof}

Let $\D$ be one of the operators $\D_1$,${}^\star \D_1$ or $\D_2$.
In the following result, we use  Proposition \ref{e2two} to estimate
the error type terms.

\begin{proposition}\label{errortype}
For $S$-tangent tensors $G$ on $\H$ verifying  ${\N}_1(G)<\infty$,
set
\begin{align*}
&{\E}_1(G):=r^{-1}{\D}^{-1}(anA\c G) \,\mbox{ or }\,  {\D}^{-1}(anA\c A\c G),\\
&{\E}_2(G):={\D}^{-1}(an\sn A\c G) \, \mbox{ or }\, \D^{-1}( an \sn
G\c A),
\end{align*}
 The following estimates hold
$$
\|\t1a  {\E}_1(G)\|_{L_t^b L_x^2}+\|\t1a {\E}_2(G)\|_{L_t^b
L_x^2}\les (\gc){\N}_1(G)
$$
where $4<b<\infty$.
\end{proposition}

\begin{proof}
Using Proposition \ref{P7.3} with $p=\frac{4}{3}$, (\ref{sob.m}) and
(\ref{As}), we get
\begin{align*}
\|\t1a {\D}^{-1}(anA\c A\c G)\|_{L_t^b L_x^2}
&\les\|r^{\f12-\frac{1}{b}}A\c A\c G\|_{L_t^b L_x^{4/3}}\les
\|A\|_{L_t^\infty L_x^4}^2 \|G\|_{L_t^b L_x^4}\\
&\les (\Delta_0^2+\RR)^2{\N}_1(G),
\end{align*}
by using Proposition \ref{P4.1}, (\ref{sob.in}), (\ref{As}) and
(\ref{sob.m}), we have
\begin{align*}
\|r^{-1-\frac{1}{b}}{\D}^{-1}(anA\c G)\|_{L_t^b L_x^2} &\les \|\t1a
A\c G\|_{L_t^b L_x^2} \les \|\t1a A\|_{L_t^b L_x^4}\|G\|_{L_t^\infty
L_x^4}\\
&\les (\gc){\N}_1(G).
\end{align*}
For ${\E}_2(G)$, we infer from Proposition \ref{e2two} and
(\ref{As}) that
$$
\|\t1a {\E}_2(G)\|_{L_t^b L_x^2}\les {\N}_1(G){\N}_1(A)\les
(\gc){\N}_1(G).
$$
We thus obtain the desired estimates.
\end{proof}

By analyzing the expression of $\b$ and $C_0(F):=[\Dt, \D^{-1}]F$,
we have

\begin{corollary}\label{betaes}
The following inequalities hold for any $S$-tangent tensor $F$,
\begin{align}
\|\t1a {\D}^{-1}(an\b\c F)\|_{L_t^b L_x^2}&\les
{\N}_1(F)(\gc)\label{betae1}\\
\|\t1a C_0(F)\|_{L_t^b L_x^2}&\les
{\N}_1(r^{-\frac{1}{2}}{\D}^{-1}F)\label{come1}
\end{align}
where $4<b<\infty.$
\end{corollary}

\begin{proof}
Using Codazzi equation (\ref{0.7}), i.e.
$$
an\b=an \sn A+an(A\c A+r^{-1}A),
$$
we infer
$$
{\D}^{-1}(an\b\c F)={\E}_1(F)+{\E}_2(F).
$$
Whence (\ref{betae1}) follows from Proposition \ref{errortype}.

Similarly, using (\ref{symbol:comm1}) we can write
\begin{equation*}
C_0(F)={\E}_1({\D}^{-1}F)+{\E}_2({\D}^{-1}F)+r^{-1}{\D}^{-1}(an\sn{\D}^{-1}F).
\end{equation*}
For the last term, using (\ref{A92}) with $H=F$,  we infer
$$
\|r^{-1-\frac{1}{b}}{\D}^{-1}(an\sn{\D}^{-1}F)\|_{L_t^b
L_x^2}\les{\N}_1(r^{-\frac{1}{2}}{\D}^{-1}F)(\gc+1).
$$
The desired estimate then follows from Proposition \ref{errortype}.

\end{proof}

\begin{proof}[Proof of (\ref{str:C1}) in Proposition \ref{P7.1}]
Combining Proposition \ref{P7.3}, (\ref{come1}) and (\ref{7.24}) we
derive for $\D^{-1}$ either $\D_2^{-1}$ or $\D_1^{-1}$, $4<b<\infty$
\begin{align}
\|\t1a {\D}^{-1}C_0(\check R)\|_{L_t^b L_x^2}&\les
\|r^{1-\frac{1}{b}} C_0(\check R)\|_{L_t^b L_x^2}\les
{\N}_1(r^{\frac{1}{2}}{\D}^{-1}\check R)\les
\Delta_0^2+\R_0\label{main:er}.
\end{align}
 Observe that $C_1(\check R)$ can be
written symbolically in the form
$$
C_1(\check R)=\sn{\D}^{-1} C_0(\check R).
$$
By Hodge-elliptic estimate (\ref{7.15}) with $\f12<q<1$ and
$4<b<\infty$, (\ref{7.21}), (\ref{7.23}) and (\ref{main:er})
\begin{align*}
\|C_1(\check R)\|_{{\mathcal P}^0} &\les \|C_0(\check
R)\|_{{\mathcal P}^0}+\Delta_0\|{\mathcal D}^{-1}C_0(\check
R)\|^{q}_{L_t^b L_x^2}\|C_0(\check R)\|_{L^2(\H)}^{1-q}\\
&\les \Delta_0^2+{\R}_0+\Delta_0\|{\D}^{-1}C_0(\check R)\|_{L_t^b
L_x^2}^{q}(\Delta_0^2+{\R}_0)^{1-q}\les \gc.
\end{align*}
This is the desired estimate.
\end{proof}

\subsection{A preliminary estimate for $(\rho,\sig)$}

Define
$$
\bar\rho=\frac{1}{|S_{t}|}\int_{S_{t}} \rho d\mu_\ga \,\, \mbox{ and
} \quad  \bar\sig=\frac{1}{|S_{t}|}\int_{S_{t}}\sigma d\mu_\ga.
$$
We have

\begin{lemma}
\begin{equation}\label{avr1}
|r^{\frac{3}{2}} \bar \rho|+|r^{\frac{3}{2}}\bar \s|\les
\Delta_0^2+\RR,
\end{equation}
\begin{equation}\label{avr}
|r^{\frac{3}{2}}\rhocb|+|r^{\frac{3}{2}}\sigcb|\les \Delta_0^2+\RR.
\end{equation}
\end{lemma}

\begin{proof}
By \cite[Eq. (41)]{KR1}, i.e.
\begin{equation}\label{lrhos}
\frac{d}{ds}\rho+\frac{3}{2}\tr\chi\rho=F
\end{equation}
where $F=\div \b-\f12\chibh\c \a+(\zeta+2\zb)\c \b$, the transport
equation for $\rhob$ can be obtained as follows
\begin{align*}
\frac{d}{ds} (\rhob)&=\sn_L(r^{-2}\int_{S} \rho)\nn\\
&=-\ovl{an\tr\chi}(an)^{-1}\rhob+r^{-2}(an)^{-1}\int_{S} \left(
\sn_L \rho +\tr\chi\rho \right) an d\mu_\ga\nn\\
&=-(an)^{-1}\ovl{an\tr\chi}\rhob+(an)^{-1}\ovl{\left(-\frac{1}{2}an\tr\chi
\rho +anF\right)},
\end{align*}
and
\begin{align*}
\frac{d}{ds}(r^3 \rhob)&=\frac{3
r^3}{2}\ovl{an\tr\chi}(an)^{-1}\rhob\\
&\quad \,+r^3\left\{-(an)^{-1}\ovl{an\tr\chi}\rhob
+(an)^{-1}\ovl{\left(-\frac{1}{2}an\tr\chi
\rho +anF\right)}\right\}\\
&=-\f12 r^3 (an)^{-1}\ovl{an
\tr\chi(\rho-\rhob)}+r^3(an)^{-1}\ovl{an F}.
\end{align*}
Integrating the above identity in $t$, due to $\lim_{t\rightarrow 0}
r^\theta\rho=0$ for any $\theta>0$,  we can obtain
\begin{align*}
r^{\frac{3}{2}}\rhob(t)= r^{-\frac{3}{2}}\int_0^t
r^3\left(-\f12\ovl{an\tr\chi\c \os{\rho}}+\ovl{anF} \right) dt'.
\end{align*}
In view of (\ref{comp1}), (\ref{cond1}) and Proposition \ref{cmps},
we obtain
\begin{equation*}
r^{-\frac{3}{2}}\int_0^t \left|r^3\ovl{an\tr\chi\c
\os{\rho}}\right|\le \|r\os{\rho}\|_{L_t^2
L_\omega^2}\|r'\|_{L^2(0,t]}r^{-\frac{3}{2}}\les
\|\rho\|_{L^2(\H)}\les\RR.
\end{equation*}
By integration by part on $S=S_t$, we can obtain
\begin{align*}
r^2\ovl{an F}&=\int_{S}\left( an(\div
\b-\f12\chibh\c\a)+an(\zeta+2\zb)\c\b \right) d\mu_{\ga}\\
&=\int_{S}an(\zb\b-\f12\chibh\c \a) d\mu_\ga =\int_{S} an\Ab\c R_0
d\mu_\ga.
\end{align*}
Hence by (\ref{cond1}) and Proposition \ref{propab},
\begin{align*}
\left|r^{-\frac{3}{2}}\int_0^t {r'}^3 \ovl{anF} d\tt\right| &\le
r^{-\frac{3}{2}}\|r'\|_{L^2(0,t]}\|(r')^2 \Ab\c R_0\|_{L_t^2
L_\omega^1} \\
& \le \|R_0\|_{L^2}\|r\Ab\|_{L_t^\infty L_\omega^2} \\
&\les\gc.
\end{align*}
Following the same procedure as above, we can obtain the same
estimate for  $\bar\sigma$ in view of \cite[Eq. (42)]{KR1}.

Note that
\begin{align}
&|r(\ovl{\chih\c \chibh},  \ovl{\chih\wedge\chibh})|\les \|r^\f12
\Ab\|_{L_t^\infty L_\omega^2}^2\les (\gc)^2\label{chchi}
\end{align}
(\ref{avr}) follows by connecting (\ref{avr1}) with (\ref{chchi}).

\end{proof}

\begin{proposition}\label{7.17}
For $4<b<\infty$, there hold
\begin{align}
\|r^{-\frac{1}{b}-\frac{1}{2}}\D_1^{-1}(an(\ckk\rho,
-\ckk\sig))\|_{L_t^b L_x^2} &\les \Delta_0^2+\RR, \label{7.16}\\
\|r^{-\frac{1}{b}-\frac{1}{2}}\D_1^{-1}(an (\rho, -\sigma))\|_{L_t^b
L_x^2} &\les \gc. \label{7.1.5}
\end{align}
\end{proposition}

\begin{proof}
 Let $H=(\rhoc-\rhocb, -\sigc+\sigcb)$. In view of $
{\D}_1\D_1^{-1}H=H$,
\begin{equation*}
\D_1^{-1}(an(\rhoc, -\sigc))=\D_1^{-1}(an \D_1
\D_1^{-1}H)+\D_1^{-1}(an(\rhocb, -\sigcb)).
\end{equation*}
By Proposition \ref{P7.3},
\begin{equation*}
\|r^{-\frac{1}{b}-\frac{1}{2}}\D_1^{-1}(an(\rhocb,
-\sigcb))\|_{L_t^b L_x^2}\les \|r^{\frac{3}{2}-\frac{1}{b}} an
(\rhocb, -\sigcb)\|_{L_t^b L_\omega^2}.
\end{equation*}
Hence, in view of (\ref{avr}) and (\ref{cond1})
\begin{align}\label{7.20}
\|r^{\frac{3}{2}-\frac{1}{b}}(an(\rhocb, -\sigcb))\|_{L_t^b
L_\omega^2} &\les\|r^{\frac{3}{2}}an  (\rhocb,
-\sigcb)\|_{L_t^\infty L_\omega^\infty}^{1-\frac{2}{b}}\|r an
(\rhocb, -\sigcb)\|_{L_t^2
L_\omega^2}^{\frac{2}{b}}\nn\\
&\les \gc.
\end{align}
By Leibnitz rule, Proposition \ref{P7.3} and (\ref{sob.in}), we have
\begin{align*}
\|r^{-\frac{1}{b}-\frac{1}{2}}\D_1^{-1}(an\D_1\D_1^{-1} H)\|_{L_t^b
L_x^2}&\les\|r^{-\frac{1}{b}-\frac{1}{2}} \D_1^{-1}\D_1(an\D_1^{-1}
H)\|_{L_t^b L_x^2}\\
&\quad \, +\|r^{-\frac{1}{b}-\frac{1}{2}}\D_1^{-1}( \sn (an) \D_1^{-1}H)\|_{L_t^b L_x^2}\\
&\les \|r^{-\frac{1}{b}-\frac{1}{2}} \D_1^{-1}H\|_{L_t^b
L_x^2}+\|r^{-\frac{1}{b}}(\zeta+\zb)\D_1^{-1}H\|_{L_t^b L_x^{4/3}}\\
&\les \N_1(\D_1^{-1}H)+\|r^{-\frac{1}{b}} \D_1^{-1}H\|_{L_t^b L_x^2}\|\zeta+\zb\|_{L_t^\infty L_x^4}\\
&\les \N_1(\D_1^{-1}H)(1+\gc)\les \Delta_0^2+\RR
\end{align*}
For the last two inequalities, we employed  (\ref{As}) and
(\ref{7.24}).

(\ref{7.1.5}) can be obtained by combining (\ref{7.16}) with the
estimate for $r^{-1} \D^{-1}(an A\c \Ab)$ in view of  the estimate
for $\E_1(\Ab)$ in Proposition \ref{errortype}.
\end{proof}

\subsection{$L_t^b L_x^2$ estimates for ${\D}^{-1}E_1^G$}\label{dc0estimates}

For arbitrary $S$-tangent tensor field $F$, we denote by $E_1^G$
either $[\Dt, {\D}_1^{-1}](\check \rho, -\check \s)\c F$ or $Err\c
F$. In what follows, we establish estimates for
$\|{\D}^{-1}E_1^G\|_{L_t^b L_x^2}$ with $4<b<\infty$, which will be
employed for the Hodge-elliptic ${\P}^0$ estimates of error terms
arising in the decomposition procedure in Section
\ref{C23estimates}.

\begin{proposition}\label{er0.35}
Denote by ${\D}$ either ${\D}_1$ or ${\D}_2$, for appropriate
$S$-tangent tensor fields $F$, the following estimates hold
\begin{equation}
\|\t1a \D^{-1} E_1^G\|_{L_t^b L_x^2}\les (\gc) \N_1(F).
\end{equation}
More precisely,
\begin{align}
\|\t1a {\D}^{-1}(Err\c F)\|_{L_t^b L_x^2}&\les(\gc)
{\N}_1(F)\label{it0.13}\\
\|\t1a {\D}^{-1}([\Dt, {\D}_1^{-1}](\check \rho, -\check \s)\c
F)\|_{L_t^b L_x^2}&\les (\gc){\N}_1(F).\label{it0.12}
\end{align}
where $Err$ is defined in (\ref{errp1}) and $4<b<\infty.$
\end{proposition}

In order to prove Proposition \ref{er0.35}, we may use the error
type terms introduced in Proposition \ref{errortype} to rewrite
(\ref{errp1}) in view of (\ref{Bia2}) and (\ref{Bia3}) as
\begin{equation}\label{esim2}
Err={\D}_1^{-1}(an\tr\chi(\rhoc,-\sigc))+{\E}_1(\Ab)+{\E}_2(\Ab).
\end{equation}

\begin{proof}[Proof of Proposition \ref{er0.35}]
(\ref{it0.12}) can be obtained by using Proposition \ref{P7.3},
(\ref{come1}), (\ref{7.24}) and (\ref{sob.m}) as follows,
\begin{align*}
\|\t1a {\D}^{-1}(C_0(\check R)\c F)\|_{L_t^b L_x^2}&\les
\|r^{\frac{1}{2}-\frac{1}{b}}C_0(\check R)\|_{L_t^b
L_x^2}\|F\|_{L_t^\infty L_x^4}\\
&\les (\gc){\N}_1(F),
\end{align*}
and similarly, by using Proposition \ref{errortype}, for $i=1,2$
$$\|\t1a {\D}^{-1}({\E}_i(\Ab)\c F)\|_{L_t^b L_x^2}\les\|r^{\f12-\frac{1}{b}} \E_i(\Ab)\|_{L_t^b L_x^2}\|F\|_{L_t^\infty L_x^4}\les
(\gc){\N}_1(F).
$$
Thus, in order to prove (\ref{it0.13}), in view of (\ref{esim2}), it
remains to show
\begin{equation*}
\|\t1a{\D}^{-1}\big({\D}_1^{-1}(an\tr\chi(\rhoc,-\sigc))\c
F\big)\|_{L_t^b L_x^2}\les {\N}_1(F)\Delta_0.
\end{equation*}
For this estimate, we proceed as follows.
 Let $H=(\rhoc-\rhocb,
-\sigc+\sigcb)$, then $H=\D_1\D_1^{-1}H$.
\begin{align*}
\|\t1a{\D}^{-1}\big({\D}_1^{-1}(an\tr\chi(\rhoc,-\sigc))\c
F\big)\|_{L_t^b L_x^2} &\les \|\t1a\D^{-1}(\D_1^{-1}(an\tr\chi \D_1
\D_1^{-1}H) F)\|_{L_t^b L_x^2} \\
&\quad\, +\|\t1a \D^{-1} \big(\D_1^{-1}(an\tr\chi(\rhocb,-\sigcb))\c
F)\|_{L_t^b L_x^2}
\end{align*}
By $I_1$ and $I_2$, we denote the two terms on the right of the
inequality. Using Proposition \ref{P7.3}, (\ref{sob.m}),
(\ref{sob.in}) and (\ref{7.24}),
\begin{align*}
I_1&\les \|\t1a\D^{-1}(an\tr\chi \D^{-1}H\c F)\|_{L_t^b
L_x^2}+\|\t1a\D^{-1}(\D_1^{-1}(\sn (an\tr\chi)
\D_1^{-1} H)\c F)\|_{L_t^b L_x^2}\\
&\les \|\t1a\D^{-1}H\c F\|_{L_t^b
L_x^2}+\|r^{-\frac{1}{b}+\frac{1}{2}}\D_1^{-1}(\sn(an\tr\chi)
\D^{-1}H)\|_{L_t^b L_x^2}\|F\|_{L_t^\infty L_x^4}\\
&\les \|F\|_{L_t^\infty L_x^4}\|r^{-\frac{1}{b}}\D^{-1}H\|_{L_t^b
L_x^4} +\|r^{1-\frac{1}{b}}\sn(an\tr\chi)\|_{L_t^b
L_x^2}\|\D^{-1}H\|_{L_t^\infty L_x^4}\|F\|_{L_t^\infty L_x^4}\\
&\les \N_1(F)\N_1(\D^{-1} H)\les (\Delta_0^2+\RR)\N_1(F)
\end{align*}
where we  employed
\begin{equation*}
\|r^{\f12-\frac{1}{b}}\sn (an\tr\chi)\|_{L_t^b L_x^2}\les \|r^\f12
\sn (an \tr\chi)\|_{L_t^\infty
L_x^2}^{1-\frac{2}{b}}\|\sn(an\tr\chi)\|^{\frac{2}{b}}_{L^2}\les\gc.
\end{equation*}
By Proposition \ref{P7.3}, (\ref{7.20}) and (\ref{sob.m}),
\begin{align*}
I_2&\les \|r^{-\frac{1}{b}+\f12}\D^{-1}(an\tr\chi(\rhocb,
-\sigcb))\|_{L_t^b L_x^2}\|F\|_{L_t^\infty L_x^4}
\\&\les\|r^{\frac{3}{2}-\frac{1}{b}}an\tr\chi (\rhocb,
-\sigcb)\|_{L_t^b L_x^2}\|F\|_{L_t^\infty L_x^4}\\
&\les \|r^{\frac{3}{2}-\frac{1}{b}}(\rhocb, -\sigcb)\|_{L_t^b
L_\omega^2}\N_1(F)\\
&\les (\Delta_0^2+\R_0)\N_1(F).
\end{align*}
The proof is complete.
\end{proof}

\subsection{$L_t^b L_x^2$ estimates for $\sn_L {\D}^{-1}\F$}\label{l2estimates}

We will establish the following
\begin{proposition}\label{it0.20}
Denote by $\D^{-1}\F$ either $\D^{-2}\ckk R$ or $
\D_1^{-1}(a\delta+2a\lambda)$. There holds
\begin{equation*}
\|r^{-\frac{1}{b}} \Dt {\D}^{-1}\F\|_{L_t^b L_x^2}\les {\mathcal
R}_0+\Delta_0^2, \quad\,4<b<\infty.
\end{equation*}
\end{proposition}

{\bf Case 1:
  $\F=\D^{-1} \ckk R$.}

We denote by ${\D}^{-1}\ff$ either ${\D}_2^{-1} Err$ or
${}^\star{\D}_1^{-1}\widetilde{Err}$. To prove Proposition
\ref{it0.20}, we will rely on (\ref{it0.31}) in the following
result.
\begin{proposition}\label{aux3}
For $\ff=(Err, \widetilde{Err})$ with $Err$ and $\widetilde{Err}$
given by (\ref{errp1}), there hold
\begin{equation}\label{it0.31}
\|r^{-\frac{1}{b}} {\D}^{-1}\ff\|_{L_t^b L_x^2}\les
\Delta_0^2+\R_0,\quad\,4<b<\infty,
\end{equation}
\begin{equation}\label{itF1}
\|\sn{\D}^{-1}\ff\|_{{\mathcal P}^0}\les \Delta_0^2+{\R}_0.
\end{equation}
\end{proposition}
Assuming (\ref{it0.31}), now we prove Proposition \ref{it0.20}.

\begin{proof}[Proof of Proposition \ref{it0.20} for {\it Case 1}]
In view of the formula
\begin{equation*}
\Dt {\D}^{-2}\check R=[\Dt, {\D}^{-1}]{\D}^{-1}\check
R+{\D}^{-1}[\Dt, {\D}^{-1}]\check R+{\D}^{-2}\Dt \check R,
\end{equation*}
we only need to show for $4<b<\infty$, there hold
\begin{align}
\|r^{-\frac{1}{b}} [\Dt, {\D}^{-1}]{\D}^{-1}\check R\|_{L_t^b
L_x^2}&\les
\Delta_0^2+{\R}_0\label{lf1}\\
\|r^{-\frac{1}{b}}{\D}^{-1}[\Dt, {\D}^{-1}]\check R\|_{L_t^b
L_x^2}&\les
\Delta_0^2+{\R}_0\label{lf2}\\
\|r^{-\frac{1}{b}} {\D}^{-2}\Dt \check R\|_{L_t^b L_x^2}&\les
\Delta_0^2+{\R}_0\label{lf3}.
\end{align}
(\ref{lf1}) follows from (\ref{come1}) with $F={\D}^{-1}\check R$,
by using the fact that ${\N}_1(r^{-\frac{1}{2}}{\D}^{-2}\check
R)\les{\N}_2({\D}^{-2}\check R)\les \Delta_0^2+\R_0$. (\ref{lf2})
was proved in  (\ref{main:er}).

It only remains to prove (\ref{lf3}). Consider first the case
${\D}^{-2}\Dt \check R={\D}_2^{-1}{\D}_1^{-1}\Dt(\check\rho, -\check
\sig)$. By
 $an\b=\sn(an
A)+an(A\c A+r^{-1}A)$ , (\ref{sob.m}) and (\ref{As}),
\begin{align*}
\|\t1a \D^{-1}(an\b)\|_{L_t^b L_x^2}&\les \|\t1a{\D}^{-1}(\sn(an
A)+an(A\c
A+r^{-1}A))\|_{L_t^b L_x^2}\\
&\les \|\t1a A\|_{L_t^b L_x^2}+\|r^{-\frac{1}{b}+1} A\c A\|_{L_t^b
L_x^2}\\&\les \N_1(A)+\N_1(A)^2\les \Delta_0^2+\RR.
\end{align*}
Then by (\ref{it0.31}), we obtain
\begin{equation*}
\|\t1a {\D}^{-1}(Err+an\b)\|_{L_t^b L_x^2}\les \Delta_0^2+{\R}_0.
\end{equation*}
In view of the definition of  $Err$ in (\ref{errp1}), we have
\begin{equation}\label{lin3}
\|\t1a\D^{-2}\Dt (\check \rho, -\check \sig)\|_{L_t^b L_x^2}\les
\gc.
\end{equation}
 Using
$ {}^\star{\D}_1^{-1}\Dt {\udb}=an(\rho, \sig)+\widetilde{Err}, $
Proposition \ref{7.17} and (\ref{it0.31}),
 \begin{align}\label{lin4}
 \|\t1a\D^{-2} \Dt \udb\|_{L_t^b L_x^2} &\les
\|\t1a {\D}_1^{-1}(an(\rho, \s))\|_{L_t^b L_x^2}+\|\t1a {\D}^{-1}
\ff\|_{L_t^b L_x^2} \nn \\&\les\gc
\end{align}
In view of (\ref{lin3}) and (\ref{lin4}), (\ref{lf3}) is proved.

\end{proof}

To prove Proposition \ref{aux3}, we will rely on the following
result.

\begin{lemma}
Let $\D^{-1}$ denote one of the operators $\D_1^{-1}$, $\D_2^{-1}$
or ${}^\star \D_1^{-1}$. For any appropriate $S$-tangent tensor
field $G$, there holds
\begin{equation}\label{botoap}
\|\D^{-1}(an\rhoc\c G)\|_{L_t^b L_x^2}\les \|\La^{-\a_0}
\rhoc\|_{L_t^\infty L_x^2}\N_1(G)
\end{equation}
where $\a_0\ge\f12$ and $4<b<\infty$.
\end{lemma}

\begin{proof}
We can adapt the proof for \cite[Lemma 4.4]{Qwang} in view of $\|\sn
(an)\|_{L_x^4 L_t^\infty }\les \gc$ and $an<C$, to derive the
following estimates for $S$ tangent tensor fields $F$,
\begin{align}
&\|\Lambda^{-\a}(an \rhoc \c F_m)\|_{L^2(S)} \les
\|\Lambda^{-\a_0}\rhoc\|_{L_t^\infty L_x^2}2^m r^{-1} \|P_m
F\|_{L^2(S)}, \label{om2}\\
&\|P_m(an \rhoc\c  \D^{-1}P_l F)\|_{L^2(S)}\les
\|\La^{-\a_0}\rhoc\|_{L_t^\infty L_x^2} 2^{\a m } r^{-\a_0}\|P_l
F\|_{L^2(S)}\label{om3},
\end{align}
where $\a>\a_0\ge1/2$ and $\D^{-1}$ denotes either $\D_1^{-1}$ or
${}^\star\D_1^{-1}$.

Set $\Omega_{nl}:=\D^{-1}P_l^2(an \rhoc\c P_n^2 G),$ with $l, n\in
\Bbb N$.  We now prove
\begin{equation}\label{omeg2}
\sum_{l,n>0}\|\Omega_{nl}\|_{L_t^b L_x^2}\les
\|\Lambda^{-\a_0}\rhoc\|_{L_t^\infty L_x^2}\N_1(G),
\end{equation}
 and lower frequency terms can be treated similarly.
By duality argument, Proposition \ref{lambda:1} (iii) and Lemma
\ref{dual},
\begin{equation}\label{om4}
\|\D^{-1} \Lambda^{\a} P_l F\|_{L^2(S)}\les 2^{(-1+\a)l}r^{1-\a}
\|F\|_{L^2(S)}.
\end{equation}
We first prove (\ref{omeg2}) for the case $0<n<l$. With the help of
(\ref{om4}) and (\ref{om2}),\begin{align*}
\|\Omega_{nl}\|_{L_x^2}&\les 2^{-(1-\a)l} 2^n
r^{-\a}\|\Lambda^{-\a_0}\rhoc\|_{L_t^\infty L_x^2}\|P_n G\|_{L_x^2}.
\end{align*}
Take $L_t^b$ norm for $4<b<\infty$, and (\ref{TE}) in Lemma
\ref{sl2}
\begin{align*}
\|\Omega_{nl}\|_{L_t^b L_x^2}&\les 2^{-(1-\a)l}
2^{n(\frac{1}{2}-\frac{1}{b})}
r^{\frac{1}{2}+\frac{1}{b}-\a}\|\Lambda^{-\a_0}\rhoc\|_{L_t^\infty
L_x^2}\N_1(G).
\end{align*}
Since we can choose $\a_0<\a<\frac{1}{2}+\frac{1}{b}$, we deduce
\begin{equation*}
\sum_{0<n<l} \|\Omega_{nl}\|_{L_t^b L_x^2}\les
\|\Lambda^{-\a_0}\rhoc\|_{L_t^\infty L_x^2}\N_1(G).
\end{equation*}

For the case $0<l<n$, let us pair $\Omega_{nl}$ with any $S$ tangent
tensor $F$ with $\|F\|_{L^2(S)}\le 1$. By (\ref{om3}),
\begin{align*}
\l\Omega_{nl}, F\r&=\l P_l(an \rhoc P_n^2 G), P_l {}^\star\D^{-1}F\r\\
&\les 2^{\a n} r^{-\a_0}\|\La^{-\a_0}\rhoc\|_{L_t^\infty L_x^2}\|P_n
G\|_{L_x^2}.
\end{align*}
Hence, by (\ref{TE}) in Lemma \ref{sl2},
\begin{align*}
\|\Omega_{nl}\|_{L_t^b L_x^2}\les 2^{\a
n-(\frac{1}{2}+\frac{1}{b})n}r^{-\a_0+\frac{1}{2}+\frac{1}{b}}\|\La^{-\a_0}\rhoc\|_{L_t^\infty
L_x^2}\N_1(G).
\end{align*}
Consequently,
\begin{equation*}
\sum_{0<l<n}\|\Omega_{nl}\|_{L_t^b L_x^2}\les
\|\La^{-\a_0}\rhoc\|_{L_t^\infty L_x^2}\N_1(G).
\end{equation*}
(\ref{botoap}) follows.
\end{proof}

We are ready to prove Proposition \ref{aux3}.

\begin{proof}[Proof of Proposition \ref{aux3}]
(\ref{itF1}) can be derived by using (\ref{it0.31}), Theorem
\ref{in} and (\ref{errp}).

Now we consider (\ref{it0.31}). By letting $F=1$ in (\ref{it0.13}),
we obtain for $4<b<\infty$ that
$$ \|\t1a{\D}_2^{-1}Err\|_{L_t^b L_x^2}\les \Delta_0^2+\R_0.$$ Thus
we only need to consider ${\D}_1^{-1}\widetilde{Err}$.

By definition of $\wt{Err}$  in (\ref{errp1}), in view of
(\ref{0.7}), we rewrite $\wt{Err}$ symbolically as follows
\begin{equation}
\wt{Err}={}^\star{\D}_1^{-1}(an \tr\chi\udb+an\Ab\c(\sn A+A\cdot
A+r^{-1}A))+{}^\star{\D}_1^{-1}(an(\zeta\c\rho-\zeta{}^\star
\s)).\label{errt}
\end{equation}

By Propositions \ref{P7.3} and \ref{errortype},
\begin{equation*}
\|\t1a \D^{-2}(an \Ab\c (\sn A+r^{-1}A+A\c A))\|_{L_t^b
L_x^2}\les\gc.
\end{equation*}
According to (\ref{errt}), we consider $\W=\|\t1a\D^{-2}(an\tr\chi
\udb)\|_{L_t^b L_x^2}$, and
$$
\U=\|\t1a {\D}^{-2}(an\zb\c\check \rho)\|_{L_t^b L_x^2},\,\quad
\V=\|\t1a {\D}^{-2}(an\zb{}^\star \check \sig)\|_{L_t^b L_x^2}.
$$

By $\udb={}^\star\D_1{}^\star\D_1^{-1}\udb$, also using Propositions
\ref{P4.1} and \ref{P7.3}, (\ref{sob.in}), (\ref{As}) and
(\ref{7.24})
\begin{align*}
\W&\les \|\t1a\D^{-2} \left(({}^\star\D_1(an\tr\chi
{}^\star\D_1^{-1}\udb))-\sn(an\tr\chi){}^\star\D_1^{-1}\udb\right)\|_{L_t^b
L_x^2}\\
&\les \|r^{1-\frac{1}{b}} (an\tr\chi {}^\star\D_1^{-1}\udb)\|_{L_t^b
L_x^2}+\|r^{\frac{3}{2}-\frac{1}{b}}
\sn(an\tr\chi) {}^\star{\D}_1^{-1}\udb\|_{L_t^{b} L_x^{4/3}}\\
&\les\|r^{-\frac{1}{b}} {}^\star\D_1^{-1} \udb\|_{L_t^b
L_x^2}+\|r^{1-\frac{1}{b}}{}^\star\D_1^{-1}
\udb\|_{L_t^{b} L_x^4}\|r^\f12\sn (an\tr\chi )\|_{L_t^\infty L_x^2}\\
 &\les\N_1(r^{1/2} {}^\star\D_1^{-1}\udb)(\|r^{\f12}\sn(an\tr\chi)\|_{L_t^\infty
L_x^2}+1)\les\Delta_0^2+\RR.
\end{align*}
By (\ref{0.8}), clearly $\check \sig=curl \zeta$. Thus by
Propositions \ref{P7.3} and \ref{errortype}, we obtain
$$ \V=\|\t1a {\D}^{-1}{\E}_2(\zeta)\|_{L_t^b L_x^2}\les
{\N}_1(\zeta)\N_1(\zb)\les (\gc)^2. $$

 By Proposition
\ref{P4.1} and (\ref{botoap}), we have
\begin{align}
\U&\les \|\t1a {\D}^{-2}(an\rhoc\c \zb)\|_{L_t^b L_x^2}\les
\|r^{-\frac{1}{b}+1}\D^{-1}(an \rhoc\c\zb)\|_{L_t^b L_x^2}\nn\\&\les
\N_1(\zb)\|\La^{-\a_0}\rhoc\|_{L_t^\infty L_x^2}\les (\gc)^2\nn
\end{align}
where  we employed (\ref{smry1}) and (\ref{lrhoc}) to obtain the
last inequality.
\end{proof}

{\bf Case 2: $\F=(a\delta+2a\lambda)$.} We first give a slightly
stronger result than Proposition \ref{it0.20} for Case 2.
\begin{proposition}
There holds for $4<b<\infty$,
\begin{equation}\label{case2}
\|r^{-\f12-\frac{1}{b}}\Dt \D_1^{-1}(a\delta+2a\lambda)\|_{L_t^b
L_x^2}\les \Delta_0^2+\RR.
\end{equation}
\end{proposition}
\begin{proof}
It is easy to see
\begin{align}
\|r^{-\f12-\frac{1}{b}}\Dt \D_1^{-1}(a\delta+2a\lambda)\|_{L_t^b
L_x^2}&\les \|r^{-\f12-\frac{1}{b}} \D_1^{-1} \Dt(a\delta+2a
\lambda)\|_{L_t^b L_x^2}\nn\\&+\|r^{-\f12-\frac{1}{b}}[\Dt,
\D_1^{-1}](a\delta+2a\lambda)\|_{L_t^b L_x^2}.\label{divdel}
\end{align}
By (\ref{ldelta}) and (\ref{0.8}),
\begin{equation}\label{nabzb}
an\zb=\D_1^{-1}\Dt(
a\delta+2a\lambda)-\D_1^{-1}(an(\ckk\rho,\ckk\sigma))+\D_1^{-1}{\er_1}
\end{equation}
where, symbolically, $\er_1:=an(a\sl{\pi} \tr\chi+A\c A)$.
\begin{align}
\|r^{-\f12-\frac{1}{b}}&\D_1^{-1}\Dt(a\delta+2a\lambda)\|_{L_t^b
L_x^2} \nn \\
&\les\|r^{-\f12-\frac{1}{b}}(an\zb)\|_{L_t^b
L_x^2}+\|r^{-\f12-\frac{1}{b}}\D_1^{-1}(an(\rhoc,
\sigc))\|_{L_t^b L_x^2}\label{l11}\\
&+\|r^{-\f12-\frac{1}{b}}\D_1^{-1}\er_1\|_{L_t^b L_x^2}\label{l12}.
\end{align}
By (\ref{sob.in}) and Proposition \ref{7.17},  the two terms on the
right of (\ref{l11}) can be bounded by $$\N_1(\zb)+
\Delta_0^2+\RR\les\gc.$$

For (\ref{l12}), by Proposition \ref{P4.1} and (\ref{As}), also in
view of (\ref{comp1}) and (\ref{sob.in}), we deduce
\begin{align}
\|r^{-\f12-\frac{1}{b}}\D_1^{-1}\er_1\|_{L_t^b
L_x^2}&\les\|r^{\f12-\frac{1}{b}}\er_1\|_{L_t^b
L_x^2}\nn\\&\les\|r^{\f12-\frac{1}{b}}a\sl{\pi}
\tr\chi\|_{L_t^b L_x^2}+\|r^{\f12-\frac{1}{b}}A\c A\|_{L_t^b L_x^2}\nn\\
&\les \|r^{-\f12-\frac{1}{b}} \sl{\pi}\|_{L_t^b L_x^2}
+\|A\c A\|_{L_t^\infty L_x^2}\nn\\
 &\les
\N_1(\sl{\pi})+\|A\|_{L_t^\infty L_x^4}^2\les
\Delta_0^2+\RR.\label{diner1}
\end{align}
Thus, we  proved
\begin{equation}\label{l13}
\|r^{-\f12-\frac{1}{b}} \D_1^{-1}\Dt(a\delta+2a\lambda)\|_{L_t^b
L_x^2}\les\Delta_0^2+\RR.
\end{equation}
Repeat the derivation for (\ref{come1}), also using Lemma \ref{A.8},
\begin{align*}
\|r^{-\frac{1}{2}-\frac{1}{b}}[\D_1^{-1},\Dt](a\delta+2a\lambda)\|_{L_t^b
L_x^2}&\les (\gc) \N_1(r^{-1} \D_1^{-1}(a\delta+2a\lambda))\\&\les
(\gc)\N_1(a\sl{\pi})\les(\gc)^2.
\end{align*}
Hence (\ref{case2}) is proved.
\end{proof}

\begin{lemma}\label{p0er1}
\begin{equation}\label{p0er2}
\|\er_1\|_{\P^0}+\|\sn \D^{-1} \er_1\|_{\P^0}\les\gc,
\end{equation}
\end{lemma}

\begin{proof}
By (\ref{cor3}), we have
\begin{equation}\label{e2}
\|an\c A\c A\|_{\P^0}\les\Delta_0\N_1(A)\les\gc
\end{equation}
Using (\ref{cor3}) and (\ref{A93}), we obtain
\begin{equation}\label{e3}
\|a^2n\tr\chi\sl{\pi}\|_{\P^0}\les \N_1(a\sl\pi)\les \Delta_0^2+\RR.
\end{equation}
Thus the first inequality of (\ref{p0er2}) is proved. The second one
can be proved  by combining  (\ref{e2}), (\ref{e3}), (\ref{diner1})
and Theorem \ref{in}.
\end{proof}

\subsection{Decomposition for commutators}\label{C23estimates}

In order to prove  Proposition \ref{de:er}, it remains to decompose
the ``bad" terms which have not been treated in Proposition
\ref{P7.1}, i.e
$$an\beta\cdot{\mathcal D}^{-1}\F, \, \sn{\mathcal D}^{-1}
(an\beta\cdot{\mathcal D}^{-1}\F),$$ with $\F$ either $\D^{-1}\ckk
R$ or $(a\delta+2a\lambda).$ In view of (\ref{Coro1}) and
Proposition \ref{it0.20}, the assumptions in the following theorem
are satisfied with $F=\D^{-1}\F$. Then the proof of Proposition
\ref{de:er} is complete by using the following

\begin{theorem}\label{it0.1}
Assume that $F$ is an $S$-tangent tensor field of appropriate order
on $\mathcal H$ verifying ${\mathcal N}_2(F)<\infty$ and $\|\t1a
\sn_L F\|_{L_t^b L_x^2}<\infty $ with $4<b<\infty$. Then we have
\begin{enumerate}

\item [(i)]There exists a 1-form $E_0$ such that\begin{footnote}{In Theorem \ref{it0.1}
and the following proofs, $\check R=(\check \rho, -\check \sigma)$
and $C_0(\check R)=[\Dt, {\D}_1^{-1}](\check \rho, -\check \sigma)$,
since the other case that $\ckk R=\udb$  will not come up
here.}\end{footnote}
\begin{equation}\label{7.30} an\beta=\Dt {\mathcal
D}^{-1}\check R+E_0 \quad \mbox{with}\quad \|E_0\|_{{\mathcal
P}^0}\lesssim \Delta_0^2+{\R}_0
\end{equation}

\item [(ii)] There exists a decomposition $an\beta\cdot F=\Dt
P+E$, where $P$ and $E$ are tensor fields of the same type as
$an\beta\cdot F$ with
 \begin{equation}\label{ini3}
 \lim_{t\rightarrow 0}\|P\|_{L_x^\infty}=0
 \end{equation}
 and the estimates
\begin{eqnarray}
&&{\mathcal N}_1(P)\les\Delta_0{\N}_2(F),\, \|E\|_{{\mathcal
P}^0}\lesssim \Delta_0\cdot \big({\mathcal N}_2(F)+\|\t1a \sn_L
 F\|_{L_t^b L_x^2}\big).\label{7.31}
\end{eqnarray}

 \item [(iii)]  There exist tensors $\bar P$ and $\bar
E$ verifying (\ref{7.31}) so that
\begin{equation}\label{it0.2}
\sn{\mathcal D}^{-1}(an\beta\cdot F)=\Dt \bar P+\bar E,
\end{equation}
where $\D$ denote either $\D_1$ or $\D_2$,
and \begin{equation}\label{ini4} \lim_{t\rightarrow 0}\|\bar
P\|_{L_x^\infty}<\infty
\end{equation}
\end{enumerate}
\end{theorem}

\begin{proof}
In view of (\ref{errp1}), we have
\begin{equation}\label{new1111}
an\beta=\Dt\D^{-1} \check R+C_0(\check R)+Err.
\end{equation}
This proves (i) by noting that $E_0:=Err+C_0(\check R)$ satisfies
$\|E_0\|_{{\P}^0}\les \Delta_0^2+{\R}_0$ in view of  (\ref{errp})
and (\ref{7.21}).

Now we prove (ii). We have from (\ref{new1111}) that
\begin{align*}
an\beta\cdot F&=(\Dt{\mathcal D}_1^{-1}\check R+Err+C_0(\check
R))\cdot F=\Dt({\mathcal D}_1^{-1}\check R\cdot F)+E_1^B+E_1^G,
\end{align*}
where
\begin{align*}
E_1^{B}:=-{\D}_1^{-1}\check R\c \Dt F \quad \mbox{and}\quad
E_1^{G}:=(Err+C_0(\check R))\cdot F.
\end{align*}
By (\ref{7.2}), (\ref{7.21}) and (\ref{errp})  we obtain
\begin{equation}\label{de:01}
\|E_1^G\|_{{\mathcal P}^0}\les {\mathcal
N}_2(F)\left(\|Err\|_{{\mathcal P}^0} +\|C_0(\check R)\|_{{\mathcal
P}^0}\right)\les (\gc){\mathcal N}_2(F).
\end{equation}
By (\ref{7.1}) and (\ref{7.24}) we have
\begin{align*}
\|E_1^{B}\|_{{\P}^0}& \lesssim {\mathcal N}_1({\mathcal
D}^{-1}\check R)(\|\t1a \sn_L F\|_{L_t^b
L_x^2}+\|r^{\frac{1}{2}}\sn\Dt F\|_{L_t^2 L_x^2})\nonumber\\
&\lesssim ({\mathcal R}_0+\Delta_0^2)({\mathcal N}_2(F)+\|\t1a \sn_L
F\|_{L_t^b L_x^2}).
\end{align*}
Now we set
\begin{align}
P_1:={\mathcal D}_1^{-1} \check R\cdot F \quad \mbox{and}\quad
E_1:=E_1^{B}+E_1^{G},\label{e1p1}
\end{align}
from the above estimates we have
\begin{align*}
\|E_1\|_{{\P}^0}&\les (\gc)({\mathcal N}_2(F)+\|\t1a\sn_L F\|_{L_t^b
L_x^2}).
\end{align*}

In order to estimate ${\N}_1(P_1)$, let us estimate $\|E_1\|_{L^2}$
first. By using H\"older's inequality and (\ref{sob.01}),  we can
obtain
\begin{align*}
\|E_1^B\|_{L^2}&=\|{\D}^{-1}\check R\c \sn_L F \|_{L^2}\les
\|{\D}^{-1}\check R\|_{L_t^\infty L_x^4}\|\sn_L F\|_{L_t^2
L_x^4}\\&\les{\N}_1({\D}^{-1} \check R)(\|\sn\Dt
F\|_{L^2}+\|r^{-\frac{1}{2}}\sn_L F\|_{L^2}),
\end{align*}
and by using $\|E_1^G\|_{L^2(\H)}\les \|E_1^G\|_{{\P}^0}$ and
(\ref{de:01}) we can obtain
$$
\|E_1^G\|_{L^2}\les (\gc){\N}_2(F).
$$
Therefore
\begin{equation}\label{e1l2}
\|E_1\|_{L^2}\les (\gc){\mathcal N}_2(F).
\end{equation}

Now we  show
\begin{equation}\label{er0.31}
 {\mathcal
N}_1(P_1)\lesssim {\mathcal N}_2(F)(\gc).
\end{equation}
With the help of $\Dt P_1= an\b\c F-E_1$ and (\ref{e1l2}) we can
estimate $\|\sn_L P_1\|_{L^2}$ as follows
\begin{align*}
\|\sn_L P_1\|_{L^2}&\lesssim \|\beta\cdot F\|_{L_t^2
L_x^2}+\|E_1\|_{L^2}\lesssim (\gc){\mathcal N}_2(F).
\end{align*}
Similar to \cite[Section 6.12]{KR1}, we get $ \|\nabla P_1\|_{L_t^2
L_x^2}\lesssim (\Delta_0^2+{\mathcal R}_0){\mathcal N}_2( F).$ We
complete the proof of (\ref{7.31}).

By (\ref{e1p1})
\begin{align}
\|P_1\|_{L_x^\infty}&\le\|F\|_{L_x^\infty}\|\D_1^{-1}\ckk
R\|_{L_x^\infty}\les r^{\f12}\|\D_1^{-1} \ckk
R\|_{L_x^\infty}\N_2(F)\label{p1}
\end{align}
Since $\N_2(F)<\infty$ and $\lim_{t\rightarrow 0}\|\D_1^{-1}\ckk
R\|_{L_x^\infty}<\infty$, (\ref{ini3}) follows by letting
$t\rightarrow 0$  in (\ref{p1}). Therefore (ii) is proved.

Finally we prove (iii) by using the iteration procedure in
\cite[Section 6.12]{KR1}.  (\ref{ini4}) will be proved in Section
\ref{ap}.
 Let $P_0:=\D F$, then we can apply (ii) to
construct recurrently two sequences of $S$-tangent tensor fields
$\{P_i\}$ and $\{E_i\}$ such that
\begin{equation}\label{30000}
an\beta\cdot {\mathcal D}^{-1}P_{i-1}=\Dt P_i+E_i
\end{equation}
where $\D^{-1}$ denote either $\D_1^{-1}$ or $\D_2^{-1}$ and
\begin{align}
{\mathcal N}_1(P_i)&\le C(\gc) {\mathcal N}_2({\D}^{-1}P_{i-1}),\label{n1it}\\
\|E_i\|_{{\mathcal P}^0}&\le C(\gc)\left({\mathcal
N}_2(\D^{-1}P_{i-1}) +\|\t1a \sn_L \D^{-1} P_{i-1}\|_{L_t^b
L_x^2}\right)\label{eit}.
\end{align}
Such $P_i$ and $E_i=E_i^B+E_i^G$ can be constructed as in the proof
of (ii). Then for $i=1,2\cdots$,
\begin{eqnarray}
P_i=\D_1^{-1} \ckk R\c \D^{-1}P_{i-1}, && P_0=\D F\label{pi}\\
E_i^{B}:=-{\D}_1^{-1}\check R\c \Dt \D^{-1} P_{i-1}, &&
E_i^{G}:=(Err+C_0(\check R))\cdot \D^{-1} P_{i-1}. \label{ei}
\end{eqnarray}
In particular, $P_1$ and $E_1$ have been given by (\ref{e1p1}).

 With the above definition of $P_k$ and $E_k$, using (\ref{30000})
$$
\sn {\D}^{-1}(an\b\c F)=\Dt \bar P_k +\sn {\mathcal D}^{-1}(\Dt
P_k)+\bar E_k,
$$
where
\begin{align}
\bar P_k &=\sn{\mathcal
D}^{-1}(P_1+...+P_{k-1})+P_2+...+P_k\label{barp}\\
\bar E_k&=[\sn{\mathcal D}^{-1},
\Dt]_{g}(P_1+...+P_{k-1})+\sn{\mathcal
D}^{-1}(E_1+...+E_{k})\nn\\
&\quad +E_2+...+E_k.\nn
\end{align}
By using Lemma \ref{A.8}, it is easy to see from (\ref{n1it}) that
\begin{equation}\label{new1000}
{\mathcal N}_1( P_k)\le (C(\gc))^k {\mathcal N}_2(F).
\end{equation}
Moreover we have

\begin{proposition}\label{iter}
For $\{P_k\}_{k=1}^\infty$ and $\{E_k\}_{k=1}^\infty$ there hold
\begin{equation}\label{lldp}
\|\t1a \sn_L {\D}^{-1} P_k\|_{L_t^b L_x^2}\les (\gc)({\mathcal
N}_2({\D}^{-1}P_{k-1})+\|\sn_L {\D}^{-1}P_{k-1}\|_{L_t^b L_x^2}),
\end{equation}
\begin{equation}\label{nabd:e}
\|\sn{\D}^{-1}E_k\|_{{\mathcal P}^0}\les
\|E_k\|_{\P^0}+(\gc)({\mathcal N}_2({\D}^{-1} P_{k-1})+\|\sn_L
{\D}^{-1}P_{k-1}\|_{L_t^b L_x^2}).
\end{equation}
\end{proposition}

We will prove this result at the end of this section. We observe
that  Lemma \ref{A.8}, (\ref{lldp}), (\ref{n1it}) and (\ref{eit})
clearly imply
\begin{equation}\label{new2000}
\|E_k\|_{{\mathcal P}^0}\le (C(\gc))^k\left({\mathcal
N}_2(F)+\|\t1a\sn_L F\|_{L_t^b L_x^2}\right).
\end{equation}

It follows from (\ref{new1000}), (\ref{nabd:e}), (\ref{new2000}) and
(\ref{A.9}) that
\begin{align*}
{\mathcal N}_1(\bar P_k -\bar P_j)\le{\mathcal N}_2(F)\sum_{j\le
m\le k-1}(C(\gc))^m \lesssim (C(\gc))^j{\mathcal N}_2(F),
\end{align*}
and
\begin{align*}
\|\bar E_k-\bar E_j \|_{{\mathcal P}^0} &\le\big( {\mathcal
N}_2(F)+\|\t1a \sn_L F\|_{L_t^b
L_x^2}\big)\sum_{j\le m\le k-1}(C(\gc))^m\\
&\les (C(\gc))^j\big({\mathcal N}_2(F)+\|\t1a \sn_L F\|_{L_t^b
L_x^2}\big).
\end{align*}
Therefore $\{\bar P_k\}$ forms a Cauchy sequence relative to the
norm ${\mathcal N}_1(\cdot )$, while $\{\bar E_k\}$ forms a Cauchy
sequence relative to the $\P^0$ norm. Denote by $\bar P$ and $\bar
E$ their corresponding limits, we have
\begin{eqnarray*}
{\mathcal N}_1(\bar P)\les (\gc){\mathcal N}_2(F) \quad
\mbox{and}\quad \|\bar E\|_{{\mathcal P}^0}\les (\gc) \big({\mathcal
N}_2(F)+\|\t1a \sn_L F\|_{L_t^b L_x^2}\big).
\end{eqnarray*}
 We also observe that for sufficiently small $\Delta_0$,
\begin{align*}
\|\sn {\D}^{-1}(an\b\c F)-\Dt \bar P_k -\bar
E_k\|_{L^2}=\|\sn{\mathcal D}^{-1}(\Dt P_k)\|_{L^2}\les{\mathcal
N}_1(P_k).
\end{align*}
Letting $k\rightarrow +\infty$, we get
$$
\|\sn {\D}^{-1}(an\b\c F)-\Dt \bar P-\bar E\|_{L_t^2 L_x^2}=0.
$$
Hence $\sn {\D}^{-1}(an\b\c F)=\Dt \bar P+\bar E$. This completes
the proof of (\ref{it0.2}) in (iii). (\ref{ini4}) will be proved in
Appendix.

 Now we conclude this section by proving Proposition \ref{iter}.  We first prove
(\ref{nabd:e}). By using (\ref{7.15}) we have
\begin{equation*}
\|\sn {\D}^{-1} E_k \|_{{\mathcal P}^0}\les \|E_k \|_{{\mathcal
P}^0}+(\gc)\|{\D}^{-1} E_k \|_{L_t^b L_x^2}^{q}\| E_k
\|^{1-q}_{L^2},
\end{equation*}
where $4<b<\infty$ and $1/2<q<1$.

Thus it suffices to show for $4<b<\infty$ that
\begin{equation}\label{it0.3}
\|\t1a {\D}^{-1} E_{k}\|_{L_t^b L_x^2}\les (\gc)\left({\mathcal
N}_2({\D}^{-1}P_{k-1})+\|\sn_L {\D}^{-1}P_{k-1}\|_{L_t^b
L_x^2}\right).
\end{equation}
By the construction of $P_k$ and $E_k$, it suffices to show it for
$k=1$. To this end, in view of $E_1=E_1^{G}+E_1^{B}$, we can
complete the proof by using Proposition \ref{er0.35} for $\|\t1a
\D^{-1} E_1^G\|_{L_t^b L_x^2}$ and the estimate
\begin{align*}
\|\t1a {\D}^{-1} E_1^{B}\|_{L_t^b L_x^2}&\les
\|r^{\frac{1}{2}-\frac{1}{b}}E_1^{B}\|_{L_t^b
L_x^{4/3}}\les\|{\D}_1^{-1}\check R\|_{L_t^\infty L_x^4}\|\Dt
F\|_{L_t^b L_x^2}\\
&\les\N_1(\D^{-1} \ckk R)\|\Dt F\|_{L_t^b L_x^2}\les
(\Delta_0^2+{\R}_0)\|\sn_L F\|_{L_t^b L_x^2}.
\end{align*}
which follows from Proposition \ref{P7.3}, H\"older inequality,
(\ref{sob.m}) and (\ref{7.24}).

In order to prove (\ref{lldp}), we first note that
\begin{equation}\label{it0.21}
\|\t1a \sn_L {\D}^{-1} P_k\|_{L_t^b L_x^2}\les\|\t1a [\Dt,
{\D}^{-1}]P_k\|_{L_t^b L_x^2}+\|\t1a {\D}^{-1} \Dt P_k\|_{L_t^b
L_x^2}.
\end{equation}
By using (\ref{come1}), the first term on the right hand side of
(\ref{it0.21}) can be estimated as
\begin{align*}
\|\t1a C_0(P_k)\|_{L_t^b L_x^2}&\les
{\N}_1(r^{-\frac{1}{2}}{\D}^{-1} P_k)\les
{\N}_2(r^{\frac{1}{2}}{\D}^{-1}P_k)\les {\N}_1(P_k),
\end{align*}
while by using (\ref{30000}), (\ref{it0.3})  and (\ref{betae1}), the
second term can be estimated as
\begin{align*}
\|\t1a {\D}^{-1} \Dt P_k\|_{L_t^b L_x^2}&\les \|\t1a
{\D}^{-1}(an\b\c
\D^{-1} P_{k-1}-E_k)\|_{L_t^b L_x^2} \\
&\les \|\t1a {\D}^{-1}(an\b\c \D^{-1}P_{k-1})\|_{L_t^b
L_x^2}+\|\t1a {\D}^{-1} E_k\|_{L_t^b L_x^2}\\
&\les(\gc)({\mathcal N}_2(\D^{-1} P_{k-1})+\|\sn_L
\D^{-1}P_{k-1}\|_{L_t^b L_x^2}).
\end{align*}
Therefore (\ref{lldp}) is proved.
\end{proof}

\section{\bf Main estimates}\label{me}
\setcounter{equation}{0}

\begin{lemma}
Let  $F=\sn \tr\chi, \mu, A\c \Ab, r^{-1} \Ab$,  there holds
\begin{equation}\label{d0m}
\|\sn\D^{-1}(an F)\|_{\P^0}\les \gc+\|F\|_{\P^0}
\end{equation}
where $\D^{-1}$ is one of the operators $\D_1^{-1},\,\D_2^{-1},\, {}^\star \D_1^{-1}$.
\end{lemma}

\begin{proof}
By Proposition \ref{P7.3} and (\ref{smry1}), we have
$$
\|{\D}^{-1} (an F)\|_{L_t^b L_x^2}\les\|an r F\|_{L_t^b
L_x^2}\les\|r F\|_{L_t^b L_x^2}\les \gc.
$$
By (\ref{smry1}), $\|F\|_{L^2(\H)}\les\gc$. We can infer from
Theorem \ref{in} and (\ref{A93}) that
\begin{equation*}
\|\sn{\D}^{-1}(anF)\|_{\P^0}\les \|anF\|_{\P^0}+\gc\les
\|F\|_{\P^0}+\gc.
\end{equation*}
as desired.
\end{proof}

Now we improve BA1 with the help of Theorem \ref{T1.1}.

\subsection{Estimates for $\nu$ and $|a-1|$}

\begin{proposition}\label{trace2}
\begin{equation*}
\|\nu\|_{L_\omega^\infty L_t^2}\les \Delta_0^2+\RR, \quad
|a-1|\le\frac{1}{4}.
\end{equation*}
\end{proposition}
\begin{proof}
We rewrite (\ref{lloga}) as follows
\begin{align*}
-\sn \nu&=\sn_L (\sn  a)+\er_2, \mbox{ and }
\er_2=\frac{1}{2}\tr\chi \sn  a +\chih \c \sn a +A\c\nu.
\end{align*}
 Let us
denote symbolically $ \er_2=\tr\chi\c \Ab+A\c \Ab,$ hence
\begin{equation*}
-\sn(an\nu)=\Dt\sn a+an \wt\er_2
\end{equation*}
with $\wt \er_2=-\sn\log(an) \nu+\er_2=A\c \Ab+r^{-1} \Ab.$ By
(\ref{cor3}) we can obtain
\begin{equation*}
\|an\wt\er_2\|_{\P^0}\les \Delta_0^2+\RR.
\end{equation*}
Applying Theorem \ref{T1.1} to $P= \sn a$ and $E=an\c\wt\er_2$,
 we have
$$
 \|\nu\|_{L_\omega^\infty L_t^2}\les \N_1(\nu)+\N_1(P)+\|an\wt \er_2\|_{\P^0}\les\gc,
$$
where in view of (\ref{smry1}), $ \N_1(\nu)+\N_1(P)\les\gc$.

In view of $\nu:=-\frac{d}{ds} a$ and $a(p)=1$,
$$
|a-1|\le \int_0^t |\nu|  na dt'\les \|\nu\|_{L_\omega^\infty L_t^2}\les \gc.
$$
With $\gc$ being sufficiently small, $|a-1|\le \frac{1}{4}$ can be achieved.
Then Proposition \ref{trace2} follows.
\end{proof}

\subsection{Estimate for $\zb$}

\begin{proposition}\label{trace1}
\begin{equation*}
\|\zb\|_{L_\omega^\infty L_t^2}\les \Delta_0^2+\RR.
\end{equation*}
\end{proposition}

\begin{proof}
By (\ref{errp1})
\begin{align*}
 \sn \D_1^{-1}\left(an(\rho, \sig)\right)&=\sn\D_1^{-1}{}^\star \D_1^{-1}\Dt \udb-\sn \D_1^{-1} \wt Err\\
 &=\Dt \sn \D^{-2}\ckk R+C(\ckk R)+\sn \D^{-1}\ff.
 \end{align*}
By Proposition \ref{de:er}, there exists $P$
and $E$ such that $C(\ckk R)=\Dt P+E$.

Let $\ti P= P+\sn \D^{-2}\ckk R$ and $\ti E=\sn \D^{-1} \ff+E$.
Then by (\ref{Coro1}) and (\ref{itF1})
\begin{equation*}
\sn \D_1^{-1}\left(an(\rho, \sig)\right)=\D_t \ti P+\ti E,\quad
\N_1(\ti P)+\|\ti E\|_{\P^0}\les \Delta_0^2+\RR.
\end{equation*}
In view of (\ref{nabzb}),
\begin{align}
\sn (an\zb)&= \sn \D_1^{-1}(\Dt(a\delta+2a\lambda))+\sn
\D_1^{-1}(an(\rho,\sig))+\sn \D_1^{-1}
\er_1,\nn\\
&=\Dt\sn \D_1^{-1}(a \delta+2a\lambda)+[\sn \D_1^{-1},
\Dt](a\delta+2a\lambda)+\sn \D_1^{-1} \er_1+\Dt \ti P+ \ti
E.\label{nzb2}
\end{align}
In view of Proposition \ref{de:er}, there exists $P'$ an $E'$ such
that
$$
[\sn \D_1^{-1}, \Dt](a\delta+2a\lambda)=\Dt P'+E', \with
\N_1(P')+\|E'\|_{\P^0}\les \gc.
$$
Let $P''=P'+\sn \D_1^{-1}(a\delta+2a\lambda)$, we conclude that
\begin{equation*}
\Dt\sn \D_1^{-1}(a \delta+2a\lambda)+[\sn \D_1^{-1},
\Dt](a\delta+2a\lambda)=\Dt P''+E'.
\end{equation*}
By (\ref{Coro1}), $\N_1(P'')\les \gc$. Hence, we obtain
\begin{equation*}
\sn (an \zb)=\Dt P_3+E_3,
\end{equation*}
where $E_3=E'+\sn \D_1^{-1}\er_1+\ti E$ and $P_3=P''+\ti P$. Also
using Lemma \ref{p0er1} for $\|\sn \D_1^{-1} \er_1\|_{\P^0}$, we can
conclude that
$$
\|E_3\|_{\P^0}\les \gc,\quad \quad
\N_1(P_3)\les\gc.
$$
Proposition \ref{trace1} then follows by  using
Theorem \ref{T1.1} and $\N_1(\zb)\les\gc$.
\end{proof}

\subsection{Estimate for $\chih$}

\begin{proposition}
\begin{equation*}
\|\chih\|_{L_\omega^\infty L_t^2}\les \gc
\end{equation*}
\end{proposition}

\begin{proof}
 First,
by (\ref{0.7}),
\begin{align*}
\div(an\chih)&=\f12 an\sn\tr\chi+\f12 an\tr\chi
\zeta-an\b+an\zb\chih\\
&=anM+ an A\c A+r^{-1}an\zeta-an\b
\end{align*}
from (\ref{errp1}),
$$
an\beta={\mathcal D}_1^{-1}\Dt(\check \rho, -\check \sigma)-\ff
\quad \mbox{with} \quad \ff=Err.
$$
Hence
\begin{align*}
\div(an\chih)&=an M+an A\c A+r^{-1}an\zeta-\D_1^{-1}\Dt (\ckk\rho,
-\sigc)+\ff
\end{align*}
This gives
\begin{equation*}
an\chih=-{\D}_2^{-1}{\D}_1^{-1}\Dt (\check
\rho,-\check\sig)+{\mathcal D}_2^{-1} (\ff+an M +anA\cdot
A+r^{-1}an\zeta).
\end{equation*}
Set ${\mathcal D}^{-2}={\mathcal D}_2^{-1}{\mathcal D}_1^{-1}$ and
${\D}^{-1}={\D}_2^{-1}$, we obtain after taking covariant
derivatives
\begin{equation}\label{8.3}
\sn(an\chih)=-\sn{\mathcal D}^{-2}\Dt\check R+F+\sn {\D}^{-1}(anM),
\end{equation}
where $F=\sn{\mathcal D}^{-1}(\ff+an A\cdot A+r^{-1}an\zeta)$ and
$M=\sn\tr\chi$.

By (\ref{itF1}) and (\ref{d0m}),
\begin{equation}\label{f}
\|F\|_{{\mathcal P}^0}\lesssim \Delta_0^2+\RR.
\end{equation}

Consider the first term on the right of (\ref{8.3}). By using the
notations in (\ref{7.40}), we can write
\begin{equation*}
\sn{\mathcal D}^{-2}\Dt(\check R)=\Dt(\sn {\mathcal D}^{-2}\check
R)+C(\check R).
\end{equation*}
where, by Proposition \ref{de:er}, there exist tensors $P'$ and $E'$
so that $C(\check R)=\Dt P'+ E'$ and
\begin{equation}\label{crr}
{\mathcal N}_1(P')+\|E'\|_{{\mathcal P}^0}\lesssim \gc,
\quad\,\lim_{t\rightarrow 0} r\|P'\|_{L_x^\infty}=0.
\end{equation}
 Thus (\ref{8.3})
becomes
\begin{equation}\label{main0.2}
\sn(an\chih)=\Dt P+\sn{\mathcal D}^{-1} (anM)+E
\end{equation}
where $P=\sn{\mathcal D}^{-2}\check R+P'$ and $E=F+E'$. By using
(\ref{Coro1}),(\ref{f}) and (\ref{crr})
\begin{equation}\label{b4}
{\N}_1(P)+\|E\|_{{\mathcal P}^0}\les \Delta_0^2+\R_0,\quad\,
\lim_{t\rightarrow 0} r\|P\|_{L_x^\infty}=0
\end{equation}
By combining (\ref{main0.2}) with
(\ref{s4}) we obtain
\begin{align*}
\frac{d}{ds} M+\frac{3}{2}\tr\chi M&=-\chih\c M-2(an)^{-1} \chih(\Dt
P+E+\sn \D^{-1}(an M))-\f12 (\tr\chi)^2(\zeta+\zb),
\end{align*}
regarding $\iota$  as an element of $A$, symbolically,
\begin{align*}
\sn_L M+\frac{3}{2}\tr\chi M=A\c M&+(an)^{-1}\chih\cdot(\Dt
P+E+\sn{\mathcal D}^{-1} (anM))\\&+(r^{-1} A+A\c A)\c A+r^{-2} A.
\end{align*}
Then
\begin{align*}
\Dt (r^3 M)=-\frac{3}{2}r^3 an \kp M&+r^3an\{A\c M+(an)^{-1}\chih
(\Dt
P+E+\sn \D^{-1} (an M))\}\\
&+r^3 an A(A\c A+r^{-1} A)+r an A.
\end{align*}
Let us pair $\sn \tr\chi$ with vector fields $X_i$ in Lemma
\ref{correct4}, which is still denoted by $M$. Regarding $\kp$ also
as  an element of $A$,   integrating in $t$, in view of
$\lim_{s\rightarrow 0} r\sn \tr\chi=0$,
\begin{align*}
M&=r^{-3}\int_0^t {r'}^3 an A\c M + {r'}^3\chih\c \Dt P
\\&+r^{-3}\int_0^t {r'}^3A\c (E+\sn \D^{-1}(an M)+an A\c
A+{r'}^{-1}an  A)+an r' A dt'.
\end{align*}
Using Lemma \ref{correct4}, (\ref{5.33}) and (\ref{5.32}) in view of
$\lim_{t\rightarrow 0} \|\chih\|_{L_x^\infty}<\infty$, we can obtain
\begin{align*}
\| r^{-3}&\int_0^t {r'}^3  A\c\left( an(M+ A\c A+r^{-1 } A )+E+\sn
\D^{-1}(an M)\right)+ {r'}^3 \chih\c \Dt P\|_{\B^0}\\&\les
(\N_1(A)+\|A\|_{L_\omega^\infty L_t^2})(\N_1(P)+\|E\|_{\P^0}+\|\sn
\D^{-1}(an M)\|_{\P^0}\\&+\|an M\|_{\P^0}+\|r^{-1} an
A\|_{\P^0}+\|an A\c A\|_{\P^0}).
\end{align*}
By Proposition \ref{eq1.2}, (\ref{hlm}), (\ref{cor3}) and
(\ref{smry1}), we obtain
\begin{equation*}
\|r^{-3} \int_0^t an r' A\|_{\P^0}\les \|r^{-1}an A\|_{\P^0}\les
\N_1(A)\les\gc.
\end{equation*}
Hence, in view of (\ref{A93}),  BA1, (\ref{cor3}) and (\ref{d0m}),
we can obtain
\begin{align}
\|M\|_{\P^0} &\les \left({\mathcal N}_1(P)+\|M\|_{{\mathcal P}^0}
+\|E\|_{{\mathcal P}^0}+\gc\right)\left({\mathcal
N}_1(A)+\|A\|_{L_\omega^\infty L_t^2}\right)+\gc \nn\\&\les \Delta_0
\left(\|M\|_{{\mathcal
P}^0}+\N_1(P)+\|E\|_{\P^0}+\Delta_0^2+{\R}_0\right)+\gc.\label{8.13.1}
\end{align}
Since $0<\Delta_0<1/2$ can be chosen to be sufficiently small such
that the first term of (\ref{8.13.1}) can be absorbed,  in view of
(\ref{b4}) we then obtain that
\begin{equation}\label{mp3}
\|\sn \tr\chi\|_{{\mathcal P}^0}\les \Delta_0^2+{\R}_0.
\end{equation}
Thus, by setting $\tilde E=E+\sn{\D}^{-1}(anM)$ we obtain from
(\ref{main0.2}), (\ref{d0m}) and (\ref{mp3}) the decomposition
\begin{equation}\label{main0.3}
\sn(an\chih)=\D_t P+\tilde E \quad \mbox{and} \quad
{\N}_1(P)+\|\tilde E\|_{{\mathcal P}^0}\les \Delta_0^2+{\mathcal
R}_0.
\end{equation}
By Theorem \ref{T1.1} and (\ref{smry1}), we conclude
$$
\|\chih\|_{L_\omega^\infty L_t^2}\les \N_1(\chih)+\N_1(P)+\|\ti
E\|_{\P^0}\les \gc.
$$
as expected.

\end{proof}
Similar to the derivation of (\ref{mp3}), we  can get
\begin{equation}\label{8.14.1}
\|r^\f12 \sn \tr\chi\|_{{\mathcal B}^0}\les \Delta_0^2+{\R}_0.
\end{equation}

\subsection{Estimate for $\zeta$}

\begin{proposition}
\begin{equation*}
\|\zeta\|_{L_\omega^\infty L_t^2}\les \gc.
\end{equation*}
\end{proposition}

\begin{proof}
By using (\ref{hdg1}) and (\ref{hdg2}),
\begin{align*}
\div(an\zeta) &=an\zb\c \zeta-an\mu-an\rhoc+\f12 a^2n\delta \tr\chi, \\
\curl(an\zeta) &=an\sigc+an(\zeta+\zb)\wedge\zeta.
\end{align*}
Symbolically, $ \D_1(an\zeta)=an A\c A-an(\mu,0)-an(\rhoc, -\sigc)+
an(a\delta\tr\chi, 0). $ Hence
\begin{align*}
an\zeta&=-\D_1^{-1}(an(\rhoc, -\sigc))+\D_1^{-1}(an A\c
\Ab)-\D_1^{-1}(an(\mu, 0))+\D_1^{-1}(an(r^{-1}\Ab, 0)).
\end{align*}
Let $J$ be the involution $(\rho, \sigma)\rightarrow (-\rho,
\sigma),$ $\ff=\widetilde{Err}$ is given by (\ref{errp1}),
\begin{align*}
\sn (an\zeta)=&\sn{\mathcal D}_1^{-1}\cdot J\cdot {}^\star {\mathcal
D}_1^{-1}\Dt \bb+\sn{\D}_1^{-1}\c J\cdot \ff-\sn{\mathcal
D}_1^{-1}(an\mu, 0) \\&+\sn{\mathcal D}_1^{-1}(an\c A\cdot
A)+\sn\D_1^{-1}(r^{-1}an\Ab, 0).
\end{align*}
Set ${\mathcal D}^{-2}={\mathcal D}_1^{-1}\cdot J\cdot{}^\star {\mathcal
D}_1^{-1}$ and ${\mathcal D}^{-1}={\mathcal D}_1^{-1}$. By using
(\ref{7.40}), we get
\begin{eqnarray*}
\sn(an\zeta)=\Dt\sn{\mathcal D}^{-2}\bb+C(\check R)+\sn{\mathcal
D}^{-1}(anM) +F
\end{eqnarray*}
where $M=(\mu,0)$ and $F=\sn{\D}^{-1} \big(\ff+an (A\cdot
\Ab+r^{-1}\Ab)\big)$.

By (\ref{itF1}) and  (\ref{d0m}), we derive $\|F\|_{{\mathcal
P}^0}\les \Delta_0^2+{\R}_0$.

In view of Proposition \ref{de:er}, for some tensors $\ti P$ and
$\ti E$ such that  $C(\ckk R)=\Dt \ti P+\ti E$. With $ E=\ti E+F$, $
P=\ti P+\sn\D^{-2} \bb$,  we can write $$\sn(an \zeta)=\Dt
P+\sn{\mathcal D}^{-1}(anM)+ E,$$ and
 \begin{equation}\label{0.12}
 {\mathcal N}_1(P)+\|E\|_{\P^0}\le \Delta_0^2+\R_0, \quad\, \lim_{t\rightarrow 0}r\|P\|_{L_x^\infty}=0
\end{equation}

Let $M=(\mu,0)$, (\ref{mumu}) can be written as
\begin{align*}
\frac{d}{ds}M+\tr\chi M&=2(an)^{-1}\chih \left(\Dt P+ E+\sn \D^{-1}(an
M) \right)-2\chih\c \zb\c \zeta+(\zeta-2\zb) \tr\chi
\zeta\\&+\sn\tr\chi(\zeta-\zb)+\tr\chi \rhoc+\left(\frac{1}{4} a^2
\tr\chi+\Ab \right)|\chih|^2-\f12 a\nu (\tr\chi)^2.
\end{align*}
Symbolically,
\begin{align}
\frac{d}{ds} M+\tr\chi M=&(an)^{-1}\chih\cdot \left(\Dt
P+\sn{\mathcal D}^{-1}(anM)+E\right) +\tr\chi\check \rho\nn\\& +A\c
\left(A\c \underline{A}+r^{-1} A+\sn\tr\chi \right)
+r^{-1}a\tr\chi A.\label{mu2}
\end{align}
In view of  $\lim_{t\rightarrow 0} r^2 \mu=0$ in Lemma \ref{inii},
we deduce
\begin{equation}\label{m5}
M= r^{-2}\int_0^t  {r'}^2\left( an\kp\c M+\chih\c \D_t P+A\c \ti F+an
\tr\chi \rhoc+{r'}^{-1}a^2n \tr\chi A \right) dt'
\end{equation}
with $\ti F=an (A\c \Ab+\sn \tr\chi+r^{-1} A)+E+\sn \D^{-1}(an M).$

In view of Proposition \ref{main:lem}, regarding $\kp$ as an element
of $A$,
\begin{align}
\Big\|r^{-2}\int_0^t{ r'} ^2 & \left(an\kp\c M+\chih\c \D_t P+A\c \ti F\right)
dt'\Big\|_{\B^0}\label{b1}\\&\les \left(\N_1(A)+\|A\|_{L_\omega^\infty
L_t^2}\right) \left(\|anM\|_{\P^0}+\N_1(P)+\|\ti F\|_{\P^0}\right)\nn.
\end{align}
Note that by (\ref{cor3}), (\ref{A93}), (\ref{d0m}) and (\ref{As}),
we deduce
\begin{align}
\|\ti F\|_{\P^0}&\les \|an(A\c
\Ab+r^{-1}A)\|_{\P^0}+\|an\sn\tr\chi\|_{\P^0}+\|
E\|_{\P^0}+\|\sn \D^{-1}(an M)\|_{\P^0}\nn\\
&\les\|\sn\tr\chi\|_{\P^0}+\|E\|_{\P^0}+\|M\|_{\P^0}+\gc\nn.
\end{align}
By (\ref{mp3}) and  (\ref{0.12})
\begin{equation*}
\|\ti F\|_{\P^0}\les \gc+\|M\|_{\P^0}.
\end{equation*}
Hence, by (\ref{As}), BA1, (\ref{A93}) and (\ref{0.12})
\begin{equation*}
(\ref{b1})\les \Delta_0 (\|M\|_{\P^0}+\gc).
\end{equation*}

Assuming the following estimate for $\P^0$ norm of the last two
terms in (\ref{m5})
\begin{equation}\label{b5}
\Big\|r^{-2} \int_0^t {r'}^2 \left( an \tr\chi \rhoc+{r'}^{-1}a^2n \tr\chi A \right)
dt' \Big\|_{\P^0}\les \gc,
\end{equation}
since $0<\Delta_0<1/2$ can be chosen sufficiently small, we conclude
that
\begin{equation}\label{8.25.1}
\|M\|_{{\mathcal P}^0}\les \Delta_0^2+{\mathcal R}_0.
\end{equation}
By (\ref{d0m}) and (\ref{smry1}), $\|\sn \D^{-1} (an M)\|_{\P^0}\les
\gc$.

In view of (\ref{0.12}), we have $\sn (an\zeta)=\D_t P+E'''$, with
$E'''= E+\sn \D^{-1}(an M)$. By Theorem \ref{T1.1},
$$
\|\zeta\|_{L_\omega^\infty L_t^2}\les
\N_1(P)+\|E'''\|_{\P^0}+\N_1(\zeta)\les\gc.
$$

We now  prove (\ref{b5}). With the help of (\ref{7.7}), by letting
$$p':={}^\star\D_1^{-1}{\underline \b}, \quad e'=[{}^\star
\D_1^{-1}, \Dt]\udb-\wt{ Err}+an A\c \Ab,$$ also noting that
(\ref{7.21}) gives $\|[\Dt, {}^\star\D_1^{-1}] \check R\|_{{\mathcal
P}^0}\lesssim \Delta_0^2+{\mathcal R}_0$, combined with
(\ref{errp}), (\ref{cor3}) and (\ref{As}), we can get the following
decomposition
\begin{equation}\label{8.17.10}
an(\check\rho, \check \sigma)=\Dt p'+e' \quad \mbox{with} \quad
{\N}_1(p')+\|e'\|_{{\mathcal P}^0}\les \Delta_0^2+{\R}_0.
\end{equation}

(\ref{b5}) can be derived by establishing the following inequalities
\begin{align}
& \left\|r^{-2} \int_0^t {r'}^2\tr\chi\Dt p' dt'\right\|_{{\mathcal
P}^0}\les \gc,\label{b2} \\
&\left\|r^{-2}\int_0^t {r'}^2 \tr\chi \left(e'+{r'}^{-1}an A\right)
dt'\right\|_{{\mathcal P}^0}\les \Delta_0^2+{\R}_0.\label{b3}
\end{align}
To prove (\ref{b3}), by Proposition \ref{eq1.2}, (\ref{hlm}), also
in view of (\ref{cor3}) and (\ref{8.17.10}), we can obtain
\begin{equation*}
\left\|r^{-2}\int_0^t r'\left(e'+{r'}^{-1}an A\right) dt'\right\|_{{\mathcal
P}^0}\les \|e'\|_{\P^0}+\|r^{-1}an A\|_{\P^0}\les\gc.
\end{equation*}
By (\ref{5.33}) and (\ref{8.17.10}), we get
\begin{align*}
 &\left\|r^{-2}\int_0^t {r'}^2\iota \c e'dt'\right\|_{\B^0}\les
 (\N_1(\iota)+\|\iota\|_{L_\omega^\infty
 L_t^2})\|e'\|_{\P^0}\les\gc,
\end{align*}
and similarly
\begin{equation*}
 \left\|r^{-2}\int_0^t r'\iota \c an A dt'\right\|_{\B^0}\les
 (\N_1(\iota)+\|\iota\|_{L_\omega^\infty
 L_t^2})\|r^{-1} an A\|_{\P^0}\les\gc.
\end{equation*}
The proof of (\ref{b3}) is complete.

Now we prove (\ref{b2}). Recall that $p'={}^\star
\D_1^{-1}{\underline \b}$, then
\begin{equation}\label{mp5}
\lim_{s\rightarrow 0}r p'=0,\,\, \|\tr\chi p'\|_{\P^0}\les
{\N}_1(p')\lesssim {\mathcal R}_0+\Delta_0^2.
\end{equation}
Using  Proposition \ref{eq1.2} and (\ref{mp5}), also in view of
(\ref{s1}), we derive
\begin{align*}
& \left\|\frac{1}{r^2}\int_0^t {r'}^2\tr\chi \Dt p' dt'\right\|_{\P^0} \\
& \qquad \lesssim
\sum_{k>0} \left\| E_k\frac{1}{r}\int_0^t {r'}^2\tr\chi \Dt p' dt'
\right\|_{L_t^2 L_\omega^2} + \left\|\frac{1}{r}\int_0^t {r'}^2\tr\chi
\Dt p' dt' \right\|_{L_t^2 L_\omega^2}\\
&\qquad \lesssim \sum_{k>0} \left(\|E_k( r'\tr\chi p')\|_{L_t^2
L_\omega^2}+\|E_k r^{-1}\int_0^t \Dt({r'}^2\tr\chi) p'
dt'\|_{L_t^2 L_\omega^2}\right) +\|\sn_L p'\|_{L_t^2 L_x^2}\\
&\qquad \lesssim \|\tr\chi p'\|_{\P^0}+\|\tr\chi
p'\|_{L^2(\H)}+\|\sn_L p'\|_{L_t^2 L_x^2}\\
&\qquad +\sum_{k>0} \left(\left\|E_k r^{-1}\int_0^t{ r'}^2 \ovl{an\tr\chi}\tr\chi
p'dt'\right\|_{L_t^2 L_\omega^2} \right.\\
&\qquad \qquad \qquad \qquad\qquad
+ \left. \left\|E_k r^{-1}\int_0^t {r'}^2 an \left((\tr\chi
)^2+|\chih|^2\right) p' dt'\right\|_{L_t^2 L_\omega^2} \right).
\end{align*}
Using  (\ref{5.33}), (\ref{cor3}), BA1 and (\ref{smry1}),
\begin{align*}
\sum_{k>0} \left\|E_k r^{-1}\int_0^t {r'}^2 an |\chih|^2\c p'
dt'\right\|_{L_t^\infty L_\omega^2} &\les \left(
\N_1(\chih)+\|\chih\|_{L_\omega^\infty L_t^2}\right) \|an \chih \c
p'\|_{\P^0}\\&\les (\gc) \N_1(p')\les \gc,
\end{align*}
where for the last inequality, we employed (\ref{8.17.10}).

It is easy to see by (\ref{cor2})
\begin{align*}
&\sum_{k>0} \left\|E_k r^{-1}\int_0^t{ r'}^2 \ovl{an\tr\chi}\tr\chi
p'dt' \right\|_{L_t^2 L_\omega^2}\les \|\tr\chi p'\|_{\P^0}\les\N_1(p').
\end{align*}
and
\begin{align}
\sum_{k>0}& \left\|E_k r^{-1}\int_0^t {r'}^2 an (\tr\chi)^2 p'
dt'\right\|_{L_t^2 L_\omega^2}\nn\\
&\le \sum_{k>0}\left( \left\|E_k r^{-1}\int_0^t {r'}^2 an \kp\tr\chi p'
dt' \right\|_{L_t^2 L_\omega^2}+ \left\|E_k r^{-1}\int_0^t {r'}^2
\ovl{an \tr\chi} \tr\chi p' dt' \right\|_{L_t^2 L_\omega^2}\right)\nn\\
&\les \|an\kp\c r\tr\chi  p'\|_{\P^0}+\|\tr\chi p'\|_{\P^0}\les
(\gc+1)\N_1(p')\label{m8}
\end{align}
where we employed  (\ref{cor2}) and
\begin{equation}\label{m9}
\|an \kp \c \tr\chi p'\|_{\P^0}\les \N_1(p')(\gc).
\end{equation}
To see (\ref{m9}), we first deduce with the help of (\ref{A93}) that
\begin{equation}\label{m7}
 \|an\kp\c r\tr\chi  p'\|_{\P^0}\les \|\kp\c r\tr\chi p'\|_{\P^0}.
\end{equation}
By (\ref{7.1}) with $G=r\tr\chi \kp$
\begin{equation}\label{m6}
\|\kp\c r\tr\chi p'\|_{\P^0}\les \N_1(p') \left( \|r^{\f12} \sn (r \tr\chi
\kp)\|_{L^2(\H)}+\|r^{1-\frac{1}{b}}\tr\chi \kp\|_{L_t^b L_x^2}\right)
\end{equation}
where $b>4$.
 Since $\|r\kp\|_{L^\infty}\les C$ and
\begin{align*}
\|r^{\f12} \sn (r\tr\chi \kp)\|_{L^2(\H)}&\les\|r^{\frac{3}{2}} \sn
\tr\chi \kp\|_{L^2(\H)}+\|r^{\frac{3}{2}}  \tr\chi
\sn \kp\|_{L^2(\H)}\\
&\les \|\sn \tr\chi \|_{L^2(\H)}+\|\sn \kp\|_{L^2(\H)} \\
&\les \gc
\end{align*}
where for the last inequality, we employed (\ref{smry1}) and
Proposition \ref{n1tr}.  And by (\ref{eq.7}), we have
\begin{align*}
\|r^{1-\frac{1}{b}}\tr\chi \kp\|_{L_t^b
L_x^2}&\les\|r^{-\frac{1}{b}}\kp\|_{L_t^b L_x^2} \les\gc.
\end{align*}
Thus in view of (\ref{m6}) and (\ref{m7}), (\ref{m9}) is proved.
In view of (\ref{mp5}),  (\ref{b2}) is proved.
\end{proof}

Similar to (\ref{8.25.1}), we can obtain
$$
\|r^\f12 \mu\|_{\B^0}\les \gc.
$$

\section{\bf Appendix}\label{ap}
\setcounter{equation}{0}

\subsection{Dyadic sobolev inequalities}
We start with  proving Lemma \ref{sl2} and a few useful consequences,
which will be used to prove (\ref{ini4}) in the second subsection.

Let us still regard $\kp$ and $\iota$ as elements of $A$. Using
Propositions \ref{smr}, \ref{n1tr}, Lemma \ref{gauss3} and
(\ref{cond1}), more precisely
\begin{equation*}
\|\K\|_{L_t^2 L_x^2}+\|\b\|_{L_t^2 L_x^2}+\N_1(A)\lesssim
\gc,
\end{equation*}
we can adapt the approach in \cite[Lemma 5.3]{KRs} and \cite[Chapter 9]{WangQ} to derive
\begin{lemma}\label{sl1}
For any smooth $S_t$ tangent tensor fields $F$ and all $q<2$
sufficiently close to $q=2$,
\begin{align}
&\|r^{\frac{1}{2}-\frac{1}{q}}[P_k, \Dt]F\|_{L_t^q
L_x^2}+2^{-k}\|r^{\frac{3}{2}-\frac{1}{q}}\sn[P_k, \Dt]F\|_{L_t^q
L_x^2}\lesssim 2^{-\frac{k}{2}+}\N_1(F) ,\label{cc2}\\
&\|r^{-\f12}[P_k, \Dt]F\|_{L_t^1 L_x^2}+2^{-k}\|r^{\f12}\sn[P_k,
\Dt]F\|_{L_t^1 L_x^2}\lesssim 2^{-k+}\N_1(F).\label{cc1}
\end{align}
\end{lemma}
Now we are ready to prove Lemma \ref{sl2}.

\begin{proof}[Proof of Lemma \ref{sl2}]
The result is trivial when $q=2$. So we only need to
consider the case $q>2$. It is easy to get the following estimate
\begin{equation*}
\|r^{-\frac{1}{2}-\frac{1}{q}}P_k F\|_{L_t^q L_x^2}\lesssim
\|r^{-1}P_k F\|_{L_t^2 L_x^2}^{\frac{2}{q}}\|r^{-\frac{1}{2}}P_k
F\|_{L_t^\infty L_x^2}^{\frac{q-2}{q}}.
\end{equation*}
Moreover we have by integrating along an arbitrary null geodesic,
\begin{align}
\|r^{-\frac{1}{2}}P_k F \|_{L_x^2 L_t^\infty}^2&\lesssim \|P_k(\Dt
 F)\|_{L_t^2 L_x^2}\|r^{-1}P_k F\|_{L_t^2 L_x^2}+\|r^{-1}[P_k,
\Dt ]F\cdot P_k F\|_{L_t^1 L_x^1}\nonumber\\
&\quad  +\|r^{-1}P_k F\|_{L_t^2 L_x^2}^2.\label{f9}
\end{align}
We can obtain the following
estimate with $\frac{1}{q'}+\frac{1}{q}=1$ and $2<q<\infty$
\begin{align}
\|r^{-1} [P_k,\Dt F]\cdot P_k F\|_{L_t^1 L_x^1}&\lesssim
\|r^{\frac{1}{2}-\frac{1}{q}}[P_k,\Dt] F\|_{L_t^{q'}
L_x^2}\|r^{-\frac{1}{q}-\frac{1}{2}}P_k F\|_{L_t^{q} L_x^2}\nn\\
&\lesssim
2^{-\frac{k}{2}+}{\N}_1(F)\|r^{-\frac{1}{q}-\frac{1}{2}}P_k
F\|_{L_t^{q} L_x^2}\label{comm2}
\end{align}
where we employed Lemma \ref{sl1} to derive the last inequality.

Combining the above estimates we obtain
\begin{align}
\|r^{-\frac{1}{2}-\frac{1}{q}} P_k F\|_{L_t^q L_x^2}^{q} &\lesssim
\|r^{-1}P_k F\|_{L_t^2 L_x^2}^2\Big(\|P_k (\Dt F)\|_{L_t^2
L_x^2}\|r^{-1}P_k F\|_{L_t^2 L_x^2}\nonumber\\
&+ 2^{-\frac{k}{2}+}{\N}_1(F)\|r^{-\frac{1}{q}-\frac{1}{2}}P_k
F\|_{L_t^{q} L_x^2}+\|r^{-1}P_k F\|_{L_t^2
L_x^2}^2\Big)^{\frac{q}{2}-1} \label{f6}
\end{align}
  Using (\ref{FB}), we get for any $0\le \a
<\frac{\frac{1}{2}-\frac{1}{q}}{\frac{q}{2}+1}$,
\begin{equation*}
\|r^{-\frac{1}{2}-\frac{1}{q}}P_k F\|_{L_t^q L_x^2}\lesssim
2^{-k(\frac{1}{2}+\frac{1}{q})}(1+2^{-\a k}){\N}_1(F).
\end{equation*}
Combine  (\ref{f9}), (\ref{comm2})  and (\ref{TE}) with $2<q<\infty$, by using (\ref{FB})
we obtain  $$\|r^{-\f12} P_k F\|_{L_t^\infty L_x^2}\les 2^{-\f12 k}\N_1(F).$$

To see (\ref{TEE}), we derive by (\ref{sob.01}) and (\ref{FBB})
\begin{align*}
\|r^{-\frac{1}{q}} F_k\|_{L_t^q L_x^4}
&\les \|r^{\f12-\frac{1}{q}}\sn F_k\|_{L_t^q L_x^2}^{\frac{1}{2}}
\|r^{-\f12-\frac{1}{q}}F_k\|_{L_t^q L_x^2}^{\f12}
+\|r^{-\f12-\frac{1}{q}}F_k\|_{L_t^q L_x^2}\\
&\les (2^{\frac{k}{2}}+1)\|r^{-\frac{1}{2}-\frac{1}{q}}P_k
F\|_{L_t^q L_x^2} .
\end{align*}
Combined with (\ref{TE}), (\ref{TEE}) follows. 
\end{proof}

\begin{lemma}
Let $\D^{-1}$ denote one of the operators $\D_1^{-1}$, $\D_2^{-1}$
and ${}^\star\D_1^{-1}$. There hold the following estimates for
appropriate $S$ tangent tensor $G$,
\begin{align}
\|\D^{-1} P_k^2 G\|_{L_t^\infty L_x^2} &\les r^{\frac{3}{2}}2^{-\frac{3k}{2}} \N_1(G),\quad  k>0,\label{knf3}\\
\|\sn \D^{-1} P_k^2 G\|_{L_t^\infty L_x^4} &\les
\N_1(G)\left(1+2^{-\frac{k}{2}}r^{\f12}\|\K\|_{L_x^2}^{\f12}\right),\quad
k>0.\label{knf11}
\end{align}
\end{lemma}

\begin{proof}
(\ref{knf3}) can be obtained by using Lemma \ref{dual} and
(\ref{TE}). Now consider (\ref{knf11}). For any $S$ tangent tensor fields $F$,
define
$$\|F\|_{H_x^1}=\|\sn F\|_{L^2(S)}+\|r^{-1}F\|_{L^2(S)}.$$
  Recall by B\"ochner identity contained in  \cite{KR4}, there holds
\begin{equation}\label{Bochner1}
\|\sn^2 F\|_{L^2_x}\les\|\sD F\|_{L_x^2}+\|\K\c
F\|_{L_x^2}+\|\K\|_{L_x^2}^{\f12}\|\sn F\|_{L_x^4}+r^{-1}
\|F\|_{H_x^1},
\end{equation}
and by Sobolev embedding
\begin{equation}\label{sob2}
\|F\|_{L_x^\infty}\les r^{\frac{1}{p}}\|\sn^2
F\|_{L_x^2}^{\frac{1}{p}}\|F\|_{H_x^1}^{\frac{p-1}{p}}+\|
F\|_{H_x^1}, \quad\,2<p<\infty.
\end{equation}
It is easy to observe from \cite[P.38]{KC} that symbolically
\begin{equation}\label{knf2}
\sD={}^\star\D\D\pm(\K+r^{-2}Id).
\end{equation}
Hence with $2<p<\infty$,
\begin{equation}\label{Bochner1}
\|\sn^2 F\|_{L_x^2}\les \|{}^\star\D\D
F\|_{L_x^2}+\|\K\|_{L_x^2}^{\frac{p}{p-1}} r^{\frac{1}{p-1}}
\|F\|_{H_x^1}+\|\K\|_{L_x^2}\|\sn F\|_{L_x^2}+r^{-1} \|F\|_{H_x^1}.
\end{equation}
For $F=\D^{-1} H$, using Proposition \ref{P4.1}
\begin{align*}
\|\sn^2\D^{-1} H\|_{L_x^2}&\les \|{}^\star\D H\|_{L_x^2}
+\|\K\|_{L_x^2}^{\frac{p}{p-1}} r^{\frac{1}{p-1}} \|\D^{-1} H\|_{H_x^1}
+\|\K\|_{L_x^2}\|\sn \D^{-1} H\|_{L_x^2}\\
&\quad \, +r^{-1} \|\D^{-1}H\|_{H_x^1}\\
&\les \|{}^\star \D H\|_{L_x^2}+\|\K\|_{L_x^2}^{\frac{p}{p-1}}
r^{\frac{1}{p-1}} \| H\|_{L_x^2}+ \left(\|\K\|_{L_x^2}+r^{-1}\right)
\|H\|_{L_x^2}
\end{align*}
Let $H=P_k^2 G$,  by (\ref{FBB}), Proposition \ref{P4.1},  Lemma
\ref{dual}, we obtain,
\begin{align*}
\|\sn^2 \D^{-1} P_k^2 G\|_{L_x^2}&\les \|{}^\star\D P_k^2 G
\|_{L_x^2}+(\|\K\|_{L_x^2}+r^{-1})\|P_k^2 G\|_{L_x^2}\\
&\les\left((2^k+1) r^{-1}+\|\K\|_{L_x^2}\right)\|P_k G\|_{L_x^2}.
\end{align*}
By (\ref{sob.01}), Proposition \ref{P4.1}, (\ref{TE})
\begin{align}
\|\sn \D^{-1} P_k^2 G\|_{L_x^4}&\les \|\sn^2 \D^{-1}
P_k^2 G\|_{L_x^2}^{\f12}\|\sn \D^{-1} P_k^2 G\|_{L_x^2}^\f12
+r^{-\f12}\|\sn \D^{-1} P_k^2G\|_{L_x^2}\nn\\
&\les ( 1+2^{-\frac{k}{2}}r^\f12\|\K\|_{L_x^2}^\f12)\N_1(G).\nn
\end{align}
\end{proof}

\subsection{Proof of (\ref{ini4})}

We first prove (\ref{ini4}) by assuming the following results.
\begin{lemma}\label{lem:in1}
Denote by  $\D^{-1} G$ one of  the terms $\D_1^{-1} G$, $\D_2^{-1}G$ and ${}^\star\D_1^{-1}G$ for  appropriate $S$ tangent tensor $G$.
Let $F=\D^{-1} \ckk R\c \D^{-1} G$, we have
\begin{align*}
\|\sn F\|_{B_{2,1}^0}&\les \N_1(G)\left(\gc+c_0r^{\frac{1}{2}}(\|\D^{-1} \ckk
R\|_{L_x^\infty}+\|\ckk R\|_{L_x^2} +r^2\|\D \ckk
R\|_{L_x^\infty})\right)
\end{align*}
where $c_0$ depends on $\|r\K\|_{L_x^\infty}+\K_{\a_0}.$
\end{lemma}
\begin{lemma}\label{lem:in2}
Let $\D^{-1}$ denotes one of the operators $\D_1^{-1}$, $\D_2^{-1}$ and ${}^\star\D_1^{-1}$. For $S$
tangent tensor $H$, there hold
\begin{equation}\label{t1}
\|\sn \D^{-1} H\|_{B_{2,1}^1}\les \|{}^\star\D
H\|_{B_{2,1}^0}+\|H\|_{L_\omega^2}+c_0 r^{\f12} \N_1(H),
\end{equation}
\begin{equation}\label{tipp}
\|\sn \D^{-1} H\|_{L_x^\infty}\le \|{}^\star\D H\|_{B_{2,1}^0}+(c_0
r^\theta+1) (\|\sn H\|_{L_x^2}+\|H\|_{L_\omega^2}+c_0 r^\f12
\N_1(H)).
\end{equation}
where $c_0$ is depending on the quantity
$r\|\K\|_{L_x^\infty}+\K_{\a_0}+r\|\sn \K\|_{L_x^2}$, and $\theta>0$
is very close to $0$.
\end{lemma}

Let us set $P=\sum_{i=1}^\infty P_i$ and  $\ti P=\sn\D^{-1}P.$ Since
$\bar P=\ti P+P-P_1$ and in view of (\ref{ini3}) that
$\lim_{t\rightarrow 0} \|P_1\|_{L_\omega^\infty}=0$, (\ref{ini4})
can be proved by establishing
\begin{equation}\label{knf9}
\lim_{t\rightarrow 0} \|P\|_{L_x^\infty}=0, \quad \mbox{ and } \quad
\lim_{t\rightarrow 0}\|\ti P\|_{L_x^\infty}<\infty.
\end{equation}
In view of (\ref{pi}),
we have  by (\ref{sob.m2}), Lemma \ref{A.8} and (\ref{new1000}),
\begin{align*}
\|P_i\|_{L_\omega^\infty}&\le\|\D_1^{-1}\ckk R\|_{L_\omega^\infty}
\|\D^{-1}P_{i-1}\|_{L_\omega^\infty}\les  r^\f12 \|\D_1^{-1} \ckk R\|_{L_\omega^\infty} \N_2(\D^{-1} P_{i-1})\\
&\les r^\f12\N_1(P_{i-1}) \|\D_1^{-1}\ckk R\|_{L_\omega^\infty}\les
r^\f12(C(\gc))^{i-1}\N_2(F)\|\D_1^{-1} \ckk R\|_{L_\omega^\infty}.
\end{align*}
Summing over $i\ge 1$, with $(\gc)$ sufficiently small,
\begin{equation*}
\|P\|_{L_\omega^\infty}\les r^{\f12}\N_2(F)\|\D_1^{-1} \ckk
R\|_{L_\omega^\infty}.
\end{equation*}
Noting that $\lim_{t\rightarrow 0} \|\D_1^{-1} \ckk R\|_{L_\omega^\infty}<\infty$ and $\N_2(F)<\infty$,
$$
\lim_{t\rightarrow 0} \|P\|_{L_x^\infty}=0.
$$

It remains to prove the second part of (\ref{knf9}). By
(\ref{tipp}), there holds
\begin{equation}\label{tip}
\|\ti P\|_{L_x^\infty}\le \|\sn P\|_{B_{2,1}^0}+(c_0 r^\theta+1)
\left(\|\sn P\|_{L_x^2}+\|P\|_{L_\omega^2}+c_0 r^\f12 \N_1(P)\right).
\end{equation}
Recall the definition of $P_i$ in (\ref{pi}). For each
$P_i=\D^{-1} \ckk R\c \D^{-1} P_{i-1}$,
by Lemma \ref{lem:in1}, there holds
\begin{align*}
\|\sn P_i\|_{B_{2,1}^0} &\les
c_0\N_1(P_{i-1})r^{\frac{1}{2}} \left(\|\D^{-1} \ckk
R\|_{L_x^\infty}+\|\ckk R\|_{L_x^2} +r^2\|\D \ckk
R\|_{L_x^\infty}\right)\\
&\quad\, +\N_1(P_{i-1}) (\gc).
\end{align*}
In view of (\ref{new1000}), summing over $i\ge 1$ gives
\begin{align*}
\|\sn P\|_{B_{2,1}^0}&\les \sum_{i\ge 1}\|\sn P_i\|_{B_{2,1}^0}\\
&\les c_0\sum_{i\ge 1} \left(C(\gc)\right)^i\N_2(F)r^{\frac{1}{2}} \left(\|\D^{-1} \ckk
R\|_{L_x^\infty}+\|\ckk R\|_{L_x^2} +r^2\|\D \ckk
R\|_{L_x^\infty} \right)\\
&+\sum_{i\ge 1}(C (\gc))^i\N_2(F)(\gc).
\end{align*}
Hence
\begin{align}
\|\sn P\|_{B_{2,1}^0}&\les c_0\N_2(F)r^{\frac{1}{2}} \left(\|\D^{-1} \ckk
R\|_{L_x^\infty}+\|\ckk R\|_{L_x^2} +r^2\|\D \ckk
R\|_{L_x^\infty}\right) \nn \\
&\quad\, +\N_2(F)(\gc)\label{pbsf1}.
\end{align}
By (\ref{pi}), (\ref{sob.m}), Lemma \ref{A.8}, (\ref{new1000}) and
(\ref{7.24}),
\begin{align*}
\|P_i\|_{L_\omega^2}&\les r^{-1}\|\D_1^{-1} \ckk
R\|_{L_x^4}\|\D^{-1}P_{i-1}\|_{L_x^4}\les \N_1(\D^{-1} \ckk
R)r^{-1}\N_1(\D^{-1} P_{i-1})\\&\les \N_1(P_{i-1})(\gc)\les
(C(\gc))^{i-1}(\gc)\N_2(F).
\end{align*}
Therefore
\begin{equation}\label{pbsf2}
\|P\|_{L_\omega^2}\les (\gc)\N_2(F).
\end{equation}
It is easy to derive from $P=\sum_{i\ge 1} P_i$ and (\ref{new1000})
that
\begin{equation}\label{pbsf3}
\N_1(P)\les (\gc)\N_2(F).
\end{equation}
 Thus in view of (\ref{tip}), the combination of (\ref{pbsf1}), (\ref{pbsf2}) and (\ref{pbsf3}) implies
 $$
 \lim_{t\rightarrow 0} \|\ti P\|_{L_x^\infty}\les \N_2(F)( \gc)<\infty,
 $$
as desired. 

\begin{proof}[Proof of Lemma \ref{lem:in1}]
Let us compute $\|\sn F \|_{L_x^2}$ first.
\begin{align*}
\|\sn F\|_{L_x^2}&=\|\sn( \D^{-1} \ckk R\c \D^{-1} G)\|_{L_x^2}\\
&\les \|\sn \D^{-1} \ckk R\|_{L_x^2}\|\D^{-1}
G\|_{L_x^\infty}+\|\D^{-1} \ckk R\|_{L_x^4}\|\sn D^{-1} G\|_{L_x^4}.
\end{align*}
By (\ref{sob.m2}), (\ref{sob.m}), Proposition \ref{P4.1}, Lemma
\ref{A.8} and (\ref{7.24}),
\begin{align*}
\|\sn F\|_{L_x^2}&\les r^\f12 \|\ckk R\|_{L_x^2} \N_2(\D^{-1}G)+\N_1(\D^{-1} \ckk R)\N_1(\sn \D^{-1} G)\\
&\les \left(r^\f12\|\ckk R\|_{L_x^2}+\N_1(\D^{-1} \ckk R)\right)\N_1(G) \\
& \les \left(r^\f12\|\ckk R\|_{L_x^2}+\gc\right)\N_1(G).
\end{align*}
 Now we prove
\begin{equation}\label{bsf1}
\sum_{k>0} \|P_k \sn (\D^{-1} \ckk R \c \D^{-1} G)\|_{L_x^2}\les
c_0\N_1(G)r^{\f12} \left(\|\D^{-1} \ckk R\|_{L_x^\infty}+\|\ckk
R\|_{L_x^2}+r^2\|\sn \ckk R\|_{L_x^\infty}\right) .
\end{equation}
Indeed, we will employ GLP decompositions to write $$G=\sum_{m>
0}P_m^2 G+P_{\le 0} G+U(\infty)G.$$ For simplicity, we consider the
high frequency terms $ \sum_{m> 0}P_m^2 G$. The other two terms can
be treated similar to {\it Case 2}.

{\it Case 1: $k<m$.} By (\ref{FBB}) and  (\ref{knf3}),
\begin{align*}
\|P_k\sn (\D^{-1} \ckk R\c \D^{-1} G_m)\|_{L_x^2}&\les
2^{k-\frac{3m}{2}}r^\f12\N_1(G) \|\D^{-1} \ckk R\|_{L_x^\infty}.
\end{align*}
Thus we obtain
\begin{equation*}
\sum_{k>0}\sum_{m>k} \|P_k \sn (\D^{-1} \ckk R\c \D^{-1}
G_m)\|_{L_x^2}\les r^\f12 \N_1(G) \|\D^{-1} \ckk R\|_{L_x^\infty}.
\end{equation*}

{\it Case 2: $k>m$.} We decompose further such that
\begin{equation}\label{inbs1}
P_k \sn (\D^{-1} \ckk R\c \D^{-1} G_m)=P_k \sn
(\sum_{n>0}P_n^2+P_{\le 0})(\D^{-1} \ckk R\c \D^{-1} G_m).
\end{equation}
For simplicity we consider the high frequency terms, and the low
frequency terms can be treated similarly. We can adapt the proof for
\cite[Proposition 4.5]{Qwang} to obtain the following inequality for
$S$ tangent tensor field $F$ and $1>\a>\a_0\ge \frac{1}{2}$
\begin{align}
\|P_k \sn P_n^2 F\|_{L_x^2}&\les
\Big(2^{\min(k,n)}2^{-2|n-k|}r^{-1}+2^{\min(k,n)}
2^{-(1-\a)\max(k,n)}\K_{\a_0}r^{-\a}\nn\\&\quad\quad+2^{-|k-n|}\|\K\|_{L_x^2}^\a
\K_{\a_0}\Big)\|P_n F\|_{L_x^2}\label{knf1}.
\end{align}
Let $\I_{nm}=\|P_n(\D^{-1} \ckk R\c \D^{-1} G_m)\|_{L_x^2}$, we have
\begin{align}
\|P_k \sn P_n^2&(\D^{-1} \ckk R\c \D^{-1} G_m)\|_{L_x^2} \nn \\
&\les
\Big(2^{\min(k,n)}2^{-2|n-k|}r^{-1}+2^{\min(k,n)}
2^{-(1-\a)\max(k,n)}\K_{\a_0}r^{-\a}\nn\\&+
2^{-|k-n|}\|\K\|_{L_x^2}^\a \K_{\a_0} \Big)\I_{nm}.\label{knf12}
\end{align}

Now we estimate $\I_{nm}$. Let us first consider the case that
$n>m>0$. By Proposition \ref{P2} (iii), we have
\begin{align*}
\I_{nm}&\les r^2 2^{-2n} \|\ti P_n\sD(\D^{-1} \ckk R\c \D^{-1}
G_m)\|_{L_x^2}\\&\les r^2 2^{-2n} \Big(\|\sD\D^{-1} \ckk R\c
\D^{-1}G_m\|_{L_x^2}+\|\ti P_n(\sn \D^{-1} \ckk R\c \sn \D^{-1}
G_m)\|_{L_x^2}\\
&\quad \,+\|\sD\D^{-1} G_m\c \D^{-1}\ckk R\|_{L_x^2} \Big).
\end{align*}
 By  (\ref{knf3}) and  (\ref{knf2}), we have
 \begin{align*}
 \|\sD\D^{-1} &\ckk R\c \D^{-1}G_m\|_{L_x^2}\\
 &\les \left(\|\sn \ckk R\|_{L_x^\infty}+\|\K\|_{L_x^\infty}\|\D^{-1} \ckk R\|_{L_x^\infty}
 +r^{-2} \|\D^{-1} \ckk R\|_{L_x^\infty}\right) 2^{-\frac{3m}{2}} r^{\frac{3}{2}} \N_1(G).
 \end{align*}
 By (\ref{WB}), Propositions \ref{P4.1} and (\ref{knf11}),
 \begin{align*}
 \|\ti P_n(\sn \D^{-1} \ckk R\c \sn \D^{-1} G_m)\|_{L_x^2}
 &\les 2^{\frac{n}{2}}r^{-\f12} \|\sn \D^{-1} \ckk R\|_{L_x^2}\|\sn \D^{-1} G_m\|_{L_x^4}\\
 &\les 2^{\frac{n}{2}}r^{-\f12}\|\ckk R\|_{L_x^2} \N_1(G)
 \left(1+2^{-\frac{m}{2}}r^\f12 \|\K\|_{L_x^2}\right).
 \end{align*}
 By (\ref{knf2}), (\ref{FBB}) and (\ref{TE}), (\ref{knf3}), we obtain
\begin{align*}
\|\sD\D^{-1} &G_m\c \D^{-1}\ckk R\|_{L_x^2}\\
& \les \|\D^{-1} \ckk R\|_{L_x^\infty}
\left(\|{}^\star\D G_m\|_{L_x^2}+\|\K\|_{L_x^\infty}\|\D^{-1} G_m\|_{L_x^2}
+r^{-2} \|\D^{-1} G_m\|_{L_x^2}\right)\\
&\les \|\D^{-1} \ckk R\|_{L_x^\infty} \left(2^{\frac{m}{2}}
r^{-\frac{1}{2}}+\|\K\|_{L_x^\infty}
2^{-\frac{3m}{2}}r^{\frac{3}{2}}+r^{-\f12} 2^{-\frac{3m}{2}} \right)\N_1(G).
\end{align*}
Hence
\begin{align}
\I_{nm}&\les r^2 2^{-2n}\N_1(G)\Big(2^{\frac{m}{2}} r^{-\f12}
\|\D^{-1} \ckk R\|_{L_x^\infty}+2^{\frac{n}{2}} r^{-\f12} \|\ckk
R\|_{L_x^2} \big(1+2^{-\frac{m}{2}}r^\f12 \|\K\|_{L_x^2}\big)\nn\\
&\quad
\quad+\|\sn \ckk R\|_{L_x^\infty} 2^{-\frac{3m}{2}}r^{\frac{3}{2}}
+\|\K\|_{L_x^\infty} \|\D^{-1} \ckk R\|_{L_x^\infty}
2^{-\frac{3m}{2}} r^{\frac{3}{2}}\Big).\label{inm}
\end{align}
Combined with (\ref{knf12}), summing over $k,n,m>0$ for the cases
$k>n>m$ and $n>k>m$, we summarize the results as follows
\begin{align*}
\sum_{k, n,m>0, k>m, n>m}&\|P_k \sn P_n^2(\D^{-1} \ckk R\c \D^{-1}
G_m)\|_{L_x^2}\\
&\les c_0 \N_1(G) r^{\frac{1}{2}} \left(\|\D^{-1} \ckk
R\|_{L_x^\infty}+\|\ckk R\|_{L_x^2} +r^2\|\sn \ckk R\|_{L_x^\infty}\right)
\end{align*}
and  $c_0$ depends on
$\|r^2\K\|_{L_x^\infty}+\K_{\a_0}+\|\K\|_{L_x^2}$.

It remains to estimate $\I_{nm}$ when $k>m>n$. By (\ref{knf3}),
\begin{equation}\label{inm1}
\I_{nm}\les \|\D^{-1} \ckk R\c \D^{-1} G_m\|_{L_x^2}\les
2^{-\frac{3m}{2}}r^{\frac{3}{2}}\N_1(G)\|\D^{-1} \ckk
R\|_{L_x^\infty}.
\end{equation}
Combined with (\ref{knf1})
\begin{equation}\label{knf4}
\sum_{k, n,m>0, k>m>n}\|P_k \sn P_n^2(\D^{-1} \ckk R\c \D^{-1}
G_m)\|_{L_x^ 2}\le c_0\N_1(G) r^{\frac{1}{2}}\|\D^{-1} \ckk
R\|_{L_x^\infty}
\end{equation}

\end{proof}

Recall the following  expression holds symbolically for any $S$
tangent tensor $F$, (see \cite{KR4},\cite[page 300]{Qwang}),
\begin{equation}\label{com4}
[\sn, \sD]F=\sn(K\c F)+K\c \sn F.
\end{equation}

\begin{proof}[Proof of Lemma \ref{lem:in2}]
Assuming (\ref{t1}), we first prove (\ref{tipp}). For simplicity let
us set $\ti H=\sn \D^{-1}H$. We have by \cite[Proposition 3.20
(x)]{KRs},
\begin{equation}\label{knf10}
\| \ti H\|_{L_x^\infty}\les \sum_{k>0}2^k r^{-1}\| P_k \ti
H\|_{L_x^2}+r^\theta c_0(\|\sn\ti H\|_{L_x^2}+\|r^{-1} \ti
H\|_{L_x^2}),
\end{equation}
where $c_0$ depends on $\|\K\|_{L_x^2},$ and $\theta>0$ is very
close to $0$.

The first term on the right of (\ref{knf10}) can be bounded in view
of (\ref{t1}),
\begin{equation}\label{a95}
\|\ti H\|_{B_{2,1}^1}\les \|{}^\star \D H\|_{B_{2,1}^0}+\|H\|_{L_\omega^2}+
r^{\f12}c_0\N_1(H).
\end{equation}
We then estimate $\|\sn \ti H\|_{L_x^2}$,
by applying (\ref{Bochner1}) to $F= \D^{-1} H$. By using Proposition
\ref{P4.1} and (\ref{sob.m}) we can obtain
\begin{align}
\|\sn \ti H\|_{L_x^2}&\les \|{}^\star\D
H\|_{L_x^2}+r^{-1}\|H\|_{L_x^2}+r^{\f12} c_0\N_1(H)\label{a91},
\end{align}
with $c_0$ depending on $\|\K\|_{L_x^2}$.

By combining (\ref{a91}), (\ref{a95}) and (\ref{knf10}) and
Proposition \ref{P4.1}, (\ref{tipp}) follows.

Now consider (\ref{t1}).
Using  GLP projections, we need to  prove
\begin{align}
&\sum_{k,m>0}2^k r^{-1} \|P_k \sn P_m^2 \D^{-1} H\|_{L_x^2}\les \|{}^\star\D H\|_{B_{2,1}^0}+ c_0 r^{\f12}\N_1(H)\label{lem:in3}\\
&\sum_{k,m>0}2^k r^{-1} \|P_k \sn P_{\le 0} \D^{-1} H\|_{L_x^2}\les
c_0r^{\f12}\N_1(H)\label{lem:in4},
\end{align}
where $c_0$  depends on $\|\K\|_{L_x^2}+r\|\sn
\K\|_{L_x^2}+r\|\K\|_{L_x^\infty}$.

 The proof of (\ref{lem:in4}) is similar to  the following {\it Case 1} of the treatment for
   $$\I_{km}:= 2^k r^{-1} \|P_k \sn P_m^2 \D^{-1} H\|_{L_x^2},$$
 thus  we will give the proof of  (\ref{lem:in3}) only.

{\it Case 1: $k>m$.} By (\ref{FD}),
\begin{equation}\label{ikm1}
\I_{km}\le 2^{-k}r\left(\|P_k \sn \sD P_m^2 \D^{-1}
H\|_{L_x^2}+\|P_k [\sD, \sn]P_m^2 \D^{-1} H\|_{L_x^2}\right).
\end{equation}
Let us denote the two  terms on the right by $\I_{km}^1$ and
$\I_{km}^2$ respectively.  In view of (\ref{knf2}), (\ref{FBB}) and Lemma \ref{dual}
we have
\begin{align}
\I_{km}^1&\les 2^{-k}r\|P_k \sn P_m^2({}^\star \D H+\K \D^{-1} H+r^{-2} \D^{-1} H)\|_{L_x^2}\nn\\
&\les 2^{m-k} \|P_m{}^\star\D H\|_{L_x^2}+ 2^{-k} r
\|H\|_{L_x^2}+2^{-k} r\|P_k\sn P_m^2(\K\c \D^{-1} H)\|_{L_x^2}
\label{ini5}
\end{align}
We only need to employ (\ref{knf1}) to estimate the last term of (\ref{ini5}).
\begin{align}
 2^{-k} r\|P_k\sn P_m^2(\K\c \D^{-1} H)\|_{L_x^2}
 &\les \Big(2^{-3|m-k|}+2^{-|m-k|} 2^{-(1-\a)k}\K_{\a_0}r^{1-\a}\nn\\
&+2^{-k}2^{-|k-m|}r\|\K\|_{L_x^2}^\a \K_{\a_0} \Big)\|P_m (\K \D^{-1}
H)\|_{L_x^2} \label{ini6}.
 \end{align}
For the last two terms, in view of $k>m$, (\ref{sob.m2}) and Lemma \ref{A.8}
\begin{align}
\sum_{k>m} \Big(2^{-|m-k|}
2^{-(1-\a)k}&\K_{\a_0}r^{1-\a}+2^{-k}2^{-|k-m|}r\|\K\|_{L_x^2}^\a
\K_{\a_0} \Big)\|P_m (\K \D^{-1} H)\|_{L_x^2}\nn \\
&\les \K_{\a_0} \left(r^{1-\a}+r\|\K\|_{L_x^2}^\a\right)\|\K\|_{L_x^2}
r^{\f12}\N_1(H)\label{knf6}.
 \end{align}
Let us decompose $\K=\sum_{n} P_n^2 \K+\bar\K$ and consider the high
frequency term for the purpose of simplicity. With the help of
Proposition \ref{P4.1}, the proof contained in
\cite[pages 299--300]{Qwang} implies for $m,n> 0$
\begin{equation}\label{kdin1}
\|P_m (\K_n \D^{-1}H)\|_{L_x^2}\les
2^{-\frac{3}{4}|m-n|}\|P_n\K\|_{L_x^2} \left(\|\D^{-1}
H\|_{L_x^\infty}+\|H\|_{L_x^2}\right).
\end{equation}
Therefore the first term on the right of (\ref{ini6}) can be
estimated as follows,
\begin{align*}
\sum_{k>m} \sum_{n>0} & 2^{-3|m-k|}\|P_m (\K_n\D^{-1}
H)\|_{L_x^2}\\
&\les\sum_{k>m} \sum_{n>0} 2^{-3|m-k|
-\frac{3}{4}|m-n|}\|P_n\K\|_{L_x^2} \left(\|\D^{-1}
H\|_{L_x^\infty}+\|H\|_{L_x^2}\right) \\
&\les \|\K\|_{B_{2,1}^0}
r^{\f12}\N_1(H)
\end{align*}
where we employed (\ref{sob.m2}), Lemma \ref{A.8} and (\ref{sob.m})
to obtain the last inequality. It is easy to check by (\ref{FB})
that $\|\K\|_{B_{2,1}^0}\les \|\K\|_{L_x^2}+r \|\sn \K\|_{L_x^2}$.
Consequently,
\begin{equation*}
\sum_{k>m}\I_{km}^1\les\|{}^\star\D H\|_{B_{2,1}^0}+c_0
r^{\f12}\N_1(H).
\end{equation*}

Now we consider $\I_{km}^2$ with the help of (\ref{com4}) and
(\ref{sob.m2}), also using (\ref{FBB}) and Lemma \ref{dual}
\begin{align}
\I_{km}^2&\les 2^{-k}r\left(\|P_k\sn(K\c P_m^2 \D^{-1} H)\|_{L_x^2}+\|P_k(K\c \sn P_m^2\D^{-1} H)\|_{L_x^2}\right)\nn\\
&\les2^{-k} r\Big(\|\sn \K\|_{L_x^2}\|P_m^2\D^{-1} H\|_{L_x^\infty}+\|\K\|_{L_x^\infty}\|\sn P_m^2\D^{-1} H\|_{L_x^2} \nn\\
&\quad\, +r^{-2}\|P_k(\sn P_m^2 \D^{-1} H)\|_{L_x^2}\Big)\nn\\
&\les 2^{-k}\left( r\|\sn \K\|_{L_x^2}r^\f12
\N_2(\D^{-1}H)+r\|\K\|_{L_x^\infty}\|H\|_{L_x^2}+r^{-1}
\|H\|_{L_x^2}\right).\nn
\end{align}
Also using Lemma \ref{A.8} and (\ref{sob.m})
\begin{align*}
\sum_{k>m>0}&\I_{km}^2\les \left( r\|\sn \K\|_{L_x^2}+r\|\K\|_{L_x^\infty}\right)
r^{\frac{1}{2}}\N_1(H)+r^{-1} \|H\|_{L_x^2}.
\end{align*}

{\it Case 2: $k<m$.} Consider $\I_{km}$ in this case by (\ref{FD})
and (\ref{FBD}).
\begin{align}
\I_{km}&\le 2^{k-2m}r\|P_k\sn \sD P_m^2 \D^{-1} H\|_{L_x^2}\nn\\
&\le 2^{k-2m}r \left(\|P_k \sD \sn P_m^2 \D^{-1} H\|_{L_x^2}+\|P_k[\sn, \sD] P_m^2 \D^{-1}H\|_{L_x^2}\right)\nn\\
&\le 2^{3k-2m}r^{-1} \|P_k\sn P_m^2 \D^{-1}
H\|_{L_x^2}+2^{k-2m}r\|P_k[\sn, \sD] P_m^2
\D^{-1}H\|_{L_x^2}\label{knf7}.
\end{align}
Let us denote by $\I_{km}^1$ the first term in the line of
(\ref{knf7}) and by $\I_{km}^2$ the second term. Consider
$\I_{km}^1$ first. By (\ref{FBB}), (\ref{FD}) and (\ref{knf2})
\begin{align}
\I_{km}^1&\les 2^{4k-2m}r^{-2} \|P_m^2\D^{-1} H\|_{L_x^2}\les 2^{4k-4m} \|P_m^2 \sD \D^{-1}H\|_{L_x^2}\nn\\
&\les 2^{-4|k-m|}\left(\|P_m^2{}^\star\D
H\|_{L_x^2}+\|P_m^2(\K\D^{-1} H)\|_{L_x^2}+r^{-2} \|P_m^2\D^{-1}
H\|_{L_x^2}\right)\nn
\end{align}
 By (\ref{FB}), Proposition \ref{P4.1}, (\ref{sob.m2}) and Lemma \ref{A.8},
\begin{align}
\|P_m^2(\K\D^{-1}H)\|_{L_x^2}
&\les 2^{-m}r \left(\|\sn \K\|_{L_x^2}\|\D^{-1}H\|_{L_x^\infty}+\|\K\|_{L_x^\infty}\|H\|_{L_x^2}\right)\nn\\
&\les 2^{-m} r^{\frac{3}{2}}\N_1(H)\left(\|\sn
\K\|_{L_x^2}+\|\K\|_{L_x^\infty}\right)\label{knf8}.
\end{align}
Also using  Lemma \ref{dual},
\begin{align}
\I_{km}^1&\les 2^{-4|k-m|} \left(\|P_m^2{}^\star\D
H\|_{L_x^2}+r^{-1}2^{-m} \| H\|_{L_x^2}+2^{-m}
r^{\frac{3}{2}}\N_1(H)(\|\sn
\K\|_{L_x^2}+\|\K\|_{L_x^\infty})\right)\nn.
\end{align}
Hence, we obtain
\begin{align}
\sum_{k,m>0, k<m} \I_{km}^1&\les \|{}^\star \D
H\|_{B_{2,1}^0}+r^{-1} \|H\|_{L_x^2}+ r^{\f12}\N_1(H) \left(r\|\sn
\K\|_{L_x^2}+r\|\K\|_{L_x^\infty}\right).
\end{align}
Now consider $\I_{km}^2$ in view of (\ref{com4}).
\begin{align*}
\I_{km}^2&\les 2^{k-2m}r\Big(\|P_k\sn(\K P_m^2 \D^{-1}
H)\|_{L_x^2}+\|P_k(\K\sn P_m^2 \D^{-1} H)\|_{L_x^2}\\
&\quad\, +r^{-2}\|P_k\sn
P_m^2\D^{-1}H\|_{L_x^2}\Big).
\end{align*}
By (\ref{FBB}) and  Lemma \ref{dual}
\begin{align*}
\I_{km}^2&\les
2^{k-2m}r\left(\|\K\|_{L_x^\infty}\|H\|_{L_x^2}+2^{k-m}r^{-2}\|H\|_{L_x^2}\right).
\end{align*}
Also using (\ref{sob.m}),
\begin{equation*}
\sum_{k,m>0, m>k}\I_{km}^2\les r^{-1}\|H\|_{L_x^2}+
r^{\frac{1}{2}}\N_1(H)\c r\|\K\|_{L_x^\infty}.
\end{equation*}
Thus
\begin{equation*}
\sum_{k,m>0, m>k}\I_{km}\les r^{-1}\|H\|_{L_x^2}+ c_0
r^{\frac{1}{2}}\N_1(H)
\end{equation*}
where $c_0$ depends on $r(\|\K\|_{L_x^\infty}+\|\sn \K\|_{L_x^2})$.
\end{proof}




\begin{thebibliography}{9999}
%
%
\bibitem{KC} Christodoulou, D. and Klainerman, S., {\it  The Global
Nonlinear Stability of the Minkowski Space}, Princeton Mathematical
Series 41, 1993.

\bibitem{KR1} Klainerman, S. and  Rodnianski, I.,
{\it Causal geometry of Einstein-vacuum spacetimes with finite
curvature flux}, Invent. Math.,  159(2005), no. 3, 437--529.

\bibitem{KR4}Klainerman, S. and Rodnianski, I.,  {\it A geometric
Littlewood-Paley theory},  Geom. Funct. Anal., 16 (2006), 126--163.

\bibitem{KRs} Klainerman, S. and Rodnianski, I.,
{\it Sharp Trace theorems for null hypersurfaces on Einstein metrics
with finite curvature flux},   Geom. Funct. Anal., 16  (2006),  no. 1, 164--229.


\bibitem{KRradiusp} Klainerman, S. and Rodnianski, I., {\it On the
radius of injectivity of null hypersurfaces}, 2006,
arXiv:math/0603010v1

\bibitem{KRradius} Klainerman, S. and Rodnianski, I., {\it On the
radius of injectivity of null hypersurfaces},  J. Amer. Math. Soc.,
21 (2008), no. 3, 775--795.

\bibitem{KR2} Klainerman, S. and Rodnianski, I., {\it On the
breakdown criterion in general relativity}, J. Amer. Math. Soc., 23
(2010), no. 2, 345--382.

\bibitem{Stein1} Stein, E. M., {\it Topics in harmonic analysis related
to the Littlewood-Paley theory}, Annals of Mathematics Studies, No.
63,  Princeton University Press, 1970.

\bibitem{Stein2} Stein, E. M., {\it Harmonic analysis: real-variable
methods, orthogonality, and oscillatory integrals},  With the
assistance of Timothy S. Murphy, Princeton Mathematical Series 43,
Monographs in Harmonic Analysis III, Princeton University Press,
Princeton, NJ, 1993.

\bibitem{WangQ} Wang, Q.,
{\it Causal Geometry of Einstein-Vacuum spacetimes},
 PhD thesis, Princeton University 2006.

 \bibitem{Qwang} Wang, Q.,
 {\it On the geometry of null cones in Einstein Vacuum Spacetimes},
 Ann. Inst. H. Poincar\`e Anal. Non Lin\'eaire, 26(2009), 285-328

 \bibitem{Qwang1} Wang, Q.,
  {\it Improved breakdown criterion for solution of Einstein vacuum
equation in CMC gauge}, 2010, arXiv:1004.2938.



\end{thebibliography}



\end{document}